\documentclass[10pt]{amsart}
\usepackage[T1]{fontenc}
\usepackage[utf8]{inputenc}
\usepackage{lmodern}
\usepackage{amssymb,mathtools,mathrsfs}
\usepackage{microtype}
\usepackage[margin=1in]{geometry}
\usepackage[dvipsnames]{xcolor}
\usepackage{booktabs,longtable,array,tabularx}
\usepackage{url}
\usepackage{enumitem,graphicx,xurl}
\graphicspath{{figures/}}

\usepackage{amssymb,mathtools}
\usepackage{microtype}
\usepackage[margin=1in]{geometry}
\usepackage{amssymb,mathtools}
\usepackage{microtype}
\usepackage[margin=1in]{geometry}
\usepackage{xcolor}
\usepackage{color}
\usepackage{booktabs,longtable,array}
\usepackage{tikz}
\usetikzlibrary{arrows.meta,calc,patterns}
\usepackage[colorlinks=true,linkcolor=blue,citecolor=blue,urlcolor=blue]{hyperref}

\newtheorem{theorem}{Theorem}[section]
\newtheorem{proposition}[theorem]{Proposition}
\newtheorem{corollary}[theorem]{Corollary}
\newtheorem{lemma}[theorem]{Lemma}

\theoremstyle{definition}

\newtheorem{example}[theorem]{Example}
\theoremstyle{remark}
\newtheorem{remark}{Remark}

\newcommand{\R}{\mathbb R}
\newcommand{\Z}{\mathbb Z}
\newcommand{\N}{\mathbb N}
\newcommand{\C}{\mathbb C}
\newcommand{\Q}{\mathbb Q}

\newcommand{\Li}{\operatorname{Li}}
\newcommand{\lcm}{\operatorname{lcm}}
\newcommand{\sinc}{\operatorname{sinc}}
\newcommand{\csch}{\operatorname{csch}}
\DeclareMathOperator{\kei}{kei}
\newcommand{\ee}[1]{\mathrm e^{2\pi i(#1)}}
\newcommand{\en}[1]{\mathrm e^{-2\pi i(#1)}}

\newcommand{\uu}{{\boldsymbol{u}}}
\newcommand{\vv}{{\boldsymbol{v}}}
\newcommand{\xx}{{\boldsymbol{x}}}
\newcommand{\yy}{{\boldsymbol{y}}}
\newcommand{\mm}{{\boldsymbol{m}}}
\newcommand{\rr}{{\boldsymbol{r}}}
\newcommand{\aaa}{{\boldsymbol{a}}}
\newcommand{\jj}{{\boldsymbol{j}}}
\newcommand{\dd}{{\boldsymbol{d}}}
\newcommand{\zz}{{\boldsymbol{z}}}
\newcommand{\pp}{{\boldsymbol{p}}}
\newcommand{\ww}{{\boldsymbol{w}}}
\newcommand{\xxi}{{\boldsymbol{\xi}}}

\hypersetup{pdftitle={Arithmetic Poisson summation formulas},pdfauthor={Pablo Andr\'es Panzone and Lucas Oliveira},pdfcreator={},pdfproducer={}}
\ifdefined\pdfinfoomitdate\pdfinfoomitdate=1\fi
\ifdefined\pdfsuppressptexinfo\pdfsuppressptexinfo=15\fi
\title{Arithmetic Poisson summation formulas}
\author{Pablo Andr\'es Panzone}
\address{Departamento e Instituto de Matem\'atica, UNS, Av. Alem 1253, (8000) Bah\'ia Blanca, Argentina}
\email{pablopanzone@hotmail.com}
\author{Lucas Oliveira}
\address{Departamento de Matem\'atica, UFRGS, Porto Alegre, Brazil, 91509-900}
\email{lucas.oliveira@ufrgs.br}
\begin{document}
\begin{abstract}
We formulate Poisson summation with periodic weights in forms adapted to arithmetic sums (arithmetic Poisson summation formulas, APSF), in scalar, polynomial-congruence and vector-divisibility versions, extending Berndt's arithmetic Poisson formula. We use APSF as a common framework to rederive Ramanujan's formula for $\zeta(2m+1)$ and its character analogues due to Katayama and Berndt-Straub, the functional equations of $\zeta$, of Dirichlet's beta function and of Estermann's function, and the exponents of Tolev and Chen-Wang for power-free values of $x^2+y^2+1$ and $x^2+y^2+z^2+k$. It also gives reciprocity laws with Kelvin-function kernels for Lambert series twisted by a Dirichlet character over imaginary quadratic fields of arbitrary class number; asymptotic formulas with error terms for $r$-free values of $x_1^2+\cdots+x_n^2+k$, with Gaussian-weighted versions and the meromorphic continuation of the associated Dirichlet series; an argument of Terras, relating two-dimensional Bessel series for all odd zeta values and even Dirichlet beta values.
\end{abstract}
\maketitle

\section{Introduction}

Our starting point is Berndt's arithmetic Poisson formula \cite{Berndt82}, which separates divisor coefficients inside a Poisson summation identity. For $S\subseteq\N$ a nonempty subset and for $a:\Z\to\C$ an arithmetic function, consider the companion function
\begin{equation}\label{eq:berndteq1}
 b(n)=\sum_{d\mid n}\mu(d)a(n/d),\qquad n\ge1,
\end{equation}
with $\mu$ the M\"obius function. Under convergence assumptions permitting the sums, integrals, and regroupings involved in the formula below, we have
\begin{equation}\label{eq:berndteq2}
 \sum_{n\ne0}f(n)\left(\sum_{\substack{d\mid |n|\\d\in S}}b(d)\right)
 =\sum_{\substack{d\ge1\\d\in S}}\frac{b(d)}d
 \left(\sum_{r\in\Z}\widehat f(r/d)-df(0)\right).
\end{equation}
Combining his arithmetic formula with Vorono\u\i's Bessel expansion \cite{Voronoi05}, Berndt obtains the circle identity \cite[p. 299]{Berndt82}
\begin{equation}\label{eq:berndteq3}
 \sum_{1\le n\le x}'r_2(n)
 =\pi x+4\sum_{d=1}^\infty\frac{\chi_{-4}(d)}d
 \left\{\frac4\pi\sum_{r=1}^\infty\sin^2\left(\frac{\pi xd}{4r}\right)-\frac{\pi xd}{2}\right\}.
\end{equation}
Here $r_2(n)$ counts representations of $n$ as a sum of two integer squares, and $\chi_{-4}$ is $0$ on even integers and $(-1)^{(d-1)/2}$ on odd positive integers. At a positive integer endpoint the prime means that only $r_2(x)/2$ is counted. The endpoint convention is that of Berndt's circle formula \cite{Berndt82} and the braces prescribe the indicated order of summation. 

Berndt's proof detects divisibility by the properties of the finite sum
$\sum_{j=0}^{d-1}\ee{jx/d}$: he uses the oscillatory and cancellative behavior of this sum to rewrite the left-hand side of \eqref{eq:berndteq2} in a form suited to the classical Poisson summation formula. To cover more cases, the idea is to replace the period $d$ of the above exponential sum by a positive integer $h(d)$, giving us the expression
$$
{S}(d,x):=S_h(d,x)=\sum_{j=0}^{h(d)-1} e^{2\pi i j x/h(d)}.
$$
The choice of $h$ determines the arithmetic restriction. We formulate scalar, vector-divisibility, and polynomial-congruence versions below. Their applications range from Mersenne--Fermat detectors and Ramanujan reciprocity to Gaussian counts and Bessel expansions of Epstein zeta functions. We call these identities \textbf{arithmetic Poisson summation formulas} (APSF). \\

\noindent\textbf{Notations and conventions}
In what follows, by \emph{Poisson admissible} we mean conditions on a function $f:\R^n\to\C$ permitting the usual Poisson summation formula in the stated order. When a formula uses translated residue classes, admissibility on each of those classes is required. We use number theorist normalization of the Fourier transform
$$
 \widehat f(\xxi)=\int_{\R^n}f(\xx)\ee{\xx\cdot\xxi}\,d\xx.
$$
Vectors are written in bold, for example $\vv=(v_1,\ldots,v_n)$; their coordinates and scalar arithmetic indices are not. Divisors of a nonzero integer are positive divisors of its absolute value, and $(b*c)(n)=\sum_{d\mid |n|}b(d)c(n/d)$. We put $\sinc x=\sin(\pi x)/(\pi x)$, with value $1$ at zero; $\tau(n)$ denotes the number of positive divisors, whereas $\tau(\chi)$ denotes a character Gauss sum that will be introduced in Example \ref{sec:basic-characters}. Every outer arithmetic series below is taken after its inner Fourier sum has been completed, unless joint absolute convergence is stated.\\

We can now state the three main theorems of this work:

\begin{theorem}\label{thm:T1}
Let $b:\N\to\C$, $h:\N\to\N$ with $h(d)\to\infty$ as $d\to\infty$, and $\omega:\N\to\R\setminus\{0\}$ be arithmetic functions. Let $g,c:\R\to\C$ be continuous with $c(0)=0$, and define
\begin{equation}\label{eq:T1eq1}
 F_d(y)=g(h(d)y)c\left(\frac{h(d)y}{\omega(d)}\right).
\end{equation}
Suppose that for each $d$ the function $F_d$ is Poisson admissible and
\begin{equation}\label{eq:T1eq2}
 \sum_{d=1}^\infty|b(d)|\left(\sum_{r\in\Z}|F_d(r)|\right)<\infty.
\end{equation} If, for $n\ne0$, we define
\begin{equation}\label{eq:T1eq3}
 B(n)=\sum_{h(d)\mid n}b(d)c\left(\frac n{\omega(d)}\right),
\end{equation}
then
\begin{equation}\label{eq:T1eq4}
 \sum_{n\ne0}g(n)B(n)=\sum_{d=1}^\infty b(d)\left(\sum_{r\in\Z}\widehat F_d(r)\right),
\end{equation}
where the summation is taken in the above order.
\end{theorem}

\begin{theorem}\label{thm:vector-dilation}
Let $P\in\Z[x_1,\ldots,x_n]$ satisfy $P(0)=0$, let $b:\N\to\C$, $h:\N\to\N$, and $\omega:\N\to\R\setminus\{0\}$, and assume $h(d)\to\infty$. Let $g:\R^n\to\C$ and $c:\R\to\C$ be continuous, with $c(0)=0$, and define
$$
 F_d(\xx)=g(h(d)\xx)c\left(\frac{P(h(d)\xx)}{\omega(d)}\right).
$$
Assume Poisson admissibility of every $F_d$ and absolute convergence of
\begin{equation}\label{eq:vec-hyp}
 \sum_{d\ge1}|b(d)|\sum_{\vv\in\Z^n}|F_d(\vv)|<\infty.
\end{equation}
For $\vv\ne0$, define
$$
 B_{\rm vec}(\vv)=\sum_{h(d)\mid\gcd(\vv)}b(d)c\left(\frac{P(\vv)}{\omega(d)}\right).
$$
Then
\begin{equation}\label{eq:vector-dilation}
 \sum_{\vv\in\Z^n\setminus\{0\}}g(\vv)B_{\rm vec}(\vv)
 =\sum_{d\ge1}b(d)\left(\sum_{\mm\in\Z^n}\widehat F_d(\mm)\right).
\end{equation}
\end{theorem}
\newpage 

\begin{theorem}\label{thm:polynomial-T1}
Let $P\in\Z[x_1,\ldots,x_n]$, let $b:\N\to\C$, $h:\N\to\N$, and $\omega:\N\to\R\setminus\{0\}$, and assume $h(d)\to\infty$. Let $g:\R^n\to\C$ and $c:\R\to\C$ be continuous, with $c(0)=0$, and put
$$
 F_d(\xx)=g(\xx)c\left(\frac{P(\xx)}{\omega(d)}\right).
$$
Assume each $F_d$ is Poisson admissible on every residue class modulo $h(d)$ and
\begin{equation}\label{eq:poly-hyp}
 \sum_{d\ge1}|b(d)|\sum_{\substack{\vv\in\Z^n\\h(d)\mid P(\vv)}}|F_d(\vv)|<\infty.
\end{equation}
Let $B$ be the function of Theorem \ref{thm:T1}, extended by $B(0)=0$; explicitly, for $m\ne0$,
$$
 B(m)=\sum_{h(d)\mid m}b(d)c\left(\frac m{\omega(d)}\right),\qquad B(0)=0.
$$
Then
\begin{equation}\label{eq:poly-T1}
 \sum_{\vv\in\Z^n}g(\vv)B(P(\vv))
 =\sum_{d\ge1}\frac{b(d)}{h(d)^n}
 \left(\sum_{\mm\in\Z^n}\mathcal G_{P,h(d)}(\mm)\widehat F_d\left(\frac \mm{h(d)}\right)\right),
\end{equation}
where $\mathcal G_{P,q}(\mm)=\sum_{\uu\in C_q(P)}\en{\mm\cdot \uu/q}$ and $C_q(P)=\{\uu\in(\Z/q\Z)^n:P(\uu)\equiv0\pmod q\}$. The parentheses indicate the order of summation; they may be removed under joint absolute convergence in $d,\mm$.
\end{theorem}

\begin{remark}
Let $h(d)=\omega(d)=d$. Then
\begin{equation}\label{eq:berndt-cor}
 \sum_{n\ne0}g(n)(b*c)(n)
 =\sum_{d\ge1}b(d)\left(\sum_{r\in\Z}\int_\R\ee{ry}g(dy)c(y)\,dy\right).
\end{equation}
This is the interpolated form of the scalar identity. Berndt's formula itself is recovered by the origin-corrected case $c\equiv1$ in Remark \ref{rem:zero-term}; a restriction $d\in S$ is included by replacing $b(d)$ with $b(d)\mathbf1_S(d)$.
\end{remark}
\begin{remark}
Patkowski \cite[Theorem 1.1]{Patkowski21} obtains a related formula using the cosine Fourier transform and Mellin inversion. In his notation, with $C(f)(x)=\int_0^\infty f(t)\cos(2\pi xt)\,dt$ and $M(f)$ the Mellin transform, one such formula reads
\begin{equation}\label{eq:patkowski}
 \sum_{n\ge1}a(n)C(f)(n)
 =\frac12\sum_{m\ge1}\frac{b(m)}m
 \left\{M^{-1}(\zeta(s)M(f)(s))(1/m)+\frac{f(0)}2\right\}.
\end{equation}
Here $M(f)(s)=\int_0^\infty f(x)x^{s-1}\,dx$, and the inverse Mellin integral in the displayed formula is taken on a vertical line $0<\Re s<1$, under the hypotheses of the cited formula.
\end{remark}

\begin{remark}
It is important to notice that the theorems above all follow from the periodic-weight identity of Theorem \ref{thm:gen_APSF}. Although the formula in Theorem \ref{thm:gen_APSF} is more concise and easier to state, it masks the arithmetic nature of several applications: readers interested in more general versions of these results will not find it too difficult to formulate the theorems in more general settings, including sums over lattices and number fields, and, under the hypotheses given in Section \ref{sec:final}, weights with absolutely convergent exponential expansions. Our choice of presentation is stylistic: all those versions of the Poisson summation formula are classical and we don't need to state them in the most general form to obtain our main results, which are convenient formulations of the Poisson summation formulas for applications to arithmetic problems. 
\end{remark}

Now that we have this collection of arithmetic Poisson summation formulas (which we will call APSF for short), we need to understand how these results are useful: we can exploit them to obtain or recover results already known by other techniques; we can isolate some arithmetic data in a counting problem and use this to break sums into simpler ones; we can use them to obtain new formulas for quantities already known; and we can use them simply as a tool inside a larger argument to systematize the process, which can be useful in several situations. Before presenting some interesting applications in the next sections, let us spend a few lines stating immediate corollaries that we hope will show the usefulness of such theorems:

\begin{corollary}[Primitive-vector transform]\label{cor:primitive-vector}
Take $h(d)=d$, $b(d)=\mu(d)$, $\omega(d)=1$, $P(\vv)=|\vv|^2$ and $c(x)=1-\sinc x$ in Theorem \ref{thm:vector-dilation}. Then $c(P(\vv))=1$ at every nonzero integer vector. Since
$$
 \sum_{d\mid\gcd(\vv)}\mu(d)=\mathbf1_{\gcd(\vv)=1},
$$
for Schwartz $g$ we obtain the explicit formula
\begin{equation}\label{eq:primitive-explicit}
 \sum_{\substack{\vv\in\Z^n\\\gcd(\vv)=1}}g(\vv)
 =\sum_{d\ge1}\mu(d)\left(\sum_{\mm\in\Z^n}\widehat F_d(\mm)\right),
 \qquad F_d(\xx)=g(d\xx)\{1-\sinc(d^2|\xx|^2)\}.
\end{equation}
The sinc factor removes the origin and equals $1$ at every other lattice sample. Each $F_d$ is Schwartz; rapid decay bounds the outer arithmetic absolute sum by $\sum_{d,\vv\ne0}|g(d\vv)|<\infty$. More generally the formula holds under the hypotheses of Theorem \ref{thm:vector-dilation}.
\end{corollary}

\begin{corollary}[Eisenstein $q$-expansion]\label{cor:base-eisenstein}\label{prop:t1-eisenstein}
Let $Y>0$ and let $k\ge2$ be a fixed integer. In Theorem \ref{thm:T1}, take $h(d)=\omega(d)=d$ and
$$
 b(d)=1,\qquad c(x)=|x|^{2k-1},\qquad g(x)=\exp(-2\pi Y|x|).
$$
Then
\begin{align}\label{eq:t1-g}
 \sum_{n\ge1}\sigma_{2k-1}(n)\exp(-2\pi Yn)
 &=\frac{\Gamma(2k)}{(2\pi)^{2k}}\sum_{d\ge1}\sum_{r\in\Z}\Re(Yd-ir)^{-2k}\notag\\
 &=\frac{(-1)^k\Gamma(2k)}{2(2\pi)^{2k}}\{G_{2k}(iY)-2\zeta(2k)\},
\end{align}
where $G_{2k}(z)=\sum_{(c,d)\in\Z^2\setminus\{0\}}(cz+d)^{-2k}$ is the Eisenstein series of weight $2k$ for $\mathrm{PSL}(2,\Z)$. Holomorphic continuation gives its $q$-expansion from this identity.
\end{corollary}
\begin{corollary}[Polynomial values of a Dirichlet convolution]\label{cor:poly-convolution}
Set $h(d)=\omega(d)=d$. With positive divisors of $|P(\vv)|$ and the signed quotient for negative values,
\begin{equation}\label{eq:poly-convolution}
 \sum_{\vv\in\Z^n}g(\vv)(b*c)(P(\vv))
 =\sum_{d\ge1}\frac{b(d)}{d^n}
 \left(\sum_{\mm\in\Z^n}\mathcal G_{P,d}(\mm)\widehat{[g(\cdot)c(P(\cdot)/d)]}(\mm/d)\right),
\end{equation}
provided Theorem \ref{thm:polynomial-T1} applies, with the value at $P(\vv)=0$ defined as zero.
\end{corollary}
These corollaries illustrate three ways of reorganizing arithmetic sums through APSF.

\subsection{Variants of Theorem \ref{thm:T1}}
Besides the summation formulas above, there are variants useful to isolate for applications. They are stated as individual theorems, and Section 2 derives each one from Theorem \ref{thm:gen_APSF}.
\begin{theorem}\label{T2}
Let $b_0$ be arithmetic and let $g_0,c_1,\ldots,c_\ell:\R\to\C$ be continuous functions of the same parity, with $c_j(0)=0$ for every $j$. Define
$$
 F_d(y_1,\ldots,y_\ell)=g_0(dy_1\cdots y_\ell)c_1(y_1)\cdots c_\ell(y_\ell).
$$
Assume each $F_d$ is Poisson admissible on $\R^\ell$ and
\begin{equation}\label{eq:T2-grouped-hyp}
 \sum_{d\ge1}|b_0(d)|\sum_{\mm\in\Z^\ell}|F_d(\mm)|<\infty.
\end{equation}
Then
\begin{equation}\label{eq:T2}
 2^\ell\sum_{n\ge1}g_0(n)(b_0*c_1*\cdots*c_\ell)(n)
 =\sum_{d\ge1}b_0(d)\left(\sum_{r_1,\ldots,r_\ell\in\Z}I(d;r_1,\ldots,r_\ell)\right),
\end{equation}
where
$$
 I(d;r_1,\ldots,r_\ell)=\int_{\R^\ell}\ee{r_1y_1+\cdots+r_\ell y_\ell}
 g_0(dy_1\cdots y_\ell)\prod_{j=1}^\ell c_j(y_j)\,d\yy.
$$
\end{theorem}
\begin{theorem}\label{Tproduct}
Let $b_1,\ldots,b_\ell$ be arithmetic and let $g,c_1,\ldots,c_\ell:\R\to\C$ be continuous. Assume at least one $c_j(0)=0$. For $\dd=(d_1,\ldots,d_\ell)$ put
$$
 L(\dd)=\lcm(d_1,\ldots,d_\ell),\qquad
 H_{\dd}(y)=g(L(\dd)y)\prod_{j=1}^\ell c_j\left(\frac{L(\dd)y}{d_j}\right).
$$
Assume each $H_{\dd}$ is Poisson admissible and
\begin{equation}\label{eq:product-grouped-hyp}
 \sum_{d_1,\ldots,d_\ell\ge1}\left(\prod_{j=1}^\ell|b_j(d_j)|\right)\sum_{m\in\Z}|H_{\dd}(m)|<\infty.
\end{equation}
Then
\begin{equation}\label{eq:Tproduct}
 \sum_{n\ne0}g(n)\prod_{j=1}^\ell(b_j*c_j)(n)
 =\sum_{d_1,\ldots,d_\ell\ge1}\prod_{j=1}^\ell b_j(d_j)
 \left(\sum_{r\in\Z}\int_\R\ee{ry}H_{\dd}(y)\,dy\right).
\end{equation}
\end{theorem}
\begin{remark}\label{base:compatibility-remark}
The simultaneous divisibility conditions in Theorem \ref{Tproduct} combine into a single least-common-multiple condition, as used in its proof in Section \ref{sec-conv}.
\end{remark}
\begin{theorem}\label{Tshift}\label{T5}
Let $b,b_0$ be arithmetic functions of at most polynomial growth. Let $c,c_0:\R\to\C$ be continuous and of at most polynomial growth, with $c(0)=0$ and $c_0(x)=0$ for $x\le0$. Let $g$ be continuous and exponentially decreasing at both ends of $\R$, and let $\kappa\ge0$ be an integer. Assume that $g(x)c(x/d)c_0((x+\kappa)/\delta)$ is Poisson admissible on each residue class modulo $d\delta$ for every $d,\delta\ge1$. Then
\begin{align}\label{eq:Tshift}
 &\sum_{n=-\kappa}^\infty g(n)(b*c)(n)
 \left\{\sum_{\delta\mid n+\kappa}b_0(\delta)c_0\left(\frac{n+\kappa}{\delta}\right)\right\}\notag\\
 &=\sum_{d,\delta\ge1}b(d)b_0(\delta)
 \left\{\sum_{k=0}^{\delta-1}\ee{\kappa k/\delta}
 \left(\sum_{r\in\Z}J(d,\delta,k,r)\right)\right\},
\end{align}
where $(b*c)(0)=0$, the bracket at $n=-\kappa$ is zero, and
$$
 J(d,\delta,k,r)=\int_\R\ee{(r\delta+kd)y}g(d\delta y)c(\delta y)c_0(dy+\kappa/\delta)\,dy.
$$
Each inner frequency sum is completed first. The resulting outer blocks are absolutely summable over $(d,\delta)$, so they may be enumerated in any order.
\end{theorem}

\subsection{Organization of the article}
Section 2 proves the introductory theorems from general APSF and explains the origin correction. Section 3 evaluates elementary kernels, selects perfect powers, treats character and divisor coefficients, and gives resolvent and Mersenne--Fermat examples. Section 4 develops Ramanujan and character reciprocity, then explains the passage from Gaussian integers to quadratic forms and field sums. Section 5 derives the zeta and Estermann functional equations, studies squarefree values and their comparison with Tolev, extends the same argument to higher power-free values, and gives the Terras--Epstein expansions. Section 6 indicates broader lattice and Fourier summation formulas, and Section 7 records the use of AI tools.

\section{Proof of the main theorems}\label{sec:main_results}
\subsection{The general arithmetic Poisson summation formula}
\begin{theorem}[General APSF]\label{thm:gen_APSF}
Let $H:\N\to\N$ and let $\{W_d\}$ be a family of complex $H(d)$-periodic functions on $\Z^n$, let $\{f_d\}$ be a family of complex-valued functions on $\R^n$ admissible for Poisson summation on every residue class modulo $H(d)$, and let $b:\N\to\C$ satisfy
\begin{equation}\label{eq:gen_APSF_hyp}
 \sum_{d=1}^\infty|b(d)|\sum_{\mm\in\Z^n}|W_d(\mm)f_d(\mm)|<\infty.
\end{equation}
Then
\begin{equation}\label{eq:gen_APSF}
 \sum_{\vv\in\Z^n}\sum_{d=1}^\infty b(d)W_d(\vv)f_d(\vv)
 =\sum_{d=1}^\infty\frac{b(d)}{H(d)^n}
 \left(\sum_{\mm\in\Z^n}\mathcal F_{H(d)}(W_d)(\mm)\widehat f_d(\mm/H(d))\right),
\end{equation}
where the discrete Fourier transform is
\begin{equation}\label{eq:FiniteFourierTransf}
 \mathcal F_q(W)(\mm)=\sum_{\uu\in(\Z/q\Z)^n}W(\uu)\en{\mm\cdot\uu/q}.
\end{equation}
Each inner frequency sum is completed before the outer sum is taken.
\end{theorem}
\begin{proof}
\textbf{1.} Assume for a moment that the following Poisson summation formula holds for $W:(\Z/q\Z)^n\to\C$, extended periodically to $\Z^n$, and $f:\R^n\to\C$ admissible on the residue classes:
\begin{equation}\label{eq:matomaki_merikoski}
 \sum_{\vv\in\Z^n}W(\vv)f(\vv)
 =\frac1{q^n}\sum_{\mm\in\Z^n}\widehat f(\mm/q)\mathcal F_q(W)(\mm).
\end{equation}
By the convergence hypothesis \eqref{eq:gen_APSF_hyp}, we can apply Fubini's theorem to the left-hand side of \eqref{eq:gen_APSF}, and the resulting series is absolutely convergent. Since the functions $f_d$ are admissible on the required residue classes for every $d$, and each $W_d$ assumes only finitely many values, we can apply \eqref{eq:matomaki_merikoski}:
\begin{align*}
 \sum_{\vv\in\Z^n}\sum_{d=1}^\infty b(d)W_d(\vv)f_d(\vv)
 &=\sum_{d=1}^\infty b(d)\sum_{\vv\in\Z^n}W_d(\vv)f_d(\vv)\\
 &=\sum_{d=1}^\infty\frac{b(d)}{H(d)^n}
 \left(\sum_{\mm\in\Z^n}\mathcal F_{H(d)}(W_d)(\mm)\widehat f_d(\mm/H(d))\right).
\end{align*}

\noindent\textbf{2.} For the one-dimensional periodic formula see Berndt--Schoenfeld \cite[Theorem 4.4]{BerndtSchoenfeld75}. We encountered its truncated one-dimensional form in Matom\"aki--Merikoski \cite[Lemma 3.4]{MatomakiMerikoski23}. The following coset argument proves \eqref{eq:matomaki_merikoski} for the higher dimensional case:
\begin{align*}
 \sum_{\vv\in\Z^n}W(\vv)f(\vv)
 &=\sum_{\uu\in(\Z/q\Z)^n}W(\uu)\sum_{\mm\in\Z^n}f(\uu+q\mm)\\
 &=\sum_{\uu\in(\Z/q\Z)^n}W(\uu)\frac1{q^n}\sum_{\mm\in\Z^n}\widehat f(\mm/q)\en{\uu\cdot \mm/q}\\
 &=\frac1{q^n}\sum_{\mm\in\Z^n}\widehat f(\mm/q)
 \left(\sum_{\uu\in(\Z/q\Z)^n}W(\uu)\en{\uu\cdot \mm/q}\right)\\
 &=\frac1{q^n}\sum_{\mm\in\Z^n}\widehat f(\mm/q)\mathcal F_q(W)(\mm).
\end{align*}
This finishes the proof of general APSF.
\end{proof}

\subsection{The scalar formula, polynomial and vector divisibility}
\begin{proof}[Proof of Theorem \ref{thm:T1}]
Applying Theorem \ref{thm:gen_APSF} with period $H(d)=1$, weight $W_d=1$, coefficient $b(d)$, and test function $F_d$, its lattice-side absolute-convergence assumption is \eqref{eq:T1eq2}, and it gives
$$
 \sum_d b(d)\sum_{j\in\Z}F_d(j)
 =\sum_d b(d)\left(\sum_{r\in\Z}\widehat F_d(r)\right).
$$
Now, just observe that the origin vanishes because $c(0)=0$ and regrouping the absolutely convergent left-hand side by $n=h(d)j$ gives us the expression
$$
 \sum_d b(d)\sum_{j\ne0}g(h(d)j)c\left(\frac{h(d)j}{\omega(d)}\right)
 =\sum_{n\ne0}g(n)\sum_{h(d)\mid n}b(d)c\left(\frac n{\omega(d)}\right),
$$
which is the required arithmetic side of the theorem.
\end{proof}
\begin{remark}[Berndt form and the zero term]\label{cor:berndt}\label{rem:zero-term}
For $h(d)=\omega(d)=d$, the left side is $\sum_{n\ne0}g(n)(b*c)(n)$; if $g,c$ have the same parity this equals twice the positive sum. The case $c=1$ needs a separate origin correction. Suppose that each $g(d\,\cdot)$ is Poisson admissible and that
$$
 \sum_{d\ge1}|b(d)|\sum_{m\ne0}|g(dm)|<\infty.
$$
Apply Theorem \ref{thm:gen_APSF} first with finitely many nonzero coefficients $b(d)$, period $H(d)=1$, weight $W_d=1$, and test function $f_d(x)=g(dx)$. Removing the lattice origin in each block and then passing to the limit gives
$$
 \sum_{n\ne0}g(n)(b*\mathbf1)(|n|)
 =\sum_{d\ge1}b(d)\left\{\frac1d\sum_{r\in\Z}\widehat g(r/d)-g(0)\right\}.
$$
No parity assumption on $g$ is needed here. For even $g$, this becomes
\begin{equation}\label{eq:c-one-correction}
 2\sum_{n\ge1}g(n)(b*\mathbf1)(n)
 =\sum_{d\ge1}b(d)\left\{\frac1d\sum_{r\in\Z}\widehat g(r/d)-g(0)\right\}.
\end{equation}
The subtraction stays within each completed block unless the separated series converge.
\end{remark}

To prove the polynomial result, we need the following lemma:

\begin{lemma}\label{lem:coset-poisson}
Let $P\in\Z[x_1,\ldots,x_n]$. For $q\ge1$ define
\begin{align}
 C_q(P)&=\{\uu\in(\Z/q\Z)^n:P(\uu)\equiv0\pmod q\},\label{eq:Vq}\\
 \mathcal G_{P,q}(\mm)&=\sum_{\uu\in C_q(P)}\en{\mm\cdot \uu/q},\qquad \mm\in\Z^n.\label{eq:rho}
\end{align}
The function $\mathcal G_{P,q}$ is periodic modulo $q$ in every coordinate, and for $F\in\mathcal S(\R^n)$,
\begin{equation}\label{eq:coset-poisson}
 \sum_{\substack{\vv\in\Z^n\\q\mid P(\vv)}}F(\vv)
 =\frac1{q^n}\sum_{\mm\in\Z^n}\mathcal G_{P,q}(\mm)\widehat F(\mm/q).
\end{equation}
The same identity holds whenever Poisson summation is valid on every residue class $q\Z^n+\uu$, with both sides interpreted in the corresponding order.
\end{lemma}
\begin{proof}
We just need to apply Theorem \ref{thm:gen_APSF} with a single nonzero arithmetic coefficient equal to $1$, period $q$, weight $W(\vv)=\mathbf1_{q\mid P(\vv)}$, and test function $F$. Its finite Fourier transform is $\mathcal G_{P,q}$ by definition, so its conclusion is \eqref{eq:coset-poisson}.
\end{proof}

\begin{proof}[Proof of Theorem \ref{thm:polynomial-T1}]
Now take the data $H(d)=h(d)$, $W_d(\vv)=\mathbf1_{h(d)\mid P(\vv)}$, and $f_d=F_d$ in Theorem \ref{thm:gen_APSF}. The residue-class admissibility assumption is the one in Lemma \ref{lem:coset-poisson}; the absolute-convergence assumption is \eqref{eq:poly-hyp}. Since $\mathcal F_{h(d)}(W_d)=\mathcal G_{P,h(d)}$, the right hand side is exactly that of \eqref{eq:poly-T1}, and for the left hand side, absolute convergence permits us to re-order the terms as
\begin{align*}
 \sum_{\vv\in\Z^n}\sum_{d\ge1}b(d)W_d(\vv)F_d(\vv)
 &=\sum_{\vv\in\Z^n}g(\vv)\sum_{h(d)\mid P(\vv)}b(d)c\left(\frac{P(\vv)}{\omega(d)}\right)\\
 &=\sum_{\vv\in\Z^n}g(\vv)B(P(\vv)).
\end{align*}
For $P(\vv)\ne0$ the inner sum is finite because $h(d)\to\infty$; for $P(\vv)=0$ every term vanishes because $c(0)=0$. This finishes the proof.
\end{proof}
\begin{proof}[Proof of Theorem \ref{thm:vector-dilation}]
In Theorem \ref{thm:gen_APSF} take dimension $n$, $H(d)=1$, $W_d=1$, coefficient $b(d)$, and test function $F_d$. Its absolute-convergence assumption is \eqref{eq:vec-hyp}. Its right hand side is the right side of \eqref{eq:vector-dilation}. The terms with $\jj=0$ on the left hand side vanish since $P(0)=0$ and $c(0)=0$; setting $\vv=h(d)\jj$ and regrouping the absolutely convergent arithmetic sum permit us to obtain
\begin{align*}
 \sum_{d\ge1}b(d)\sum_{\jj\in\Z^n}F_d(\jj)
 &=\sum_{d\ge1}b(d)\sum_{\jj\ne0}g(h(d)\jj)c\left(\frac{P(h(d)\jj)}{\omega(d)}\right)\\
 &=\sum_{\vv\ne0}g(\vv)\sum_{h(d)\mid\gcd(\vv)}b(d)c\left(\frac{P(\vv)}{\omega(d)}\right).
\end{align*}
This is the left side of \eqref{eq:vector-dilation}.
\end{proof}
\begin{remark}[An origin-retaining specialization]\label{rem:constant-polynomial}
The polynomial in Theorem \ref{thm:polynomial-T1} need not vanish at the origin. Take $P(\xx)=1$, $c(t)=t$, $\omega(d)=1$, $h(d)=d$, and $b(d)=\mathbf1_{d=1}$. Its conclusion is ordinary Poisson summation for $g$, including $g(0)$. In particular, applying it to $g(\xx)=f(\uu+B\xx)$, with $B$ an invertible real matrix, gives
$$
 \sum_{\jj\in\Z^n}f(\uu+B\jj)
 =\frac1{|\det B|}\sum_{\mm\in\Z^n}
       \en{\uu\cdot B^{-T}\mm}\widehat f(B^{-T}\mm),
$$
whenever that translated function is admissible. This elementary specialization is useful for the fixed residue classes and origin corrections below; it carries no claim of a new arithmetic transform.
\end{remark}

\subsection{Convolutions, products, and shifts}\label{sec-conv}
\begin{proof}[Proof of Theorem \ref{T2}]
Apply Theorem \ref{thm:gen_APSF} in dimension $\ell$ with $H(d)=1$, $W_d=1$, coefficient $b_0(d)$ and test function $F_d$. Its absolute-convergence assumption is precisely \eqref{eq:T2-grouped-hyp}, and it gives
$$
 \sum_{d\ge1}b_0(d)\sum_{\mm\in\Z^\ell}F_d(\mm)
 =\sum_{d\ge1}b_0(d)\left(\sum_{\rr\in\Z^\ell}\widehat F_d(\rr)\right).
$$
The Fourier integral $\widehat F_d(\rr)$ is the term $I(d;r_1,\ldots,r_\ell)$. On the arithmetic side, terms on coordinate hyperplanes vanish; since each choice of signs of the nonzero $m_j$ gives the same value (for odd functions the sign from $g_0$ cancels the product of signs from the $c_j$),  restricting to positive $m_j$ therefore gives a factor $2^\ell$. Grouping by $n=dm_1\cdots m_\ell$ gives the iterated convolution in \eqref{eq:T2}.
\end{proof}
\begin{proof}[Proof of Theorem \ref{Tproduct}]
Enumerate $\dd\in\N^\ell$ and apply Theorem \ref{thm:gen_APSF} in dimension one with period $1$, weight $1$, coefficient $\prod_jb_j(d_j)$, and test function $H_{\dd}$. Hypothesis \eqref{eq:product-grouped-hyp} justifies the enumeration and gives
$$
 \sum_{\dd}\prod_jb_j(d_j)\sum_{m\in\Z}H_{\dd}(m)
 =\sum_{\dd}\prod_jb_j(d_j)\left(\sum_{r\in\Z}\widehat H_{\dd}(r)\right).
$$
The term $m=0$ vanishes because at least one $c_j(0)=0$. On putting $n=L(\dd)m$, the left side becomes
$$
 \sum_{n\ne0}g(n)\sum_{\substack{d_1\mid n,\ldots,d_\ell\mid n}}
 \prod_{j=1}^\ell b_j(d_j)c_j(n/d_j)
 =\sum_{n\ne0}g(n)\prod_{j=1}^\ell(b_j*c_j)(n).
$$
The Fourier integral of $H_{\dd}$ is the one in \eqref{eq:Tproduct}, proving the identity.
\end{proof}

\begin{proof}[Proof of Theorem \ref{Tshift}]
\textbf{1.}For each pair $(d,\delta)$, take as the period $H=d\delta$, weight function
$W_{d,\delta}(v)=\mathbf1_{d\mid v}\mathbf1_{\delta\mid v+\kappa}$ and test function $f_{d,\delta}(x)=g(x)c(x/d)c_0((x+\kappa)/\delta)$.
Using the coefficient $b(d)b_0(\delta)$ in Theorem \ref{thm:gen_APSF}, the admissibility on the required residue classes is assumed in Theorem \ref{Tshift} and the absolute lattice-side convergence follows from polynomial growth and the exponential decay of $g$: the sums over divisors of $|v|$ and $v+\kappa$ have at most polynomial growth, while the terms with $v=0$ or $v=-\kappa$ vanish and the terms with $v<-\kappa$ vanish by the hypothesis on $c_0$. To compute the finite transform, write
$$
 W_{d,\delta}(v)=\frac1H\sum_{j=0}^{d-1}\sum_{k=0}^{\delta-1}
 \ee{(j/d+k/\delta)v}\ee{\kappa k/\delta}.
$$
and observe that the finite character orthogonality gives
$$
 \mathcal F_H(W_{d,\delta})(m)
 =\sum_{\substack{0\le j<d,\ 0\le k<\delta\\m\equiv j\delta+kd\pmod H}}
 \ee{\kappa k/\delta}.
$$

\noindent\textbf{2.} The completed Fourier sum supplied by Theorem \ref{thm:gen_APSF} is given by
$$
 \frac1H\sum_{m\in\Z}\mathcal F_H(W_{d,\delta})(m)\widehat f_{d,\delta}(m/H)
 =\sum_{k=0}^{\delta-1}\ee{\kappa k/\delta}
 \sum_{r\in\Z}\frac1H\widehat f_{d,\delta}\left(\frac{r\delta+kd}{H}\right)\,,
$$
where  $m=H\ell+j\delta+kd$ and $r=d\ell+j$; observe that each frequency sum is completed before summing over $d,\delta$. In the Fourier integral, the change of variables $x=Hy$ gives
$$
 \frac1H\widehat f_{d,\delta}\left(\frac{r\delta+kd}{H}\right)
 =\int_\R\ee{(r\delta+kd)y}g(d\delta y)c(\delta y)c_0(dy+\kappa/\delta)\,dy
 =J(d,\delta,k,r).
$$
Each completed Fourier block equals its arithmetic block, so its absolute value is bounded by the corresponding absolute lattice sum. The absolute convergence established in step 1 therefore proves absolute summability of the completed outer blocks; no absolute convergence of the individual Fourier terms is needed. Regrouping the absolutely convergent arithmetic side by $v=n$ gives the left hand side of \eqref{eq:Tshift}, including the stated vanishing endpoint terms, finishing the proof of the theorem.
\end{proof}

\section{Basic applications}\label{part:TIApps}
The first examples show how the dilation and interpolant select arithmetic coefficients. The Lambert identities below also follow by direct divisor regrouping: evaluating the completed Fourier sum returns the corresponding lattice sum. We retain the Fourier calculation to identify the kernels and their normalizations. The later reciprocity arguments use the rational partial-fraction representation before regrouping its two indices.
\subsection{Exponential kernels and arithmetic selectors}
For $s>0$ and integers $a\ge1$, $j\ge0$ with $N=a+j$ even, use $c(x)=x^a$, $g(x)=x^j e^{-s|x|}$. The elementary Fourier integral and partial-fraction identity are
\begin{equation}\label{eq:mittag-leffler}
 \widehat{e^{-s|\cdot|}}(r)=\frac{2s}{s^2+4\pi^2r^2},\qquad
 \sum_{r\in\Z}\frac{u}{u^2+r^2}=\pi\coth(\pi u).
\end{equation}
Differentiating $N$ times and setting $u=sh(d)/(2\pi)$ gives
$$
 \sum_r\widehat F_d(r)
 =\frac{h(d)^N}{\omega(d)^a}\left.\frac{d^N}{dt^N}\coth(t/2)\right|_{t=sh(d)}
 =\frac{2h(d)^N}{\omega(d)^a}\Li_{-N}(e^{-sh(d)}),
$$
where $\Li_{-N}(e^{-t})=\sum_{m\ge1}m^N e^{-tm}$. Differentiation under the integral and frequency sum follows from exponential decay and locally uniform convergence for $s>0$. Thus
\begin{equation}\label{eq:master}
 2\sum_{n\ge1}n^N\exp(-sn)\sum_{h(d)\mid n}\frac{b(d)}{\omega(d)^a}
 =2\sum_{d\ge1}\frac{b(d)h(d)^N}{\omega(d)^a}\Li_{-N}(\exp(-sh(d))).
\end{equation}
\begin{remark}[Convergence]
The exact absolute-convergence condition for this regrouping is
$$
 \sum_{d\ge1}\frac{|b(d)|h(d)^N}{|\omega(d)|^a}\Li_{-N}(\exp(-sh(d)))<\infty.
$$
For example, it holds for every $s>0$ if $h(d)\ge c d^\delta$ for constants $c,\delta>0$ and $|b(d)|h(d)^N/|\omega(d)|^a$ has at most polynomial growth. The left side of \eqref{eq:master} then converges absolutely as well, by regrouping its absolute majorant. A lower bound on $h(d)/\omega(d)$ alone does not supply the required exponential decay.
\end{remark}

\begin{example}[Additive twists]
Let $\alpha\in\R$ and choose
$$
 c_\alpha(x)=\cos(2\pi\alpha x)-1,\quad g(x)=\exp(-s|x|),\quad h(d)=\omega(d)=d.
$$
The function $c_\alpha$ is continuous and $c_\alpha(0)=0$. The elementary lattice sums
$$
 \sum_{m\in\Z}\exp(-x|m|)\cos(2\pi\alpha m)=\frac{\sinh x}{\cosh x-\cos(2\pi\alpha)},
 \qquad\sum_{m\in\Z}\exp(-x|m|)=\coth(x/2)
$$
give the exact identity from Theorem \ref{thm:T1}:
\begin{align}\label{eq:additive}
 &2\sum_{n\ge1}\exp(-sn)\sum_{d\mid n}b(d)\{\cos(2\pi\alpha n/d)-1\}\notag\\
 &=\sum_{d\ge1}b(d)\left\{\frac{\sinh(sd)}{\cosh(sd)-\cos(2\pi\alpha)}-\coth(sd/2)\right\}.
\end{align}
The formula is valid for any polynomially growing $b$.
\end{example}

\begin{example}[Logarithm of Euler's product]\label{ex:euler-product}
For $s>0$, put $q=e^{-s}$ and take in Theorem \ref{thm:T1}
$$
 b(d)=1,\qquad h(d)=\omega(d)=d,\qquad g(x)=e^{-s|x|},\qquad
 c(x)=\begin{cases}(1-\sinc x)/|x|,&x\ne0,\\0,&x=0.\end{cases}
$$
This interpolant is continuous, since $c(x)=\pi^2|x|/6+O(|x|^3)$ at zero, and satisfies $c(m)=1/|m|$ at nonzero integers. The resulting identity is
\begin{equation}\label{eq:euler-product-apsf}
 2\sum_{n\ge1}\sigma_{-1}(n)q^n
 =\sum_{d\ge1}\left(\sum_{r\in\Z}I_{sd}(r)\right)
 =-2\log\prod_{d\ge1}(1-q^d),
\end{equation}
where the individual Fourier kernel has the real integral representation
\begin{equation}\label{eq:euler-product-kernel}
 I_a(r)=\frac12\int_{-1}^1
 \log\frac{a^2+(2\pi r+\pi v)^2}{a^2+(2\pi r)^2}\,dv,
 \qquad a>0.
\end{equation}
The product in \eqref{eq:euler-product-apsf} is Euler's classical product; this gives its logarithm through a scalar APSF kernel.

Here $F_d(x)=e^{-sd|x|}c(x)$. For $a>0$, evenness and
$\sinc x=\frac12\int_{-1}^1\cos(\pi vx)\,dv$ give
$$
 I_a(r)=2\int_0^\infty e^{-ax}(1-\sinc x)\cos(2\pi rx)\,\frac{dx}{x}.
$$
The cosine form of Frullani's integral is
$$
 \int_0^\infty e^{-ax}\{\cos(bx)-\cos(cx)\}\,\frac{dx}{x}
 =\frac12\log\frac{a^2+c^2}{a^2+b^2}.
$$
Indeed, differentiating in $c$ gives $c/(a^2+c^2)$ on both sides, and both vanish at $c=b$. Apply this identity after expanding
$\cos(2\pi rx)\cos(\pi vx)$ into its two cosine terms. Integrating over the symmetric interval $[-1,1]$ gives
\eqref{eq:euler-product-kernel}. As the integrals involved converge
absolutely (the cancellation at the origin is of order $x^2$, while
the exponential factor ensures convergence at infinity), observe that the remaining
integral can also be evaluated explicitly: indeed, if we consider the auxiliary function
$$
 H_a(y)=y\log(a^2+y^2)-2y+2a\arctan(y/a),
$$
we obtain the expression
$$
 I_a(r)
 =\frac{H_a(2\pi r+\pi)-H_a(2\pi r-\pi)}{2\pi}
  -\log(a^2+4\pi^2r^2).
$$

To apply Theorem \ref{thm:T1}, we first verify Poisson admissibility.
For fixed $a>0$, the function $e^{-a|x|}c(x)$ and its derivatives
decay exponentially at infinity; its first derivative has a finite
jump at the origin, and its second derivative is integrable on each
half-line. So, integrating by parts separately on the two half-lines gives us
$$
 I_a(r)=O_a\bigl((1+|r|)^{-2}\bigr).
$$
The periodization of $e^{-a|x|}c(x)$ converges uniformly and defines
a continuous periodic function whose Fourier series converges
absolutely by the preceding estimate; this establishes Poisson
admissibility and gives the formula
$$
 \sum_{r\in\Z}I_a(r)
 =2\sum_{m\ge1}\frac{e^{-am}}m
 =-2\log(1-e^{-a}).
$$
The required arithmetic summation is also absolutely convergent,
since, for $0<q<1$,
$$
 \sum_{d,m\ge1}\frac{q^{dm}}m
 =\sum_{m\ge1}\frac{q^m}{m(1-q^m)}
 \le\frac1{1-q}\sum_{m\ge1}\frac{q^m}m
 <\infty.
$$
Thus Theorem \ref{thm:T1} applies, and we may regroup the double sum
by $n=dm$. The coefficient of $q^n$ is then
$\sum_{m\mid n}m^{-1}=\sigma_{-1}(n)$, whereas summing first over $m$
gives $-\sum_{d\ge1}\log(1-q^d)$. These two expressions yield
\eqref{eq:euler-product-apsf}; to finish our discussion, we just need to observe that the sum $\sum_{d\ge1}q^d<\infty$
ensures that the infinite product converges to a positive number,
so its logarithm is the real logarithm.
\end{example}

\begin{example}[Perfect powers]
Fix an integer $\ell\ge2$. Define the multiplicative function $\lambda_\ell$ by
\begin{equation}\label{eq:lambda-ell-dirichlet}
 \sum_{n\ge1}\frac{\lambda_\ell(n)}{n^z}=\frac{\zeta(\ell z)}{\zeta(z)},\qquad\Re z>1.
\end{equation}
Its prime-power values are
$$
 \lambda_\ell(p^j)=\begin{cases}
 1,&j\equiv0\pmod\ell,\\
 -1,&j\equiv1\pmod\ell,\\
 0,&\text{otherwise}.
 \end{cases}
$$
Consequently,
\begin{equation}\label{eq:lambda-selector}
 \sum_{d\mid n}\lambda_\ell(d)=\begin{cases}1,&n\text{ is an }\ell\text{-th power},\\0,&\text{otherwise}.
 \end{cases}
\end{equation}
In \eqref{eq:master}, take
$$
 h(d)=\omega(d)=d,\qquad b(d)=d^a\lambda_\ell(d).
$$
The divisor coefficient on the left becomes the indicator in \eqref{eq:lambda-selector}, and hence
\begin{equation}\label{eq:perfect-powers}
 2\sum_{m\ge1}m^{\ell N}\exp(-sm^\ell)
 =2\sum_{d\ge1}\lambda_\ell(d)d^N\Li_{-N}(\exp(-sd)).
\end{equation}
This is an exact perfect-power selector for every even $N=a+j$ with $a\ge1$.

For $\ell=2$, $\lambda_2$ is the Liouville function $\lambda$. When $N=2$,
$$
 2\Li_{-2}(\exp(-x))=\frac14\frac{\sinh x}{\sinh^4(x/2)}.
$$
Thus \eqref{eq:perfect-powers} specializes to
\begin{equation}\label{eq:square-selector}
 2\sum_{m\ge1}m^4\exp(-sm^2)
 =\frac14\sum_{d\ge1}\lambda(d)d^2\frac{\sinh(sd)}{\sinh^4(sd/2)}.
\end{equation}
\end{example}

\begin{example}[Power-free integers]
Let $k\ge2$. The elementary M\"obius identity
$$
 \sum_{d^k\mid n}\mu(d)=\mathbf1_{\{n\text{ is }k\text{-free}\}}
$$
is inserted into \eqref{eq:master} by taking
$$
 h(d)=\omega(d)=d^k,\qquad b(d)=d^{ka}\mu(d).
$$
The factors $d^{ka}$ cancel between $b(d)$ and $\omega(d)^a$. We obtain
\begin{equation}\label{eq:kfree}
 2\sum_{\substack{n\ge1\\n\text{ is }k\text{-free}}}n^N\exp(-sn)
 =2\sum_{d\ge1}\mu(d)d^{kN}\Li_{-N}(\exp(-sd^k)).
\end{equation}
Here the nonlinear dilation $h(d)=d^k$ is essential: it is what turns the divisibility condition into the $k$-free sieve. For fixed $s>0$ the right side is absolutely convergent, including the M\"obius signs.
\end{example}

\begin{example}[Characters and divisor coefficients]\label{sec:basic-characters}
Let $\chi$ be a primitive nonprincipal Dirichlet character modulo $q$. We need to construct its continuous trigonometric interpolant. We choose its representatives in the interval $-q/2<a\le q/2$, since this symmetric choice ensures that the interpolant has the same parity as $\chi$
for every real argument, extending the parity relation at the
integer samples. Moreover, when $q$ is even, the unpaired endpoint $a=q/2$
causes no difficulty, since $\chi(q/2)=0$. Now let
$$
 \tau(\chi)=\sum_{a\bmod q}\chi(a)\ee{a/q},\qquad
 P_\chi(x)=\frac1{\tau(\overline\chi)}\sum_{a\bmod q}\overline{\chi(a)}\ee{ax/q}.
$$
The Gauss-sum identity gives
$$
 P_\chi(n)=\chi(n)\quad(n\in\Z),\qquad P_\chi(0)=0,\qquad P_\chi(-x)=\chi(-1)P_\chi(x).
$$
Write $\chi(-1)=(-1)^\varepsilon$ with $\varepsilon\in\{0,1\}$. In Theorem \ref{thm:T1}, take
$$
 c(x)=P_\chi(x),\qquad g(x)=x^\varepsilon\exp(-s|x|),\qquad h(d)=\omega(d)=d.
$$
Since $g(dm)P_\chi(m)$ is even in $m$, Theorem \ref{thm:T1} yields
\begin{equation}\label{eq:character}
 2\sum_{n\ge1}n^\varepsilon\exp(-sn)(b*\chi)(n)
 =2\sum_{d\ge1}b(d)d^\varepsilon\sum_{m\ge1}m^\varepsilon\chi(m)\exp(-sdm),
\end{equation}
where $(b*\chi)(n)=\sum_{d\mid n}b(d)\chi(n/d)$. The sums converge absolutely when $b$ has at most polynomial growth, for example, and the two-square representation function satisfies
$$
 r_2(n)=4\sum_{d\mid n}\chi_{-4}(d)=4(\mathbf1*\chi_{-4})(n).
$$
For $\chi_{-4}$ one has the expression
$$
 2\sum_{m\ge1}m\chi_{-4}(m)\exp(-xm)=\frac{\sinh x}{\cosh^2x}.
$$
Putting $b(d)=1$ and $\varepsilon=1$ in \eqref{eq:character} gives
\begin{equation}\label{eq:r2}
 \frac12\sum_{n\ge1}nr_2(n)\exp(-sn)=\sum_{d\ge1}d\frac{\sinh(sd)}{\cosh^2(sd)}.
\end{equation}
If we consider the theta function
$$
 \vartheta(t)=\sum_{m\in\Z}\exp(-\pi tm^2),
$$
it is not hard to check that $\vartheta(s/\pi)^2=1+\sum_{n\ge1}r_2(n)\exp(-sn)$; in particular, this tells us that \eqref{eq:r2} is equivalent to
\begin{equation}\label{eq:r2-theta}
 -\frac12\frac d{ds}\vartheta(s/\pi)^2=\sum_{d\ge1}d\frac{\sinh(sd)}{\cosh^2(sd)}.
\end{equation}
The coefficient $b*\chi_{-4}$ identifies the arithmetic side of Theorem \ref{thm:T1} with the two-square representation sum. The displayed identity also follows directly by regrouping the divisor sum.
\end{example}

\begin{example}[Four squares]
The preceding two-square example has a companion using Jacobi's classical representation formula \cite{Grosswald85}:
$$
 r_4(n)=8\sum_{\substack{d\mid n\\4\nmid d}}d.
$$
Here $r_4(n)$ counts ordered representations by four integer squares, including signs. In Theorem \ref{thm:T1}, take
$$
 b(d)=d\mathbf1_{4\nmid d},\qquad h(d)=\omega(d)=d,\qquad
 c(x)=1-\sinc x,\qquad g(x)=e^{-s|x|},\quad s>0.
$$
The arithmetic side is $(\vartheta(s/\pi)^4-1)/4$. For $a=sd$, put $L_a(\xi)=2a/(a^2+4\pi^2\xi^2)$. The Fourier coefficient is
$$
 \widehat F_d(r)=L_a(r)-\int_{-1/2}^{1/2}L_a(r+u)\,du.
$$
Since the intervals tile the real line, the integral of $L_a$ is $1$, and the partial-fraction sum is $\sum_rL_a(r)=\coth(a/2)$, we obtain the formula
\begin{equation}\label{eq:four-square}
 \vartheta(s/\pi)^4-1
 =4\sum_{\substack{d\ge1\\4\nmid d}}d\{\coth(sd/2)-1\}.
\end{equation}
The majorant $2\sum_{d\ge1}d/(e^{sd}-1)<\infty$ verifies the arithmetic convergence hypothesis. This derives the theta identity from the scalar theorem using Jacobi's formula as input.
\end{example}

\begin{example}[Totient functions]
Take $\ell=2$, $b_1=b_2=\mu$, $c_1(x)=c_2(x)=x$, and $g(x)=\exp(-s|x|)$ as our data in Theorem \ref{Tproduct}. Since $(\mu*c_j)(n)=\varphi(n)$, the left side of the relation is
$$
 2\sum_{n\ge1}\varphi(n)^2\exp(-sn).
$$
Grouping $d_1,d_2$ according to $q=\operatorname{lcm}(d_1,d_2)$ gives
$$
 a(q)=\mu^2(q)\prod_{p\mid q}(1-2p).
$$
The local factor $1-2p$ comes from the three possibilities $(p\mid d_1,p\nmid d_2)$, $(p\nmid d_1,p\mid d_2)$, and $(p\mid d_1,p\mid d_2)$. Indeed, their contributions to $\mu(d_1)\mu(d_2)q^2/(d_1d_2)$ are $-p,-p,1$, respectively; putting all the contributions together, we obtain the exact identity
\begin{equation}\label{eq:eulertotient}
 2\sum_{n\ge1}\varphi(n)^2\exp(-sn)
 =2\sum_{q\ge1}a(q)\Li_{-2}(\exp(-sq))
 =\frac14\sum_{q\ge1}a(q)\frac{\sinh(sq)}{\sinh^4(sq/2)}.
\end{equation}
In a completely similar way, if we define, for each integer $A\ge1$,
$$
 J_A(n)=(\mu*x^A)(n),\qquad a_A(q)=\mu^2(q)\prod_{p\mid q}(1-2p^A),
$$
the same least-common-multiple calculation gives us
\begin{equation}\label{eq:jordan-square}
 2\sum_{n\ge1}J_A(n)^2\exp(-sn)
 =2\sum_{q\ge1}a_A(q)\Li_{-2A}(\exp(-sq)).
\end{equation}
\end{example}

\begin{example}[A resolvent identity]\label{sec:resolvents}\label{ex:resolvents}
We apply the origin-corrected formula \eqref{eq:c-one-correction}
with $g(x)=(1+s^2x^2)^{-1}$, where $s>0$. Taking
$b(d)=d^{-\eta}$ gives
\begin{equation}\label{eq:correct-resolvent}
 2\sum_{n\ge1}\frac{\sigma_{-\eta}(n)}{1+s^2n^2}
 =\sum_{d\ge1}d^{-\eta}
 \left\{\frac\pi{sd}\coth\frac\pi{sd}-1\right\},
 \qquad \Re\eta>-1.
\end{equation}
Indeed, since we have the trivial estimate
$$
 \sum_{m\ge1}\frac1{1+s^2d^2m^2}
 \le \frac{\zeta(2)}{s^2d^2},
$$
 the double series converges absolutely throughout this
half-plane. The subtraction of the origin term is essential here:
as $d\to\infty$,
$$
 \frac\pi{sd}\coth\frac\pi{sd}-1
 =\frac{\pi^2}{3s^2d^2}+O_s(d^{-4}).
$$
Keeping the subtraction inside each summand therefore allows
\eqref{eq:correct-resolvent} to hold for $\Re\eta>-1$;
separating it as $-\zeta(\eta)$ is justified only when $\Re\eta>1$. Observe also that the same formula with $b=\lambda$, the Liouville function,
selects the perfect squares, since
$\sum_{d\mid n}\lambda(d)$ equals $1$ when $n$ is a square
and $0$ otherwise. The preceding bound again gives absolute
convergence, and we obtain
\begin{align*}
 \sum_{d\ge1}\lambda(d)
 \left\{\frac\pi{sd}\coth\frac\pi{sd}-1\right\}
 &=2\sum_{m\ge1}\frac1{1+s^2m^4}\\
 &=\frac{\pi}{\sqrt{2s}}
   \frac{\sinh x+\sin x}{\cosh x-\cos x}-1,
 \qquad x=\frac{\sqrt2\pi}{\sqrt s}.
\end{align*}
Alternatively, we can satisfy the condition $c(0)=0$ directly in this case
by taking
$$
 c(x)=x^2,\qquad g(x)=(1+s^2x^2)^{-2}.
$$
Differentiating the partial-fraction expansion used above
evaluates the corresponding lattice sum, and under the
absolute convergence hypotheses of the summation formula, we get
$$
 2\sum_{n\ge1}\frac{n^2}{(1+s^2n^2)^2}
   \sum_{d\mid n}\frac{b(d)}{d^2}
 =\frac{\pi}{2s^3}\sum_{d\ge1}\frac{b(d)}{d^3}
   \left\{\coth x_d-x_d\csch^2x_d\right\},
 \qquad x_d=\frac\pi{sd}.
$$
\end{example}

\begin{example}[Mersenne--Fermat detector]\label{sec:mersenne} In earlier work, one of the authors considered generating functions for Mersenne and Fermat primes \cite{Panzone12} that use a divisor sum which vanishes when its argument has at least three distinct prime factors; the support computation below is that of \cite{Panzone12} and the new point here is the Gaussian representation \eqref{eq:mersenne-Gaussian}.  We first express that sum through Gaussian APSF, then identify its support. Let $\nu:\N\to\C$ be additive on coprime arguments, with $|\nu(k)|\le C_\nu k^A$ for some $A\ge0$. For an integer $M\ge3$ and real $s,z\ge0$, put
$$
 \mathcal D_{\nu,s,z}(M)=
 \sum_{\substack{3\le d\le M\\d\ \mathrm{odd}}}\frac1{d^z}
 \sum_{k\mid d^2-1}\mu(k)\nu(k)^2 k^{1-s}.
$$
For $t\ge M$, the Gaussian representation is
\begin{align}\label{eq:mersenne-Gaussian}
 \mathcal D_{\nu,s,z}(M)+O_\nu(t^{4A+5}e^{-t})
 ={}&\sqrt{\frac\pi t}
 \sum_{\substack{3\le d\le M\\d\ \mathrm{odd}}}\sum_{1\le k\le M^2}
 \frac{\mu(k)\nu(k)^2}{d^zk^s}\notag\\
 &\quad\times\sum_{r\in\Z}
 \exp\left(-\frac{\pi^2r^2}{k^2t}\right)
 \ee{\frac{r(d^2-1)}k}.
\end{align}
The implied constant can be chosen independently of $M,t,s,z$ in these ranges.

For the fixed integers $d,k$, put $f(x)=\exp\{-t(x-d^2+1)^2\}$. Apply Theorem \ref{thm:polynomial-T1} with the auxiliary constant polynomial $P=1$, interpolant $c(u)=u$, $b(j)=\mathbf1_{j=1}$, $h(j)=j$, $\omega(j)=1$, and $g(y)=f(ky)$. This is the origin-retaining specialization of Remark \ref{rem:constant-polynomial}; scaling gives $\widehat g(r)=k^{-1}\widehat f(r/k)$. Thus
$$
 \widehat f(r/k)=\sqrt{\frac\pi t}
 \exp\left(-\frac{\pi^2r^2}{k^2t}\right)
 \ee{\frac{r(d^2-1)}k}.
$$
The polynomial theorem therefore gives the exact equality $k\sum_{k\mid n}f(n)=\sum_r\widehat f(r/k)$. The centre $n=d^2-1$ contributes $k\mathbf1_{k\mid d^2-1}$. Every other integer has distance at least one from the centre, and
$$
 \sum_{j\in\Z\setminus\{0\}}e^{-tj^2}
 \le\frac{2e^{-t}}{1-e^{-3t}}\qquad(t\ge1).
$$
Multiply by $\mu(k)\nu(k)^2d^{-z}k^{-s}$ and sum over the finite $d,k$ ranges. All divisors of $d^2-1$ are at most $M^2$. The error is bounded by
$$
 C_\nu^2\frac{2e^{-t}}{1-e^{-3t}}
 \sum_{d\le M}\sum_{k\le M^2}k^{2A+1}
 \ll_\nu M^{4A+5}e^{-t}\le C'_{\nu}t^{4A+5}e^{-t},
$$
which proves \eqref{eq:mersenne-Gaussian} and its uniformity.

The arithmetic cancellation occurs at $s=1$. The following identity is due to Golomb \cite{Golomb70} (see also \cite{Panzone12}): if $N>1$ has distinct prime divisors $p_1,\ldots,p_j$, then
\begin{equation}\label{eq:mersenne-detector}
 \sum_{k\mid N}\mu(k)\nu(k)^2=
 \begin{cases}
 -\nu(p_1)^2,&j=1,\\
 2\nu(p_1)\nu(p_2),&j=2,\\
 0,&j\ge3.
 \end{cases}
\end{equation}
Indeed only squarefree divisors contribute. Write such a divisor as $\prod_{i\in S}p_i$, expand $(\sum_{i\in S}\nu(p_i))^2$, and sum with sign $(-1)^{|S|}$. The coefficient of $\nu(p_i)^2$ is $-(1-1)^{j-1}$, and that of $2\nu(p_i)\nu(p_l)$, for $i\ne l$, is $(1-1)^{j-2}$. The case $j=1$ is immediate. Thus the sum vanishes for $j\ge3$; if every $\nu(p)\ne0$, it is nonzero for $j=1,2$. The choice $\nu(k)=\log k$ has this property.

For odd $d\ge3$, put $a=(d-1)/2$ and $b=(d+1)/2$. Then $b-a=1$, $(a,b)=1$, and $d^2-1=4ab$. If this last integer has only the prime divisor $2$, then $a=1,b=2$, so $d=3$. If it has exactly two distinct prime divisors, the even member of $a,b$ is $2^u$ and the odd member is $p^v$, for an odd prime $p$ and integers $u,v\ge1$. When $v=1$, the two orientations give
$$
 p=2^u-1,\quad d=2p+1,
 \qquad\text{or}\qquad
 p=2^u+1,\quad d=2p-1.
$$
These are respectively Mersenne and Fermat primes: primality forces $u$ to be prime in the first case and a power of $2$ in the second. If $u=1$, the only possible odd prime power adjacent to $2$ is $3$, already included. When $u,v>1$, Mih\u{a}ilescu's theorem \cite{Mihailescu04} gives the pair $8,9$, hence $d=17$. Conversely, each listed case has at most two distinct prime divisors in $d^2-1$. Therefore the support of \eqref{eq:mersenne-detector}, for $N=d^2-1$ and $\nu(p)\ne0$, is exactly
$$
 \{3,17\}\ \cup\ \{2p+1:p\text{ is a Mersenne prime}\}
 \ \cup\ \{2p-1:p\text{ is an odd Fermat prime}\}.
$$
Together with \eqref{eq:mersenne-Gaussian}, this gives the weighted Gaussian detector. For $s\ne1$ the same Gaussian representation holds, while the prime-factor cancellation in \eqref{eq:mersenne-detector} is specific to $s=1$.
\end{example}

\section{Reciprocity laws and special values of modular forms}\label{sec:reciprocity}
\subsection{Classical Ramanujan reciprocity}
We first obtain reciprocity formulas for the Riemann zeta function and Dirichlet $L$-functions, and then pass to norm sums in field extensions. We express Ramanujan's formula \cite[Chapter 14, Entry 21(i)]{BerndtNotebooksII} through the scalar APSF kernel. The partial-fraction expansion of $\coth$ and finite polynomial division are Siegel's method, as recorded and extended to characters by Katayama \cite[p. 235]{Katayama74JMSJ}. To begin, define for each $m\ge1$ and $\alpha>0$,
$$
 A_m(\alpha)=\sum_{n\ge1}\frac1{n^{2m+1}(e^{2\pi\alpha n}-1)}.
$$

\begin{theorem}[Ramanujan reciprocity law]\label{b-thm:ramanujan}
For every integer $m\ge1$ and $\alpha>0$,
\begin{align}\label{eq:ramanujan}
2A_m(\alpha)+\zeta(2m+1)
={}&(-1)^m\alpha^{2m}\bigl(2A_m(\alpha^{-1})+\zeta(2m+1)\bigr)\notag\\
&+\frac{\zeta(2m+2)}\pi\bigl(\alpha^{-1}-(-1)^m\alpha^{2m+1}\bigr)\notag\\
&+\frac2\pi\sum_{j=1}^m(-1)^{m-j}\alpha^{2m-2j+1}\zeta(2j)\zeta(2m-2j+2).
\end{align}
\end{theorem}
\begin{proof}
For $a>0$ put $L_a(r)=a/[\pi(a^2+r^2)]$. Theorem \ref{thm:T1} with
\begin{equation}\label{eq:ramanujan-data}
 h(d)=\omega(d)=d,\quad b(d)=d^{-2m-1},\quad
 g_\alpha(x)=e^{-2\pi\alpha|x|},\quad c(x)=1-\sinc x
\end{equation}
has arithmetic side $2A_m(\alpha)$ and Fourier coefficient
$$
 \widehat F_d(r)=L_{\alpha d}(r)
       -\int_{-1/2}^{1/2}L_{\alpha d}(r+u)\,du.
$$
The integral terms sum to $1$. Consequently, separating only the zero frequency gives the first channel
\begin{equation}\label{eq:lorentz}
 2A_m(\alpha)+\zeta(2m+1)
 =\frac{\zeta(2m+2)}{\pi\alpha}+\frac{2\alpha}{\pi}S_m(\alpha),
 \qquad S_m(\alpha)=\sum_{d,r\ge1}\frac{d^{-2m}}{\alpha^2d^2+r^2}.
\end{equation}
Applying the same calculation at $\alpha^{-1}$ gives the second channel
\begin{equation}\label{b-eq:ram-second-channel}
 2A_m(\alpha^{-1})+\zeta(2m+1)
 =\frac{\alpha\zeta(2m+2)}\pi+\frac{2}{\pi\alpha}S_m(\alpha^{-1}).
\end{equation}
Both double sums converge absolutely: each summand of $S_m(\alpha)$ is at most $d^{-2m}r^{-2}$ and $m\ge1$.

The required finite partial-fraction decomposition is
$$
 \frac{d^{-2m}}{\alpha^2d^2+r^2}
 =\sum_{j=1}^m(-1)^{m-j}\alpha^{2m-2j}
       d^{-2j}r^{-2m+2j-2}
  +(-1)^m\alpha^{2m}\frac{r^{-2m}}{\alpha^2d^2+r^2}.
$$
Every product series here is absolutely convergent, as is the remainder. Fubini's theorem therefore permits the exchange of $d,r$ in that remainder and yields
\begin{equation}\label{b-eq:ram-division}
 S_m(\alpha)=\sum_{j=1}^m(-1)^{m-j}\alpha^{2m-2j}
 \zeta(2j)\zeta(2m-2j+2)
 +(-1)^m\alpha^{2m-2}S_m(\alpha^{-1}).
\end{equation}
Insert this into the first channel and eliminate $S_m(\alpha^{-1})$ using the second. The zero-frequency terms give
$\zeta(2m+2)(\alpha^{-1}-(-1)^m\alpha^{2m+1})/\pi$;
the finite products give the last line of \eqref{eq:ramanujan}. This proves the formula directly from the two completed channels.
\end{proof}

\begin{remark}The same formula holds for $\Re\alpha>0$, by locally uniform convergence and analytic continuation. This gives the useful evaluation
\begin{equation}\label{eq:zeta5}
 \zeta(5)=\frac{\pi^5}{294}-\frac{74}{35}A_2(1)+\frac4{35}A_2(2).
\end{equation}
Separating even and odd indices gives
$$
 A_2((1-i)/2)=\frac{33}{16}A_2(1)-A_2(1/2)-\frac1{16}A_2(2).
$$
Formula \eqref{eq:ramanujan} at $\alpha=1/2$ and $(1-i)/2$ gives
\begin{align*}
 2A_2(1/2)&=\frac18A_2(2)-\frac{15}{16}\zeta(5)+\frac{\pi^5}{288},\\
 2A_2((1-i)/2)+\frac12A_2(1)+\frac54\zeta(5)&=\frac{\pi^5}{252}.
\end{align*}
Eliminating $A_2(1/2)$ and $A_2((1-i)/2)$ proves \eqref{eq:zeta5}. This identity was found numerically by Plouffe \cite{Plouffe98,Plouffe06web} and proved by Vep\v{s}tas \cite{Vepstas06}, who also evaluates Ramanujan’s relation at the two points used above; here the only new
ingredient is the APSF derivation of Theorem \ref{b-thm:ramanujan}.
\end{remark}

\subsection{Reciprocities with characters}
For character reciprocity we need one additional Fourier calculation. The character is initially defined only at integers, whereas Theorem \ref{thm:T1} requires a continuous function. The Gauss-sum interpolation below supplies that function and makes its Fourier transform a finite sum of translates of an elementary rational kernel. The underlying finite Fourier identities are classical; see \cite[Chapter 8]{Apostol76}, \cite[\S27.10]{DLMF}, and the character Poisson formulas of \cite{BerndtJNT75}. A Dirichlet character $\chi$ modulo $q$ is a $q$-periodic completely multiplicative function with $\chi(1)=1$, nonzero exactly on the integers coprime to $q$; it is primitive if it is not induced from a proper divisor of $q$. For a primitive nonprincipal character $\psi$, using the same symmetric residue representatives as in Example \ref{sec:basic-characters}, put
$$
 \tau(\psi)=\sum_{a\bmod q}\psi(a)\ee{a/q},\qquad
 \delta_\psi=\begin{cases}1,&\psi(-1)=1,\\0,&\psi(-1)=-1,\end{cases}
$$
$$
 I_\psi(y)=\frac{\operatorname{sgn}(y)^{1-\delta_\psi}}{\tau(\bar\psi)}
 \sum_{a\bmod q}\bar\psi(a)\ee{ay/q},\qquad I_\psi(0)=0.
$$
The finite Gauss identity gives $I_\psi(m)=\psi(|m|)$ for every nonzero integer $m$, and $\tau(\psi)\tau(\bar\psi)=\psi(-1)q$.

\begin{remark}[Why this character interpolant?]
Theorem \ref{thm:T1} requires a continuous function that vanishes at zero, whereas a Dirichlet character is initially defined on integers. The function $I_\psi$ supplies the samples $\psi(|m|)$ required by the Lambert series. For even $\psi$, its finite exponential sum is already even. For odd $\psi$, the sign factor makes it even; continuity at zero follows because the finite exponential sum vanishes there. Primitivity gives the Gauss identity, and the nonprincipal hypothesis gives this vanishing. Finally, the explicit exponential sum lets us compute the Fourier integral term by term, which is the reason for choosing this interpolation in the next lemma.
\end{remark}

\begin{lemma}\label{prop:resolvent}
For ${\beta_0}>0$,
\begin{equation}\label{eq:reciprocity-layer}
 \sum_{n\ge1}\psi(n)\exp(-2\pi {\beta_0}n)
 =\frac q{\pi\tau(\bar\psi)}\sum_{m\ge1}\bar\psi(m)
 \frac{(q\beta_0)^{\delta_\psi}(im)^{1-\delta_\psi}}{q^2{\beta_0}^2+m^2}.
\end{equation}
For odd $\psi$ the sum on the right is taken in its natural order.
\end{lemma}
\begin{proof}
The Fourier transforms, with positive exponential, are
$$
 \int_\R\ee{ty}\exp(-2\pi {\beta_0}|y|)\,dy=\frac {\beta_0}{\pi({\beta_0}^2+t^2)},\qquad
 \int_\R\ee{ty}\operatorname{sgn}(y)\exp(-2\pi {\beta_0}|y|)\,dy
 =\frac{it}{\pi({\beta_0}^2+t^2)}.
$$
We apply Theorem \ref{thm:T1} with $b(d)=\mathbf1_{d=1}$, $h(d)=\omega(d)=d$, $g(y)=\exp(-2\pi\beta_0|y|)$ and $c(y)=I_\psi(y)$. Inserting the finite sum defining $I_\psi$, using the decomposition $m=qr+a$, and pairing positive and negative values of $m$ gives \eqref{eq:reciprocity-layer}. Now, before we merge residues, observe that each completed Fourier coefficient is $O(r^{-2})$, since $\sum_a\bar\psi(a)=0$ cancels the possible $r^{-1}$ term; observe also that for odd $\psi$, bounded character partial sums justify the merged natural-order series by summation by parts, and this proves the result.
\end{proof}

We now introduce the two Lambert series considered by Katayama \cite{Katayama74JMSJ}:
$$
 \Phi_{\chi,N}(\alpha)=\sum_{d,m\ge1}\chi(d)d^{-N}\exp(-2\pi\alpha dm),\qquad
 \Psi_{\chi,N}(\beta)=\sum_{d,m\ge1}\chi(m)d^{-N}\exp(-2\pi\beta dm).
$$
\begin{theorem}[Katayama's formulas for both parities]\label{thm:character}\label{thm:katayama}
Let $\alpha>0$ and let $\chi$ be primitive nonprincipal modulo $q$. Setting $\delta=\delta_\chi$, and letting $N=2k+\delta$, with $k\ge1$ for odd $\chi$ and $k\ge0$ for even $\chi$, we obtain 
\begin{align}\label{eq:character-reciprocity}
 L(N,\chi)+2\Phi_{\chi,N}(\alpha)
 ={}&\frac{L(N+1,\chi)}{\pi\alpha}
 +\frac{2\alpha}\pi\sum_{j=0}^{k-1}(-\alpha^2)^j\zeta(2j+2)L(N-1-2j,\chi)\notag\\
 &+\frac{2i^{1-\delta}(-1)^k\alpha^{N-1}}{\tau(\bar\chi)}
 \Psi_{\bar\chi,N}\left(\frac1{q\alpha}\right).
\end{align}
Thus the left value is $L(2k,\chi)$ for odd $\chi$ and $L(2k+1,\chi)$ for even $\chi$.
\end{theorem}
\begin{proof}
We use the same calculation as in classical Ramanujan reciprocity, with $b(d)=\chi(d)d^{-N}$. This choice gives us
\begin{equation}\label{eq:character-lorentz}
 L(N,\chi)+2\Phi_{\chi,N}(\alpha)
 =\frac\alpha\pi\sum_{r\in\Z}\sum_{d\ge1}\frac{\chi(d)d^{1-N}}{\alpha^2d^2+r^2}.
\end{equation}
The $r=0$ term is $L(N+1,\chi)/(\pi\alpha)$. For $r\ne0$, split the quotient as
$$
 \frac{d^{1-N}}{\alpha^2d^2+r^2}
 =\sum_{j=0}^{k-1}\frac{(-\alpha^2)^jd^{1-N+2j}}{r^{2j+2}}
 +\frac{(-\alpha^2)^k}{r^{2k}}\frac{d^{1-\delta}}{\alpha^2d^2+r^2}\,,
$$
and observe that the finite terms give the displayed products. Applying \eqref{eq:reciprocity-layer} with $\psi=\bar\chi$ and $\beta_0=r/(q\alpha)$ to the remainder and using $\tau(\chi)\tau(\bar\chi)=\chi(-1)q$ gives the last term of \eqref{eq:character-reciprocity}. For odd $\chi$, the last $L(1,\chi)$ and the remainder are justified first with finite $d$-sums and then by summation by parts. The only other conditional endpoint is even $\chi$, $k=0$: in this case, the bounded partial sums give
$$
 \left|\sum_{d>D}\frac{\chi(d)}{\alpha^2d^2+r^2}\right|
 \le\frac{Cq}{\alpha^2D^2+r^2},\qquad
 \sum_{r\ge1}\frac{Cq}{\alpha^2D^2+r^2}=O(q/(\alpha D));
$$
this justifies the interchange and gives Katayama’s formulas \cite[(1), (2)]{Katayama74JMSJ}. The same formulas
appear in \cite[Theorem 4]{BerndtAA28}, and the case of arbitrary periodic coefficients is in \cite[
Theorem 1]{Bradley02}.
\end{proof}

Our next result is based on the simple observation that the same computations allow both factors to carry characters. If we take $\chi_i$ to be primitive nonprincipal characters modulo $q_i$ for $i\in\{1,2\}$, and if we put $\delta_i=\delta_{\chi_i}$, we could define the extended Lambert series
$$
 \Phi_{\chi_1,\chi_2,N}(\alpha)=\sum_{d,m\ge1}\chi_1(d)\chi_2(m)d^{-N}\exp(-2\pi\alpha dm).
$$
We now prove this generalization.
\begin{theorem}[Character-pair reciprocity]\label{thm:pair-char}
Suppose $\alpha>0$, $N\ge1$ is an integer, and $\chi_1(-1)\chi_2(-1)=(-1)^{N+1}$. Put $\gamma=q_2\alpha$ and $t=(N-\delta_2+1-\delta_1)/2$. Then
\begin{align}\label{eq:pair-char}
 \Phi_{\chi_1,\chi_2,N}(\alpha)
 ={}&\frac{q_2\gamma^{\delta_2}i^{1-\delta_2}}{\pi\tau(\bar\chi_2)}
 \sum_{j=0}^{t-1}(-\gamma^2)^jL(N-\delta_2-2j,\chi_1)L(2j+1+\delta_2,\bar\chi_2)\notag\\
 &+\chi_2(-1)i^{N-1}\gamma^{N-1}\frac{q_2\tau(\chi_1)}{q_1\tau(\bar\chi_2)}
 \Phi_{\bar\chi_2,\bar\chi_1,N}\left(\frac1{q_1q_2\alpha}\right).
\end{align}
An empty finite sum is zero.
\end{theorem}
\begin{proof}
\textbf{1.} The parity condition gives $2t=N+1-\delta_1-\delta_2$, so $t$ is a
nonnegative integer. The proof starts with the finite identity
\begin{align*}
 &\frac{x^{\delta_2-N}y^{1-\delta_2}
       -(-\gamma^2)^t x^{1-\delta_1}y^{\delta_1-N}}
       {\gamma^2x^2+y^2}\\
 &\hspace{12mm}=
 \sum_{j=0}^{t-1}
 \frac{(-\gamma^2)^j}
      {x^{N-\delta_2-2j}y^{2j+1+\delta_2}}.
\end{align*}
This is the finite geometric-series identity; when $t=0$, both sides
are zero. Multiply by $\chi_1(x)\bar\chi_2(y)$ and sum over
$1\le x,y\le M$.\\

\noindent \textbf{2.} As $M\to\infty$, the right-hand side tends to
$$
 \sum_{j=0}^{t-1}(-\gamma^2)^j
 L(N-\delta_2-2j,\chi_1)L(2j+1+\delta_2,\bar\chi_2).
$$
Both exponents are at least one, and the characters are nonprincipal.
For the two terms on the left, the character-kernel identity
\eqref{eq:reciprocity-layer}, applied at scales $\alpha x$ and
$y/(q_1\gamma)$, gives respectively
$$
 \frac{\pi\tau(\bar\chi_2)}
      {q_2\gamma^{\delta_2}i^{1-\delta_2}}
 \Phi_{\chi_1,\chi_2,N}(\alpha),
 \qquad
 (-\gamma^2)^t
 \frac{\pi\tau(\chi_1)}
      {q_1\gamma^{2-\delta_1}i^{1-\delta_1}}
 \Phi_{\bar\chi_2,\bar\chi_1,N}
       \left(\frac1{q_1q_2\alpha}\right).
$$
Moving the second term to the right and multiplying by
$q_2\gamma^{\delta_2}i^{1-\delta_2}/(\pi\tau(\bar\chi_2))$
proves the formula, since
$$
 2t+\delta_1+\delta_2-2=N-1,
 \qquad
 (-1)^t i^{\delta_1-\delta_2}
 =\chi_2(-1)i^{N-1}.
$$

\noindent\textbf{3.} For completeness, the rectangular limits used above are justified
as follows. Bounded character partial sums and summation by parts give
\begin{align*}
 &\sum_{x\ge1}x^{\delta_2-N}
 \left|\sum_{y>M}
 \frac{\bar\chi_2(y)y^{1-\delta_2}}{\gamma^2x^2+y^2}\right|\\
 &\qquad\ll
 \sum_{x\ge1}
 \frac{x^{\delta_2-N}}
      {(\gamma^2x^2+M^2)^{(1+\delta_2)/2}}
 \ll\frac{\log(2M)}{M}.
\end{align*}
The first bound uses that $u^{1-\delta_2}/(\gamma^2x^2+u^2)$
tends to zero and is decreasing or has a single turning point.
The second follows by splitting at $x=M$ and using $N\ge1$.
Exchanging the variables gives the corresponding vanishing $x$-tail
for the second kernel. After the inner sums are evaluated, the outer
sums converge absolutely by exponential decay. Thus the rectangular
limits equal the iterated sums supplied by the character-kernel
identity, including when $N=1$.
\end{proof}

\begin{remark}
For $N\ge2$, Theorem \ref{thm:pair-char} is equivalent to \cite[Theorem 6.1]{BerndtStraub16}, with $\chi=\chi_1$,
$\psi=\bar\chi_2$, $k=N+1$ and $\tau=iq_2\alpha$; the case $N=1$ lies outside the hypothesis $k\ge3$ there.
For two characters of the same modulus see also \cite[Theorem 1]{BerndtAA28} and \cite[Section 4]{Berndt77}.
Our contribution here is the derivation from APSF, which also covers $N=1$. Its finite sum of products of $L$-values can also be expressed using generalized Bernoulli numbers, although this rewriting is not needed for our proof. A separate example gives an explicit theta-function evaluation. With $q_0=\exp(-2\pi\alpha)$, Theorem \ref{thm:T1} with $b(d)=\chi_{-4}(d)$, $c(x)=|x|^2$ and $g(x)=\exp(-2\pi\alpha|x|)$ gives
\begin{equation}\label{eq:CM-family-app1}
 \sum_{d,m\ge1}\chi_{-4}(d)m^2q_0^{dm}
 =\frac14\sum_{d\ge1}\chi_{-4}(d)\coth(\pi\alpha d)\operatorname{csch}^2(\pi\alpha d)
 =\frac1{16}\vartheta_2(q_0)^4\vartheta_3(q_0)^2.
\end{equation}
The first equality uses $\sum_{r\in\Z}(a+ir)^{-3}=\pi^3\coth(\pi a)\operatorname{csch}^2(\pi a)$; the second is Jacobi's theta identity, with $\vartheta_2(q)=\sum_{n\in\Z}q^{(n+1/2)^2}$ and $\vartheta_3(q)=\sum_{n\in\Z}q^{n^2}$. At $\alpha=1/2$, the classical values \cite{DLMF} $\vartheta_2=2^{-1/4}\vartheta_3$ and $\vartheta_3=\Gamma(1/4)/(\sqrt2\pi^{3/4})$ give the classical evaluation (equivalent to $\sum_{n\ge1}n^2/\cosh(\pi n)=\Gamma(1/4)^6/(128\pi^{9/2})$; cf.
Ramanujan’s Entries 15--17 in Chapter 17 of \cite{BerndtNotebooksIII})
\begin{equation}\label{eq:CM-point-app1}
 \frac14\sum_{\substack{d\ge1\\d\ \mathrm{odd}}}(-1)^{(d-1)/2}
 \frac{\cosh(\pi d/2)}{\sinh^3(\pi d/2)}
 =\frac{\Gamma(1/4)^6}{256\pi^{9/2}}.
\end{equation}
\end{remark}
\subsection{Reciprocities related to field extensions: Gaussian integers}\label{sec:fields_examples}\label{sec:field-reciprocity_examples}
We first work out the Gaussian-integer case, identifying the norm coefficients, computing the finite character transform, and deriving the resulting Lambert–Kelvin formulas. This example explains the arithmetic meaning of the sums and the constants appearing in the general imaginary-quadratic statement below. The facts about quadratic integer rings and their prime factorizations are classical; see \cite{Cox13,Grosswald85,Neukirch99}. The resolvent method is that of Koshliakov \cite{Koshliakov36}, and the untwisted imaginary-quadratic identities are studied in \cite{BK25}.

\subsubsection{The Gaussian integers: the lattice and its coefficients}
The ring of integers of $\Q(i)$ is $\Z[i]$, and
$$
 N(x+iy)=(x+iy)(x-iy)=x^2+y^2=:Q_G(x,y).
$$
Rounding the real and imaginary parts of a complex number leaves an error of squared modulus at most $1/2$. Applied to a quotient, this gives Euclidean division in $\Z[i]$; hence every ideal is principal. The norm-one elements are $\pm1,\pm i$, so every nonzero ideal has four generators of the same norm. Prime factorization gives the following number $F_G(p^j)$ of ideals of norm $p^j$:
$$
 F_G(p^j)=
 \begin{cases}
 1,&p=2,\\
 j+1,&p\equiv1\pmod4,\\
 1,&p\equiv3\pmod4,\ j\text{ even},\\
 0,&p\equiv3\pmod4,\ j\text{ odd}.
 \end{cases}
$$
Indeed, $2$ ramifies, a prime congruent to $1$ splits into two conjugate primes whose exponents may be chosen as $0,j$ through $j,0$, and a prime congruent to $3$ remains prime with norm $p^2$. These are exactly the local values of $\sum_{d\mid n}\chi_{-4}(d)$. Multiplying the local factors therefore gives
\begin{equation}\label{ex:gaussian-coefficients}
 r_2(n)=4F_G(n),\qquad F_G(n)=\sum_{d\mid n}\chi_{-4}(d),\qquad
 \sum_{n\ge1}\frac{F_G(n)}{n^s}=\zeta(s)L(s,\chi_{-4})\quad(\Re s>1).
\end{equation}
Thus a radial lattice sum can also be read as a divisor sum. The lattice basis matrix is the identity, with determinant one and dual norm $u^2+v^2$. For example, Theorem \ref{thm:polynomial-T1} with $P=1$, $c(u)=u$, $b(j)=\mathbf1_{j=1}$, $h(j)=j$, $\omega(j)=1$, and $g(\yy)=e^{-\pi t|\yy|^2}$ gives
$$
 \Theta_G(t):=1+4\sum_{n\ge1}F_G(n)e^{-\pi tn},\qquad
 \Theta_G(t)=t^{-1}\Theta_G(1/t),\quad t>0,
$$
because the planar Fourier transform of $e^{-\pi t|\yy|^2}$ is $t^{-1}e^{-\pi|\xxi|^2/t}$. For the character-weighted sums below, we will specify the test function on each residue class and then insert its finite Fourier coefficient.

\subsubsection{Computing the two Fourier kernels}
We use the modified Bessel function of the second kind of order zero,
$$
 K_0(z)=\int_0^\infty e^{-z\cosh v}\,dv,\qquad \Re z>0
 \quad\text{\cite[\S10.32]{DLMF}}.
$$
The real and imaginary parts along the ray $e^{i\pi/4}z$ are the classical order-zero Kelvin functions. We suppress their order-zero subscript and write \cite[Eq. 10.61.2]{DLMF}
\begin{equation}\label{eq:Bessel-combinations}
 \begin{aligned}
 \ker(z)&=\frac{K_0(e^{i\pi/4}z)+K_0(e^{-i\pi/4}z)}2,\\
 \kei(z)&=\frac{K_0(e^{i\pi/4}z)-K_0(e^{-i\pi/4}z)}{2i},\qquad z>0.
 \end{aligned}
\end{equation}
Both are real because the two values of $K_0$ are conjugate. We use these standard names throughout the planar reciprocity formulas displayed below.

We want to obtain the following identity, valid for $x>0$:
\begin{align}\label{ex:gaussian-untwisted}
 \sum_{n\ge1}\frac{F_G(n)}{n^2+x^2}
 ={}&\frac{\pi^2}{8x}-\frac1{4x^2}\notag\\
 &-\frac{2\pi}{x}\sum_{n\ge1}F_G(n)\kei(2\pi\sqrt{xn}).
\end{align}
The denominator determines the test function. Since $n=|\vv|^2$ when the lattice points are grouped by their squared lengths, apply Theorem \ref{thm:polynomial-T1} with
$$
 P=1,\qquad c(u)=u,\qquad b(j)=\mathbf1_{j=1},\qquad
 h(j)=j,\qquad\omega(j)=1,\qquad
 g(\yy)=\frac1{|\yy|^4+x^2}.
$$
These data give $c(P)=1$ and retain the origin, as in Remark \ref{rem:constant-polynomial}. The APSF becomes ordinary planar Poisson summation:
$$
 \frac1{x^2}+4\sum_{n\ge1}\frac{F_G(n)}{n^2+x^2}
 =\widehat g(0)+\sum_{\mm\in\Z^2\setminus\{0\}}\widehat g(\mm).
$$
Thus it remains to compute the Fourier transform of this particular $g$.

For $\Re z>0$ and $\xxi\ne0$, integrating the planar Gaussian gives the resolvent transform, with the principal square root,
$$
 \pi\int_0^\infty v^{-1}e^{-zv-\pi^2|\xxi|^2/v}\,dv
 =2\pi K_0(2\pi|\xxi|\sqrt z)
 \qquad\text{\cite[\S10.32]{DLMF}}.
$$
This is the transform of $(|\yy|^2+z)^{-1}$ away from zero, understood with Gaussian damping. With $r=|\yy|$, use the partial fractions
$$
 \frac1{r^4+x^2}
 =\frac1{2ix}\left(\frac1{r^2-ix}-\frac1{r^2+ix}\right).
$$
To make the limiting step explicit, take $z=\epsilon\pm ix$, with $\epsilon>0$, and form
the difference before removing the Gaussian damping. The resulting function is
$$
\frac1{2ix}\left(\frac1{r^2+\epsilon-ix}
-\frac1{r^2+\epsilon+ix}\right)
=\frac1{(r^2+\epsilon)^2+x^2}
\le\frac1{r^4+x^2}.
$$
The last function is integrable on the plane. Dominated convergence therefore permits
$\epsilon\downarrow0$ in the Fourier integral. Continuity of the principal-square-root
Bessel expression at $z=\pm ix$ gives
\begin{equation}\label{ex:radial-pairs}
 \widehat g(\xxi)
 =\frac{\pi}{ix}\{K_0(e^{-i\pi/4}2\pi|\xxi|\sqrt x)
                  -K_0(e^{i\pi/4}2\pi|\xxi|\sqrt x)\}
 =-\frac{2\pi}{x}\kei(2\pi|\xxi|\sqrt x).
\end{equation}
At zero, polar coordinates give
$$
 \widehat g(0)=\pi\int_0^\infty\frac{dv}{v^2+x^2}=\frac{\pi^2}{2x}.
$$
The function $g$ is smooth and has integrable derivatives of every order, so its periodization and Fourier series justify the Poisson application. Grouping the nonzero dual vectors by $|\mm|^2=n$, using $r_2(n)=4F_G(n)$, and dividing by four proves \eqref{ex:gaussian-untwisted}. The two elementary terms come from the Fourier zero frequency and the removed lattice origin, respectively.

The character calculation will also require the numerator $n/(n^2+x^2)$. Its planar counterpart $|\yy|^2/(|\yy|^4+x^2)$ is not integrable. Instead, use the integrable function
$$
 g(\yy)=\frac{|\yy|^2}{(|\yy|^4+x^2)^2},\qquad
 \frac{r^2}{(r^4+x^2)^2}
 =\frac1{4ix}\left\{\frac1{(r^2-ix)^2}-\frac1{(r^2+ix)^2}\right\}.
$$
The squared resolvents are integrable. Differentiating the Gaussian resolvent transform with respect to $z$ and substituting $z=\pm ix$ yields
\begin{equation}\label{eq:ker-integrable-transform}
 \widehat g(\xxi)=-\frac\pi x\frac{d}{dx}\ker(2\pi|\xxi|\sqrt x)
 \quad(\xxi\ne0),\qquad
 \widehat g(0)=\frac{\pi}{2x^2}.
\end{equation}
After applying APSF to this $g$, integration in $x$ will recover the required numerator and the Kelvin function $\ker$. This keeps every Poisson application within the integrable case.

\subsubsection{A Gaussian character twist, computed modulo three}
Let $\eta=\chi_{-3}$, so $\eta(0)=0$, $\eta(1)=1$, $\eta(2)=-1$ modulo $3$, and define
$$
 A_G(n)=\eta(n)F_G(n),\qquad
 Z_G(s)=\sum_{n\ge1}\frac{A_G(n)}{n^s}
       =L(s,\chi_{-3})L(s,\chi_{12})\quad(\Re s>1),
$$
where $\chi_{12}=\chi_{-3}\chi_{-4}$. The last identity follows by writing $n=de$ and distributing the completely multiplicative factor $\eta(de)$ over the divisor sum in \eqref{ex:gaussian-coefficients}.

The weight $W(a,b)=\eta(a^2+b^2)$ is zero at $(0,0)$, equals $1$ at the four nonzero axis points modulo $3$, and equals $-1$ at the four remaining points. Write $S(u)=\en{u/3}+\en{2u/3}$; this is $2$ when $3\mid u$ and $-1$ otherwise. Direct summation gives
\begin{equation}\label{ex:gaussian-finite}
 \sum_{a,b\bmod3}W(a,b)\en{(ua+vb)/3}
 =S(u)+S(v)-S(u)S(v)=3\eta(u^2+v^2).
\end{equation}
The finite transform has the same norm weight and vanishes at zero. To use Theorem \ref{thm:polynomial-T1}, fix a residue vector $\aaa\in\{0,1,2\}^2$ and take
$$
 P=1,\quad c(u)=u,\quad b(j)=\mathbf1_{j=1},\quad
 h(j)=j,\quad\omega(j)=1,\quad
 g(\yy)=f(\aaa+3\yy).
$$
Here $f$ will be a Gaussian or one of the two integrable rational functions just computed. Scaling the Fourier integral gives
$$
 \sum_{\vv\in\Z^2}f(\aaa+3\vv)
 =\frac19\sum_{\mm\in\Z^2}\en{\aaa\cdot\mm/3}\widehat f(\mm/3).
$$
Multiply by $W(\aaa)$ and sum over the nine residues. Formula \eqref{ex:gaussian-finite} gives
\begin{equation}\label{eq:gaussian-weighted-poisson}
 \sum_{a,b\in\Z}W(a,b)f(a,b)
 =\frac13\sum_{u,v\in\Z}\eta(u^2+v^2)\widehat f(u/3,v/3).
\end{equation}
For $f(\yy)=e^{-\pi t|\yy|^2}$, grouping by the norm and dividing by four gives
$$
 \sum_{n\ge1}A_G(n)e^{-\pi tn}
 =\frac1{3t}\sum_{n\ge1}A_G(n)e^{-\pi n/(9t)}.
$$
For $f(\yy)=(|\yy|^4+x^2)^{-1}$, the same grouping and \eqref{ex:radial-pairs} give the first identity below. Using $f(\yy)=|\yy|^2/(|\yy|^4+x^2)^2$ and \eqref{eq:ker-integrable-transform} gives
$$
 \sum_{n\ge1}\frac{A_G(n)n}{(n^2+x^2)^2}
 =-\frac{\pi}{3x}\frac{d}{dx}
   \sum_{n\ge1}A_G(n)\ker\left(\frac{2\pi\sqrt{xn}}3\right).
$$
Replace $x$ by the integration variable $t$, multiply by $2t$, and integrate over $x\le t<\infty$. The Kelvin series tends to zero exponentially, so
$$
 \sum_{n\ge1}\frac{A_G(n)n}{n^2+x^2}
 =\frac{2\pi}3\sum_{n\ge1}A_G(n)\ker\left(\frac{2\pi\sqrt{xn}}3\right).
$$
Finally subtract this equality from $Z_G(1)=\sum_{n\ge1}A_G(n)/n$, using $1/n-n/(n^2+x^2)=x^2/[n(n^2+x^2)]$. We obtain
\begin{align}\label{ex:gaussian-resolvents}
 \sum_{n\ge1}\frac{A_G(n)}{n^2+x^2}
 &=-\frac{2\pi}{3x}\sum_{n\ge1}A_G(n)\kei\left(\frac{2\pi\sqrt{xn}}3\right),\notag\\
 x^2\sum_{n\ge1}\frac{A_G(n)}{n(n^2+x^2)}
 &=Z_G(1)-\frac{2\pi}3\sum_{n\ge1}A_G(n)\ker\left(\frac{2\pi\sqrt{xn}}3\right).
\end{align}
Here the conditionally convergent sums are taken in increasing order of $n$. To justify the integration, complete $3\times3$ residue blocks cancel, and only $O(\sqrt u)$ blocks meet the boundary of $a^2+b^2\le u$. Hence $\sum_{n\le u}A_G(n)=O(\sqrt u)$. Summation by parts proves convergence of $Z_G(1)$ and, uniformly in the finite cutoff $M$,
$$
 \left|\sum_{n\le M}\frac{A_G(n)n}{(n^2+t^2)^2}\right|\ll t^{-5/2}
 \qquad(t\ge1).
$$
Indeed the derivative of $u/(u^2+t^2)^2$ is bounded in absolute value by $3/(u^2+t^2)^2$, and $\int_0^\infty u^{1/2}(u^2+t^2)^{-2}\,du\ll t^{-5/2}$. Since $2t\,t^{-5/2}$ is integrable at infinity, dominated convergence permits integration from $t=x$ to infinity before taking $M\to\infty$. The Kelvin series and its derivatives converge locally uniformly for $x>0$. Continuity from $\Re s>1$ identifies $Z_G(1)$ with the displayed product of Dirichlet $L$-values.

\subsubsection{The Gaussian Lambert reciprocity}
We now finish the calculation at exponents three and two. Put
$$
 \sigma_{G,j}(M)=\sum_{d\mid M}A_G(d)d^j,\qquad
 a>0.
$$
The scalar Fourier partial-fraction calculation \eqref{eq:mittag-leffler} used with Theorem \ref{thm:T1} gives
$$
 1+\frac2{e^{2\pi an}-1}=\frac1{\pi an}
              +\frac{2an}\pi\sum_{r\ge1}\frac1{a^2n^2+r^2}.
$$
Multiply first by $A_G(n)n^{-3}$ and then by $A_G(n)n^{-2}$. In the first case the required division is
$$
 \frac{n^{-2}}{a^2n^2+r^2}
 =\frac{n^{-2}}{r^2}-\frac{a^2}{r^2(a^2n^2+r^2)}.
$$
Consequently the cubic expression is
$$
 \frac{Z_G(4)}{\pi a}+\frac{2a}\pi\zeta(2)Z_G(2)
 -\frac{2a}\pi\sum_{r\ge1}\frac1{r^2}
       \sum_{n\ge1}\frac{A_G(n)}{n^2+(r/a)^2}.
$$
Insert the first formula in \eqref{ex:gaussian-resolvents} with $x=r/a$. Setting $M=rn$ replaces $\sum_{rn=M}A_G(n)r^{-3}$ by $\sigma_{G,3}(M)/M^3$. This proves
\begin{align}\label{ex:gaussian-lambert-odd}
 Z_G(3)+2\sum_{n\ge1}\frac{A_G(n)n^{-3}}{e^{2\pi an}-1}
 ={}&\frac{Z_G(4)}{\pi a}+\frac{\pi a}{3}Z_G(2)\notag\\
 &+\frac{4a^2}{3}\sum_{M\ge1}\frac{\sigma_{G,3}(M)}{M^3}
 \kei\left(\frac{2\pi}3\sqrt{\frac M a}\right).
\end{align}
For the quadratic expression use instead
$$
 \frac1{n(a^2n^2+r^2)}
 =\frac1{r^2}\frac{(r/a)^2}{n(n^2+(r/a)^2)}.
$$
The second formula in \eqref{ex:gaussian-resolvents} now contributes $Z_G(1)\zeta(2)$ and the symmetric Bessel sum. The same regrouping gives
\begin{align}\label{ex:gaussian-lambert-even}
 Z_G(2)+2\sum_{n\ge1}\frac{A_G(n)n^{-2}}{e^{2\pi an}-1}
 ={}&\frac{Z_G(3)}{\pi a}+\frac{\pi a}{3}Z_G(1)\notag\\
 &-\frac{4a}{3}\sum_{M\ge1}\frac{\sigma_{G,2}(M)}{M^2}
 \ker\left(\frac{2\pi}3\sqrt{\frac M a}\right).
\end{align}
Since $|A_G(n)|\le\tau(n)$, the original rational double sums are absolutely convergent. The Bessel sums converge absolutely by exponential decay in $\sqrt{rn}$; the isolated $Z_G(1)$ has the natural-order meaning proved above. Thus these computations justify every interchange and display the reciprocity with all constants evaluated.

\subsection{Reciprocities related to field extensions: General case}\label{sec:fields}\label{sec:field-reciprocity}
The finite transform \eqref{ex:gaussian-finite} explains why an eigenrelation matters. It turns a Fourier coefficient indexed by two residues back into a character of a norm, allowing the dual lattice sum to be regrouped with the same arithmetic coefficients: a different weight still has a finite Fourier transform, but need not have this property! For general imaginary quadratic fields, binary quadratic forms organize the lattices attached to distinct ideal classes; the form--ideal correspondence is developed in \cite{Cox13,Neukirch99}. The example above has class number one and now we give the general calculation, allowing every ideal class and all character moduli in the stated range. For real quadratic fields such as $\Q(\sqrt3)$, passing from embedded elements to ideal norms requires additional averaging over the infinite unit group \cite{Neukirch99}.

The formulas below apply to every ideal class.
Let $K=\Q(\sqrt{-D})$ have discriminant $-D<0$, and let $\mathcal O_K$
be its ring of integers. For a nonzero integral ideal $\mathfrak a$, its
norm is $N\mathfrak a=[\mathcal O_K:\mathfrak a]$. Write $h_K$ for the
number of ideal classes and $w_K$ for the number of units of $\mathcal O_K$.
For each ideal class $C$, choose an integral ideal $\mathfrak a_C$ in the
inverse class $C^{-1}$ and an oriented integral basis
$\alpha_C,\beta_C$ of $\mathfrak a_C$. Its normalized norm is a positive
primitive integral quadratic form
$$
Q_C(x,y)=\frac{N_{K/\Q}(x\alpha_C+y\beta_C)}{N\mathfrak a_C}
=a_Cx^2+b_Cxy+c_Cy^2,\qquad 4a_Cc_C-b_C^2=D.
$$
This is the classical form--ideal correspondence; see
\cite{Cox13,Neukirch99}. Put
$$
r_C(n)=\#\{(x,y)\in\Z^2:Q_C(x,y)=n\}.
$$
To interpret this count, take a nonzero element
$\alpha=x\alpha_C+y\beta_C\in\mathfrak a_C$. Since
$(\alpha)\subseteq\mathfrak a_C$, the ideal
$$
\mathfrak b=(\alpha)\mathfrak a_C^{-1}
$$
is integral, belongs to $C$, and has norm
$$
N\mathfrak b=\frac{N_{K/\Q}(\alpha)}{N\mathfrak a_C}=Q_C(x,y).
$$
Conversely, if $\mathfrak b$ is an integral ideal in $C$, then
$\mathfrak b\mathfrak a_C$ is principal. Every generator $\alpha$ of
this product lies in $\mathfrak a_C$, because
$\mathfrak b\mathfrak a_C\subseteq\mathfrak a_C$, and gives a vector
$(x,y)$ in the displayed basis. Two such generators differ by a unit,
so there are exactly $w_K$ of them. Consequently
$$
r_C(n)=w_K\,\#\{\mathfrak b\subseteq\mathcal O_K:
[\mathfrak b]=C,\ N\mathfrak b=n\}.
$$
Summing over all ideal classes gives
$$
\sum_C r_C(n)=w_KF(n),\qquad
F(n)=\sum_{d\mid n}\chi_{-D}(d),\qquad
\zeta_K(s)=\sum_{n\ge1}\frac{F(n)}{n^s}
=\zeta(s)L(s,\chi_{-D}),\quad \Re s>1.
$$
Here $F(n)$ counts all nonzero integral ideals of norm $n$; its divisor
formula follows from the usual Euler factors of the quadratic field.
For $K=\Q(i)$, the single form is $x^2+y^2$ and $w_K=4$, recovering
the Gaussian example.

For a character $\chi$ modulo $q$ define
$$
 L_K(s,\chi)=\sum_{n\ge1}\frac{\chi(n)F(n)}{n^s}
 =L(s,\chi)L(s,\chi\chi_{-D}),\qquad \Re s>1.
$$
The key finite computation is the following: for $Q(x,y)=ax^2+bxy+cy^2$, the dual quadratic expression is
$$Q(-v,u)=cu^2-buv+av^2;$$
the rotation $(u,v)\mapsto(-v,u)$ is a bijection of $\Z^2$, so it preserves the representation numbers and no new quadratic-form symbol is needed: the eigenrelation says that the transformed periodic weight is again a conjugate character of the same norm, up to rotation and an explicit scalar. Without it, the Fourier formula still exists, but retains a general finite coefficient rather than a recognizable twisted norm series.
\begin{lemma}[Finite Fourier eigenrelation]\label{prop:eig}
Let $\chi$ be primitive modulo $q>1$. Suppose either $q$ is odd and $(q,D)=1$, or $q=2^j$, $j\ge2$, and $D$ is odd. Then
\begin{equation}\label{eq:norm-finite-eigen}
 \sum_{x,y\bmod q}\chi(Q(x,y))\en{\frac{1}q(ux+vy)}
 =c_\chi\bar\chi(Q(-v,u)),\qquad
 c_\chi=\chi_{-D}(q)\chi(D)\tau(\chi)^2,
\end{equation}
where for $q=2^j$ the first factor means $\chi_{-D}(2)^j$. In particular $|c_\chi|=q$ and the left side at $(u,v)=(0,0)$ is zero.
\end{lemma}
\begin{proof}
Insert $\chi(Q)=\tau(\bar\chi)^{-1}\sum_{t\bmod q}\bar\chi(t)\ee{tQ/q}$. For a unit $t$, completion of the two squares gives
$$
 \sum_{x,y\bmod q}\ee{\frac{1}q(tQ(x,y)-ux-vy)}
 =q\chi_{-D}(q)\en{\frac{1}q\overline{tD}\,Q(-v,u)}.
$$
For odd $q$, a determinant-one change of variables modulo $q$ makes the leading coefficient
a unit. Indeed, nondegeneracy modulo each prime dividing $q$ provides a vector on which
$Q$ is nonzero; combine these vectors by the Chinese remainder theorem and extend the
resulting vector to a determinant-one basis modulo $q$. Under $(x,y)^{\mathsf
T}=T(x',y')^{\mathsf T}$, the linear coefficient vector becomes $T^{\mathsf
T}(u,v)^{\mathsf T}$; the discriminant is unchanged. In the new coordinates we may
diagonalize as
$$
Q=a(x+b(2a)^{-1}y)^2+D(4a)^{-1}y^2.
$$
The two quadratic Gauss sums have product $q\chi_{-D}(q)$. Returning to the original
coordinates gives the displayed dual expression $Q(-v,u)$. For $q=2^j$, the matrix $\left(\begin{smallmatrix}2a&b\\b&2c\end{smallmatrix}\right)$ has odd determinant $D$ and is invertible modulo $q$, so translation still completes the square. Its unshifted Gauss sum is $2^j\chi_{-D}(2)^j$: separating even and odd variables gives the recurrence $G_j=4G_{j-2}$, with $G_1=2(-1)^{ac}$ and $G_2=4$, and $(-1)^{ac}=\chi_{-D}(2)$. Finally replace $t$ by its inverse in the remaining character sum. It is
$$
 \frac{q\chi_{-D}(q)}{\tau(\bar\chi)}
 \bar\chi(-D^{-1}Q(-v,u))\tau(\chi)
 =\chi_{-D}(q)\chi(D)\tau(\chi)^2\bar\chi(Q(-v,u)),
$$
using $\tau(\chi)\tau(\bar\chi)=\chi(-1)q$. This proves the asserted normalization and phase.
\end{proof}

For general order, the symbol $K_\nu(z)$ denotes the modified Bessel function of the second kind with order $\nu$. One integral definition, valid for $\nu\in\C$ and $\Re z>0$, is \cite[\S10.32]{DLMF}
\begin{equation}\label{eq:Kdefinition}
 K_\nu(z)=\frac12\int_0^\infty u^{\nu-1}
       \exp\left(-\frac z2(u+u^{-1})\right)\,du.
\end{equation}
Observe that the order-zero transforms are \eqref{ex:radial-pairs} and \eqref{eq:ker-integrable-transform}; for a general quadratic form, only the lattice scaling and finite character transform change, and the partial fractions explain why these Kelvin functions occur \cite{Koshliakov36,KT15,BK25}.

\begin{proposition}[The two field resolvents]\label{prop:twch}
Under the hypotheses of Lemma \ref{prop:eig}, for $x>0$,
\begin{align}\label{eq:Bessel-resolvents}
 \sum_{n\ge1}\frac{\chi(n)F(n)}{n^2+x^2}
 &=-\frac{4\pi c_\chi}{q^2\sqrt D\,x}\sum_{n\ge1}\bar\chi(n)F(n)
 \kei\left(\frac{4\pi}q\sqrt{\frac{xn}{D}}\right),\notag\\
 x^2\sum_{n\ge1}\frac{\chi(n)F(n)}{n(n^2+x^2)}
 &=L_K(1,\chi)-\frac{4\pi c_\chi}{q^2\sqrt D}\sum_{n\ge1}\bar\chi(n)F(n)
 \ker\left(\frac{4\pi}q\sqrt{\frac{xn}{D}}\right).
\end{align}
The value $L_K(1,\chi)$ is the natural-order convergent Dirichlet series.
\end{proposition}
\begin{proof}
Write $Q(\uu)=|A\uu|^2$, so $|\det A|=\sqrt D/2$ and
$|A^{-T}\mm|^2=4Q(-m_2,m_1)/D$. For each residue $\aaa\bmod q$, apply Theorem \ref{thm:polynomial-T1} in its constant-polynomial specialization $P=1$, $c(t)=t$, $b(j)=\mathbf1_{j=1}$, $h(j)=j$, $\omega(j)=1$, with $g(\xx)=f(\aaa+q\xx)$, first for $f(\uu)=(Q(\uu)^2+x^2)^{-1}$ and then for $f(\uu)=Q(\uu)/(Q(\uu)^2+x^2)^2$. Multiplying by $\chi(Q(\aaa))$ and summing the finitely many residues gives explicitly
$$
 \sum_{\vv\in\Z^2}\chi(Q(\vv))f(\vv)
 =q^{-2}\sum_{\mm\in\Z^2}\mathcal F_q(\chi\circ Q)(\mm)\widehat f(\mm/q).
$$
The change of variables $\yy=A\uu$ supplies the factor $2/\sqrt D$; Lemma \ref{prop:eig} supplies $c_\chi\bar\chi(Q(-m_2,m_1))$. The weighted zero frequencies vanish. Sum over all classes, divide by $w_K$, and use rotation invariance of the representation numbers. With \eqref{ex:radial-pairs}, this proves the first resolvent. With \eqref{eq:ker-integrable-transform}, it gives
$$
 \sum_{n\ge1}\frac{\chi(n)F(n)n}{(n^2+x^2)^2}
 =-\frac{2\pi c_\chi}{q^2\sqrt D\,x}\frac{d}{dx}
 \sum_{n\ge1}\bar\chi(n)F(n)
 \ker\left(\frac{4\pi}q\sqrt{\frac{xn}{D}}\right).
$$
Complete residue blocks cancel by \eqref{eq:norm-finite-eigen}; the boundary of $Q(\uu)\le u$ meets $O_{Q,q}(\sqrt u)$ blocks. Consequently $\sum_{n\le u}\chi(n)F(n)=O(\sqrt u)$. The summation-by-parts estimate used after \eqref{ex:gaussian-resolvents} now applies with $A_G(n)$ replaced by $\chi(n)F(n)$. In the last identity, replace $x$ by $t$, multiply by $2t$, and integrate over $x\le t<\infty$ to obtain
$$
 \sum_{n\ge1}\frac{\chi(n)F(n)n}{n^2+x^2}
 =\frac{4\pi c_\chi}{q^2\sqrt D}\sum_{n\ge1}\bar\chi(n)F(n)
 \ker\left(\frac{4\pi}q\sqrt{\frac{xn}{D}}\right).
$$
Subtracting this from the naturally ordered series for $L_K(1,\chi)$ proves the second resolvent. Its $L_K(1,\chi)$ term comes from $1/n-n/(n^2+x^2)=x^2/[n(n^2+x^2)]$.
\end{proof}

For an integer $t$ put $\sigma^{\bar\chi}_{K,t}(M)=\sum_{d\mid M}\bar\chi(d)F(d)d^t$. Empty sums in the next theorem are zero.
\begin{theorem}[Twisted Lambert--Bessel reciprocity]\label{thm:tw}
Let $q,\chi$ be as in Lemma \ref{prop:eig}. For integers $k,m\ge1$ and $a>0$,
\begin{align}\label{eq:Bessel-odd-law}
 L_K(2k+1,\chi)+2\sum_{n\ge1}\frac{\chi(n)F(n)n^{-2k-1}}{\exp(2\pi an)-1}
 ={}&\frac{L_K(2k+2,\chi)}{\pi a}
 +\frac{2a}\pi\sum_{j=0}^{k-1}(-a^2)^j\zeta(2j+2)L_K(2k-2j,\chi)\notag\\
 &-\frac{8c_\chi(-a^2)^k}{q^2\sqrt D}
 \sum_{M\ge1}\frac{\sigma^{\bar\chi}_{K,2k+1}(M)}{M^{2k+1}}
 \kei\left(\frac{4\pi}q\sqrt{\frac M{aD}}\right),
\end{align}
\begin{align}\label{eq:Bessel-even-law}
 L_K(2m,\chi)+2\sum_{n\ge1}\frac{\chi(n)F(n)n^{-2m}}{\exp(2\pi an)-1}
 ={}&\frac{L_K(2m+1,\chi)}{\pi a}
 +\frac{2a}\pi\sum_{j=0}^{m-2}(-a^2)^j\zeta(2j+2)L_K(2m-2j-1,\chi)\notag\\
 &+\frac{2a}\pi(-a^2)^{m-1}\zeta(2m)L_K(1,\chi)\notag\\
 &-\frac{8ac_\chi(-a^2)^{m-1}}{q^2\sqrt D}
 \sum_{M\ge1}\frac{\sigma^{\bar\chi}_{K,2m}(M)}{M^{2m}}
 \ker\left(\frac{4\pi}q\sqrt{\frac M{aD}}\right).
\end{align}
\end{theorem}
\begin{proof}
For an absolutely summable sequence $b(n)$, the sinc-corrected APSF calculation used above gives
\begin{equation}\label{eq:Lambert-rational}
 \sum_nb(n)+2\sum_n\frac{b(n)}{\exp(2\pi an)-1}
 =\frac a\pi\sum_{n\ge1}\sum_{r\in\Z}\frac{nb(n)}{a^2n^2+r^2}.
\end{equation}
Take $b(n)=\chi(n)F(n)n^{-2k-1}$ or $\chi(n)F(n)n^{-2m}$. The zero frequency gives the first term of each formula. For positive $r$, divide as follows:
$$
 \frac{n^{-2k}}{a^2n^2+r^2}
 =\sum_{j=0}^{k-1}\frac{(-a^2)^jn^{-2k+2j}}{r^{2j+2}}
 +\frac{(-a^2)^k}{r^{2k}}\frac1{a^2n^2+r^2},
$$
$$
 \frac{n^{1-2m}}{a^2n^2+r^2}
 =\sum_{j=0}^{m-2}\frac{(-a^2)^jn^{1-2m+2j}}{r^{2j+2}}
 +\frac{(-a^2)^{m-1}}{r^{2m-2}}\frac{n^{-1}}{a^2n^2+r^2}.
$$
The finite terms produce the stated zeta products. Apply the first resolvent with $x=r/a$ to the first remainder, and the second resolvent to the second remainder. Their coefficients become, respectively, $r^{-2k-1}$ and $r^{-2m}$ times the displayed Bessel series. Set $M=rn$: the inner divisor coefficients are exactly $\sigma^{\bar\chi}_{K,2k+1}(M)/M^{2k+1}$ and $\sigma^{\bar\chi}_{K,2m}(M)/M^{2m}$. All original remainders and all Bessel double sums converge absolutely; the isolated $L_K(1,\chi)$ is interpreted by the proved summation-by-parts limit. This proves both laws for every class number, without assuming a functional equation.
\end{proof}

\begin{remark}
The untwisted laws are those of Banerjee--Kumar \cite{BK25}. Their zero-frequency and origin terms must be retained. For example the first unweighted resolvent is
$$
 \sum_{n\ge1}\frac{F(n)}{n^2+x^2}
 =\frac{\pi^2h_K}{w_K\sqrt D\,x}-\frac{h_K}{w_Kx^2}
 -\frac{4\pi}{\sqrt D\,x}\sum_{n\ge1}F(n)\kei\left(4\pi\sqrt{\frac{xn}{D}}\right).
$$
The first two terms are respectively the Fourier zero term and the subtracted lattice origin. Related functional-equation formulations are discussed in \cite{KT15,FrankeHigher,BansalMaji26}; our derivation identifies the finite transform and the two Fourier kernels that produce the twisted formulas.
\end{remark}

\begin{remark}\label{rem:rem13}
Our first approach used the modified Bessel function $K_\nu$ and led to a reciprocity with
arbitrary complex order $\nu$. The identity displayed below is an explicit
Gaussian-character specialization of the general Ramanujan--Guinand relation of Ribeiro
and Yakubovich \cite[Eq. (5.7), preprint version]{RibeiroYakubovich24}. We give an APSF
derivation that records the conductor, Gauss-sum factor, and the two origin corrections
directly. The Kelvin-function formulas above emerged later through exchanges with LLMs. We chose them as the main application because their relationship to earlier results could be documented more precisely. The broader Bessel family has many specializations, and we do not claim to have settled priority for all of them; the references below describe several classical instances. In what follows, $\nu$ is the order of the Bessel function, $t>0$ scales its argument, and $n,m$ are positive integers indexing norm values. Write $r_2(n)=\#\{(a,b)\in\mathbb Z^2:a^2+b^2=n\}$ and, for any sequence $u(n)$ of polynomial growth, put
$$
 \mathcal A_u(t,\nu)=\sum_{n,m\ge1}u(n)r_2(m)(m/n)^{\nu/2}K_\nu(2\pi t\sqrt{nm}),\qquad
 \mathcal B_u(t,\nu)=\sum_{n,m\ge1}r_2(n)u(m)(m/n)^{\nu/2}K_\nu(2\pi t\sqrt{nm}).
$$
For $t>0$ these sums are entire in $\nu$, by the exponential decay of $K_\nu$
uniformly on compact sets of orders. If $\chi$ is a nonprincipal character
modulo an odd prime $p$, set
$Z_\chi(s)=L(s,\chi)L(s,\chi\chi_{-4})$,
$u_\chi(n)=\chi(n)r_2(n)$, and
$c_{\bar\chi}=\chi_{-4}(p)\bar\chi(4)\tau(\bar\chi)^2$.
Then, for every $\nu\in\mathbb C$,
\begin{equation}\label{sec4:secondary-guinand}
 \mathcal A_{u_\chi}(\alpha,\nu)
 =\frac{2\Gamma(\nu+1)Z_\chi(\nu+1)}{\pi^{\nu+1}\alpha^{\nu+2}}
 -\frac{2\Gamma(\nu)Z_\chi(\nu)}{\pi^\nu\alpha^\nu}
 +\frac{\mathcal B_{u_{\bar\chi}}(1/(p\alpha),\nu)}
 {c_{\bar\chi}p^\nu\alpha^2}.
\end{equation}
All apparent singularities in the gamma products are interpreted by continuity.
To see the constants directly, start with real $\nu>1$ and the planar Fourier pair\footnote{\label{fn:matern}For positive real $\nu$ and $t$, the radial profile $h_t$ is a Mat\'ern kernel up to normalization: $h_t(\xx)/h_t(0)=2^{1-\nu}(2\pi t|\xx|)^\nu K_\nu(2\pi t|\xx|)/\Gamma(\nu)$. This is the standard Mat\'ern covariance function with length scale $\sqrt{2\nu}/(2\pi t)$; see \cite[Eq. (4.14)]{RasmussenWilliams06}.}
$$
 h_t(\xx)=|\xx|^\nu K_\nu(2\pi t|\xx|),\qquad
 \widehat h_t(\xxi)=\frac{\Gamma(\nu+1)t^\nu}
 {2\pi^{\nu+1}(t^2+|\xxi|^2)^{\nu+1}},\qquad
 h_t(0)=\frac{\Gamma(\nu)}{2\pi^\nu t^\nu}.
$$
It follows from the integral representation of $K_\nu$ and the Gaussian
transform \cite[\S10.32]{DLMF}. With
$R_\chi=\sum_{n,m\ge1}u_\chi(n)r_2(m)(\alpha^2n+m)^{-\nu-1}$,
apply Theorem \ref{thm:polynomial-T1} with $P=1$, $c(u)=u$, $b(j)=\mathbf1_{j=1}$, $h(j)=j$, $\omega(j)=1$, and $g(\yy)=h_{\alpha\sqrt n}(\yy)$ for each $n\ge1$. Subtract the lattice origin, multiply by $u_\chi(n)n^{-\nu/2}$, and sum over $n$. The Fourier zero frequency and the subtracted origin give the two displayed gamma terms; the remaining contribution is
$\Gamma(\nu+1)\alpha^\nu R_\chi/(2\pi^{\nu+1})$.
For the second evaluation, keep $P,c,b,h,\omega$ as above and take $g(\yy)=h_{\sqrt n/(p\alpha)}(\aaa+p\yy)$ for each residue vector $\aaa\bmod p$. Multiply by $\bar\chi(|\aaa|^2)$ and sum over the residues, as in the proof of Proposition \ref{prop:twch}. Multiplying the resulting identity by $r_2(n)n^{-\nu/2}$ and summing over $n\ge1$ gives
$\mathcal B_{u_{\bar\chi}}(1/(p\alpha),\nu)
=c_{\bar\chi}\Gamma(\nu+1)p^\nu\alpha^{\nu+2}R_\chi/(2\pi^{\nu+1})$.
Elimination proves \eqref{sec4:secondary-guinand} for every real $\nu>1$.
The Gaussian sum with weight $\chi(a^2+b^2)$ decays exponentially both at
infinity and, after the same residue-class application of Theorem \ref{thm:polynomial-T1}, at zero: its finite Fourier coefficient at zero
vanishes. Its Mellin integral therefore shows that $\Gamma(s)Z_\chi(s)$
is entire. The identity theorem now proves the formula for all $\nu$. This twisted Ramanujan--Guinand law is obtained from \cite[Eq. (5.7), preprint version]{RibeiroYakubovich24} by taking $r_1=r_2=1$, $s=\nu+1$, $x=\alpha$,
$$
\phi_1(s)=\psi_1(s)=4\pi^{-s}\zeta(s)L(s,\chi_{-4}),\qquad
\phi_2(s)=4\pi^{-s}Z_\chi(s),\qquad
\psi_2(s)=\frac{4c_\chi}{p^2}
\left(\frac\pi{p^2}\right)^{-s}Z_{\bar\chi}(s),
$$
where $c_\chi=\chi_{-4}(p)\chi(4)\tau(\chi)^2$. In these data,
$\phi_1(0)=-1$, $\operatorname{Res}_{s=1}\phi_1(s)=1$, and
$\phi_2(0)=\operatorname{Res}_{s=1}\phi_2(s)=0$; also $c_\chi c_{\bar\chi}=p^2$.
The general relationship between Bessel sums and Hecke functional equations is treated by
Chandrasekharan--Narasimhan \cite{CN61} and Berndt--Dixit--Gupta--Zaharescu
\cite{BDGZ22}.
\end{remark}
\begin{remark}[Other examples and extensions]\label{sec:general-field-reciprocity}
The elliptic curve $E:y^2=x^3-x$ has a complex-multiplication $L$-function described by a character of the Gaussian integers \cite[Chapter II]{SilvermanAdvanced}. We obtained a Bessel reciprocity example for this $L$-function through interactions with LLMs and checked its written calculations and numerical specializations. Those interactions also led us to a reciprocity between two periodically weighted Euclidean ideal lattices, including lattices from fields of unequal degrees. We checked the written theta and Mellin calculations and numerical examples, and defer a deeper study to another occasion. This concerns embedded ideal elements: in higher degree, passing to ideal norms requires averaging over the infinite unit group, as in the classical theory \cite{Neukirch99}. The constant-polynomial specialization of Theorem \ref{thm:polynomial-T1}, after a lattice change of variables, supplies the underlying Poisson identities; the full argument is reserved for that later study.
\end{remark}

\section{\texorpdfstring{$L$}{L}-functions, selectors, and Epstein series}\label{sec:L-functions}
The scalar kernels lead to functional equations, whereas an infinite power-free sieve first gives an error estimate and a half-plane of continuation. The final subsection relates vector divisibility to the Bessel expansions of Epstein zeta functions.
\subsection{The Riemann zeta function}
Use the data of Ramanujan's formula: $b(d)=d^{-2m-1}$, $h(d)=\omega(d)=d$, $g_t(x)=e^{-2\pi t|x|}$, and $c(x)=1-\sinc x$, with $m\ge1$. Write
\begin{equation}\label{eq:zeta-dual-kernel}
 A_m(t)=\sum_{n\ge1}\frac{n^{-2m-1}}{e^{2\pi tn}-1},\qquad
 D_m(t)=\frac{2t}{\pi}\sum_{d,r\ge1}\frac{d^{-2m}}{t^2d^2+r^2}.
\end{equation}
The scalar APSF and the partial-fraction expansion of $\coth$ give
\begin{equation}\label{eq:zeta-lorentz}
 2A_m(t)+\zeta(2m+1)=\frac{\zeta(2m+2)}{\pi t}+D_m(t).
\end{equation}
Direct Mellin integration in the respective convergence strips gives
\begin{align}
 2\int_0^\infty A_m(t)t^{s-1}dt
 &=\frac{2\Gamma(s)}{(2\pi)^s}\zeta(s)\zeta(s+2m+1),&&\Re s>1,\label{eq:zeta-mellin-left}\\
 \int_0^\infty D_m(t)t^{s-1}dt
 &=\sec(\pi s/2)\zeta(1-s)\zeta(s+2m+1),&&-1<\Re s<0.\label{eq:zeta-mellin-right}
\end{align}
The second calculation uses $\int_0^\infty t^s/(d^2t^2+r^2)\,dt=(\pi/2)d^{-s-1}r^{s-1}\sec(\pi s/2)$; absolute integration permits summation in that strip. To compare the two expressions, put $a=\zeta(2m+2)/\pi$, $b=\zeta(2m+1)$ and split at $1$:
$$
 M(s)=\int_0^1(2A_m(t)-a/t+b)t^{s-1}dt
 +\int_1^\infty2A_m(t)t^{s-1}dt+\frac{a}{s-1}-\frac{b}{s}.
$$
The first integral converges for $\Re s>-1$, since $D_m(t)=O(t)$ at zero; the second is entire. In $-1<\Re s<0$, integrating $-a/t+b$ on $(1,\infty)$ shows that $M(s)$ equals the Mellin transform of $D_m$. Cancelling the common factor therefore proves, initially in this strip and then meromorphically,
\begin{equation}\label{eq:zeta-fe}
 \frac{2\Gamma(s)}{(2\pi)^s}\zeta(s)=\sec(\pi s/2)\zeta(1-s).
\end{equation}
\begin{remark}[The Dirichlet beta function]\label{rem:beta-comment}
For $\beta(s)=L(s,\chi_{-4})$, take $b(d)=d^{-\lambda}$ with $\lambda>1$, $h=\omega=\mathrm{id}$, the same $g_t$, and $c(x)=\sin(\pi|x|/2)$. The arithmetic side of this APSF specialization can be written as
\begin{equation}\label{eq:beta-lhs}
 B_\lambda(t)=2\sum_{n\ge1}e^{-2\pi tn}\sum_{d\mid n}d^{-\lambda}\chi_{-4}(n/d)
 =\sum_{d\ge1}d^{-\lambda}\operatorname{sech}(2\pi td).
\end{equation}
For $C_\lambda(x)=\int_0^\infty B_\lambda(t)\cos(2\pi xt)dt$, use the cosine transform $\frac14\operatorname{sech}(\pi x/2)$ of $\operatorname{sech}(2\pi t)$ \cite{Titchmarsh48}. In $0<\Re s<1$, its Mellin integral is both
$$
 \frac{\Gamma(s)\Gamma(1-s)}\pi\cos(\pi s/2)\zeta(\lambda+1-s)\beta(1-s)
 \quad\text{and}\quad
 2^{s-1}\pi^{-s}\Gamma(s)\zeta(\lambda+1-s)\beta(s).
$$
Pair terms $4j+1$ and $4j+3$ before integrating; their rational-kernel integrals are $O_s((j+1)^{\Re s-2})$. This justifies the paired summation. Cancellation gives
\begin{equation}\label{eq:beta-fe}
 \beta(s)=2^{1-s}\pi^{s-1}\Gamma(1-s)\cos(\pi s/2)\beta(1-s),
\end{equation}
equivalent to the symmetry of $(\pi/4)^{-(s+1)/2}\Gamma((s+1)/2)\beta(s)$ under $s\mapsto1-s$; compare \cite[\S25.15(i)]{DLMF}.
\end{remark}
\subsection{The Estermann functional equation}
The divisor function leads to a useful test of the scalar formula: can its rational Fourier kernels recover the functional equation of an additively twisted Dirichlet series? Let $q\ge2$, $(p,q)=1$, and write
$$
 D(s,p/q)=\sum_{n\ge1}\frac{\tau(n)\ee{pn/q}}{n^s},\qquad \Re s>1.
$$
This is Estermann's function \cite{Estermann30}. The transformation below exchanges $p$ with its inverse modulo $q$. We retain the analytic step that makes this exchange legitimate: a divergent resonant subseries must be regularized before taking a Mellin transform.

\begin{theorem}[Estermann functional equation from scalar APSF]
Let $p\bar p\equiv1\pmod q$. Meromorphically for every $s$,
\begin{equation}\label{eq:Estermann-FE}
 \begin{aligned}
 D(s,p/q)={}&2q^{1-2s}(2\pi)^{2s-2}\Gamma(1-s)^2\\
 &\times\{D(1-s,\bar p/q)-\cos(\pi s)D(1-s,-\bar p/q)\}.
 \end{aligned}
\end{equation}
Its only pole is the double pole at $s=1$, with principal part
\begin{equation}\label{eq:Estermann-poles}
 \frac1{q(s-1)^2}+\frac{2\gamma-2\log q}{q(s-1)}.
\end{equation}
Here $\gamma$ denotes Euler's constant. The normalization agrees with \cite[Lemma 6, equation (3-1), with both shifts zero]{Bettin19}.
\end{theorem}
\begin{proof}
Write $x_+=\max(x,0)$. For $t>0$, apply Theorem \ref{thm:T1} with
\begin{equation}\label{eq:T1data}
 h(d)=\omega(d)=d,\quad b(d)=d,\quad c(x)=x_+,\quad
 g_t(x)=\ee{px/q}e^{-2\pi t|x|}.
\end{equation}
The Fourier integral is $\int_0^\infty y e^{-2\pi[td-i(r+pd/q)]y}dy$, hence
\begin{equation}\label{eq:T1Est}
 A_{p,q}(t):=\sum_{n\ge1}n\tau(n)\ee{pn/q}e^{-2\pi tn}
 =\frac1{(2\pi)^2}\sum_{d\ge1}d
 \left(\sum_{r\in\Z}[td-i(r+pd/q)]^{-2}\right).
\end{equation}
Each inner sum converges absolutely and the completed $d$-blocks have an exponentially summable arithmetic majorant. But the terms $q\mid d$, $r=-pd/q$ alone would give $(2\pi t)^{-2}\sum_{q\mid d}d^{-1}$. They cannot be separated in this form.

Replace $b(d)$ by $d^{1-w}$, initially with $\Re w>2$, and denote the resulting arithmetic side by
$$
 A_w(t)=\sum_{n\ge1}n\sigma_{-w}(n)\ee{pn/q}e^{-2\pi tn}.
$$
Now the zero frequency is $a_wt^{-2}$, where $a_w=q^{-1-w}\zeta(1+w)/(2\pi)^2$. The remaining frequencies are jointly absolutely summable and bounded as $t\downarrow0$: their absolute sum is $O_{q,w}(\sum_d d^{1-\Re w})$. Put
$$
 D_w(s,p/q)=\sum_{n\ge1}\sigma_{-w}(n)\ee{pn/q}n^{-s}.
$$
Its Mellin integral, initially for $\Re s>1$, continues to $\Re s>-1$ as
\begin{align}\label{eq:reg-Mellin}
 \frac{\Gamma(s+1)}{(2\pi)^{s+1}}D_w(s,p/q)
 ={}&\int_0^1(A_w(t)-a_wt^{-2})t^sdt\notag\\
 &+\int_1^\infty A_w(t)t^sdt+\frac{a_w}{s-1}.
\end{align}
In $-1<\Re s<1$ this equals $\int_0^\infty(A_w(t)-a_wt^{-2})t^sdt$. Thus the pole has been retained, not discarded with the zero frequency.

Using principal complex powers, the beta integral gives
\begin{equation}\label{eq:beta-kernel}
 \int_0^\infty\frac{t^s}{(td-iv)^2}\,dt
 =d^{-s-1}\Gamma(s+1)\Gamma(1-s)(-iv)^{s-1},\qquad v\in\R\setminus\{0\}.
\end{equation}
For $-1<\sigma=\Re s<0$ the interchange with the nonzero frequencies follows from
\begin{equation}\label{eq:Mellin-majorant}
 \sum_{d\ge1}d^{1-\Re w}\sum_{r+pd/q\ne0}
 \int_0^\infty\frac{t^\sigma dt}{(td)^2+(r+pd/q)^2}
 \ll_{q,\sigma}\sum_{d\ge1}d^{-\Re w-\sigma}<\infty.
\end{equation}
Writing $G(s)=(2\pi)^{s-1}\Gamma(1-s)$, cancellation in \eqref{eq:reg-Mellin} therefore gives
\begin{equation}\label{eq:D-reg-dual}
 D_w(s,p/q)=G(s)\sum_{d\ge1}d^{-s-w}
 \sum_{r+pd/q\ne0}[-i(r+pd/q)]^{s-1}.
\end{equation}

To remove $w$, group $d$ modulo $q$. The outer sum in each class is $q^{-s-w}\zeta(s+w,a/q)$, where $\zeta(s,x)$ is the Hurwitz zeta function. The finite identity
\begin{equation}\label{eq:Hurwitz-product}
 D_w(s,p/q)=q^{-2s-w}\sum_{a,b=1}^q\ee{pab/q}
                  \zeta(s+w,a/q)\zeta(s,b/q)
\end{equation}
follows from the original divisor series and supplies joint meromorphic continuation. We may consequently set $w=0$ in the finite residue-class expression, obtaining
\begin{equation}\label{eq:one-reflection}
 D(s,p/q)=G(s)q^{-s}\sum_{a=1}^q\zeta(s,a/q)
       \sum_{r+pa/q\ne0}[-i(r+pa/q)]^{s-1}.
\end{equation}

The classical Hurwitz Fourier formula is
\begin{equation}\label{eq:Hurwitz-Fourier}
 \zeta(s,x)=G(s)\sum_{m\ne0}(im)^{s-1}\ee{mx},\qquad
 \Re s<0,\quad 0<x\le1;
\end{equation}
see \cite[Theorem 12.6]{Apostol76} or \cite[Eq. 25.11.9]{DLMF}. Its series is absolutely convergent, including the periodic endpoint. Insert it in \eqref{eq:one-reflection} and put $\ell=qr+pa$. The congruence $\ell\equiv pa\pmod q$ selects exactly $a\equiv\bar p\ell\pmod q$, so
\begin{equation}\label{eq:signed-double}
 D(s,p/q)=G(s)^2q^{1-2s}\sum_{m,\ell\ne0}
       (im)^{s-1}(-i\ell)^{s-1}\ee{\bar p m\ell/q}.
\end{equation}
For equal signs of $m,\ell$, the phase factors multiply to $1$; the two quadrants give $2D(1-s,\bar p/q)$. For opposite signs they give $-e^{i\pi s}$ and $-e^{-i\pi s}$, yielding $-2\cos(\pi s)D(1-s,-\bar p/q)$. This proves \eqref{eq:Estermann-FE} in $-1<\Re s<0$, and \eqref{eq:Hurwitz-product} continues it meromorphically.

Finally, each Hurwitz factor has only a simple pole at $1$. In \eqref{eq:Hurwitz-product} with $w=0$, orthogonality gives $\sum_{a,b}\ee{pab/q}=q$. The constant term of either Hurwitz factor contributes $q\gamma$ after summing the other residue, and differentiating $q^{-2s}$ contributes $-2\log q$. This proves \eqref{eq:Estermann-poles} and excludes other poles.
\end{proof}
The case $q=1$ is the same calculation with the trivial residue class, giving the functional equation of $D(s,0)=\zeta(s)^2$.

\subsection{Squarefree values and the comparison with Tolev}\label{sec:kfree}\label{sec:modular-bridge}
Which lattice points have a squarefree value of $x^2+y^2+1$? Tolev studies this question in a square \cite{Tolev12}. We first weight the same points by a Gaussian. Grouping them by $N=x^2+y^2$ gives $r_2(N)\mu^2(N+1)$: the number of representations is retained precisely when the shifted norm is squarefree. The associated Dirichlet series is therefore
\begin{equation}\label{eq:sf-series-definition}
 Z_{\rm sf}(s)=\sum_{N\ge1}\frac{r_2(N)\mu^2(N+1)}{N^s},\qquad \Re s>1.
\end{equation}
The Gaussian count and this series contain the same coefficients; their Mellin relation will turn a counting error into a region of continuation.

For $P(x,y)=x^2+y^2+1$, use $\mu^2(N)=\sum_{d^2\mid N}\mu(d)$ and put
\begin{align}
 \Theta_{\rm sf}(t)&=\sum_{\vv\in\Z^2}\mu^2(P(\vv))e^{-\pi t|\vv|^2}
 =1+\sum_{N\ge1}r_2(N)\mu^2(N+1)e^{-\pi tN}\label{eq:theta-sf}\\
 &=\frac1t\sum_{d\ge1}\frac{\mu(d)}{d^4}
 \left(\sum_{\mm\in\Z^2}\mathcal G_{P,d^2}(\mm)e^{-\pi|\mm|^2/(td^4)}\right).\label{eq:sf-dual}
\end{align}
The Fourier expression follows from Theorem \ref{thm:polynomial-T1} with $b=\mu$, $h(d)=d^2$, $\omega=1$, $c(u)=u$, and $g(\xx)=e^{-\pi t|\xx|^2}/P(\xx)$. The product $g(\xx)c(P(\xx))$ is Gaussian; the divisor bound $\tau(P(\vv))\ll_\epsilon P(\vv)^\epsilon$ justifies the completed-modulus order.

\begin{proposition}\label{b-prop:zero-mode}
The zero frequency contributes $C/t$, where
\begin{equation}\label{eq:Csf}
 C=\sum_{d\ge1}\frac{\mu(d)\mathcal G_{P,d^2}(0)}{d^4}
 =\prod_{p\ {\rm odd}}\left(1-p^{-2}+\chi_{-4}(p)p^{-3}\right)>0.
\end{equation}
\end{proposition}
\begin{proof}
\textbf{1.} At frequency $\mm=0$, every exponential in the definition of $\mathcal G_{P,q}$ equals one. Thus $\mathcal G_{P,q}(0)$ is the number of pairs $(x,y)$ modulo $q$ satisfying $x^2+y^2+1\equiv0\pmod q$. The zero-frequency part of \eqref{eq:sf-dual} is therefore
$$
 \frac1t\sum_{d\ge1}\frac{\mu(d)\mathcal G_{P,d^2}(0)}{d^4}.
$$
We compute these counts at prime squares and then justify the Euler product.

\noindent\textbf{2.} Let $p$ be an odd prime, and write $\left(\frac{a}{p}\right)$ for the Legendre symbol, with value zero when $p\mid a$. For each $x$ modulo $p$, the number of solutions of $y^2\equiv-1-x^2\pmod p$ is $1+\left(\frac{-1-x^2}{p}\right)$. Consequently
$$
 \mathcal G_{P,p}(0)
 =p+\left(\frac{-1}{p}\right)
       \sum_{x\bmod p}\left(\frac{x^2+1}{p}\right).
$$
To evaluate the remaining sum, let us count the pairs $y^2\equiv x^2+1\pmod p$. Counting first in $y$ gives
$$
 \#\{(x,y)\bmod p:y^2=x^2+1\}
 =p+\sum_{x\bmod p}\left(\frac{x^2+1}{p}\right).
$$
On the other hand, this equation is $(y-x)(y+x)=1$. Each nonzero residue $u$ gives exactly one pair, determined by $y-x=u$ and $y+x=u^{-1}$, because $2$ is invertible modulo $p$. There are therefore $p-1$ pairs, and the character sum is $-1$. Since $\left(\frac{-1}{p}\right)=\chi_{-4}(p)$, we obtain
$$
 \mathcal G_{P,p}(0)=p-\chi_{-4}(p).
$$

\noindent\textbf{3.} Now, we lift each solution to modulus $p^2$. Choose integer representatives $x_0,y_0$ of a solution modulo $p$ and observe that its lifts are
$$
 x=x_0+pa,\qquad y=y_0+pb,\qquad a,b\bmod p.
$$
Expanding $x^2+y^2+1$ shows that a lift is a solution modulo $p^2$ precisely when
$$
 2x_0a+2y_0b\equiv-\frac{x_0^2+y_0^2+1}{p}\pmod p.
$$
The quotient on the right is an integer and since $x_0$ and $y_0$ cannot both vanish modulo $p$ (since their squared sum is $-1$), at least one coefficient on the left is therefore nonzero modulo $p$. We may choose the other variable freely, and the equation determines the remaining variable uniquely. Hence each solution has exactly $p$ lifts, giving
$$
 \mathcal G_{P,p^2}(0)=p\bigl(p-\chi_{-4}(p)\bigr),\qquad
 \frac{\mathcal G_{P,p^2}(0)}{p^4}
 =\frac1{p^2}-\frac{\chi_{-4}(p)}{p^3}.
$$
For $p=2$, a square is either $0$ or $1$ modulo $4$, so $x^2+y^2+1$ can only be $1$, $2$, or $3$ modulo $4$. Thus $\mathcal G_{P,4}(0)=0$, and the factor at $2$ is $1$.\\

\noindent\textbf{4.} Finally, for coprime positive integers $a,b$, the Chinese remainder theorem identifies a solution modulo $a^2b^2$ with a pair of solutions modulo $a^2$ and $b^2$. Therefore
$$
 \mathcal G_{P,(ab)^2}(0)
 =\mathcal G_{P,a^2}(0)\mathcal G_{P,b^2}(0).
$$
Only squarefree $d$ occur in the sum defining $C$, because $\mu(d)=0$ otherwise. To justify its absolute convergence, let $R\ge2$ be a real cutoff and expand the finite products
$$
 \prod_{\substack{p\le R\\p\ \mathrm{odd}}}
 \left(1+\frac{\mathcal G_{P,p^2}(0)}{p^4}\right)
 \le \exp\left(\sum_{p\ \mathrm{odd}}(p^{-2}+p^{-3})\right)<\infty.
$$
The expansion on the left is the sum of $|\mu(d)|\mathcal G_{P,d^2}(0)/d^4$ over squarefree $d$ with prime factors at most $R$; even $d$ contribute zero. Letting $R\to\infty$ proves absolute convergence. We can consequently insert the signs $\mu(d)$ and pass to the limit in the finite products, obtaining
$$
 C=\prod_{p\ \mathrm{odd}}
 \left(1-\frac{\mathcal G_{P,p^2}(0)}{p^4}\right)
 =\prod_{p\ \mathrm{odd}}
 \left(1-p^{-2}+\chi_{-4}(p)p^{-3}\right).
$$
Finally, every factor is at least $1-p^{-2}-p^{-3}\ge23/27>0$. The sum of their differences from $1$ converges absolutely, so the sum of their logarithms converges to a finite real number. The product is therefore strictly positive. This proves both the formula for $C$ and the asserted zero-frequency contribution $C/t$.
\end{proof}

\begin{theorem}\label{thm:sf-sharp}
For every $\epsilon>0$,
\begin{equation}\label{eq:sf-sharp-Gaussian}
 \Theta_{\rm sf}(t)=C/t+O_\epsilon(t^{-2/3-\epsilon})\qquad(0<t\le1).
\end{equation}
The series $Z_{\rm sf}$ in \eqref{eq:sf-series-definition} continues meromorphically to $\Re s>2/3$, with only a simple pole at $1$, of residue $\pi C$.
\end{theorem}
\begin{proof}
\textbf{1.} We start with the sum on the arithmetic side:
$$
 \Theta_{\rm sf}(t)
 =\sum_{\vv\in\Z^2}\mu^2(|\vv|^2+1)e^{-\pi t|\vv|^2}
 =1+\sum_{N\ge1}r_2(N)\mu^2(N+1)e^{-\pi tN}.
$$
The term $1$ comes from $\vv=(0,0)$, since $\mu^2(1)=1$. Since this term does not decay as $t\to\infty$, we must remove it before taking the Mellin transform that produces $Z_{\rm sf}$. Observe also that there is a second subtraction that we need to make near $t=0$: the Fourier zero frequency contributes $C/t$. So, our first step should be to estimate $\Theta_{\rm sf}(t)-C/t$ and then explain how these two subtractions give the continuation formula.\\

\noindent\textbf{2.} For each positive integer $d$, let
$$
 T_d(t)=\sum_{\substack{\vv\in\Z^2\\d^2\mid |\vv|^2+1}}
             e^{-\pi t|\vv|^2}.
$$
Thus $T_d(t)$ counts, with the same Gaussian weight, the points for which $d^2$ divides $|\vv|^2+1$. Expanding $\mu^2(|\vv|^2+1)$ gives
$$
 \Theta_{\rm sf}(t)=\sum_{d\ge1}\mu(d)T_d(t).
$$
This rearrangement is absolutely convergent for each $t>0$: the sum with $|\mu(d)|$ is bounded by
$$
 \sum_{\vv\in\Z^2}\tau(|\vv|^2+1)e^{-\pi t|\vv|^2}<\infty,
$$
where $\tau$ is the divisor-counting function.\\

\noindent\textbf{3.} For fixed $d$, apply Theorem \ref{thm:polynomial-T1} with
$$
 P(\xx)=|\xx|^2+1,\quad b(j)=\mathbf1_{j=d},\quad h(j)=j^2,
 \quad\omega(j)=1,\quad c(u)=u,\quad
 g(\xx)=\frac{e^{-\pi t|\xx|^2}}{P(\xx)};
$$
the product $g(\xx)c(P(\xx))$ is the Gaussian, whose Fourier transform is $t^{-1}e^{-\pi|\xxi|^2/t}$ and, consequently,
$$
 T_d(t)=\frac1{td^4}\sum_{\mm\in\Z^2}
 \mathcal G_{P,d^2}(\mm)e^{-\pi|\mm|^2/(td^4)}.
$$
In this sum the term $\mm=0$ is $V_d(t):=\mathcal G_{P,d^2}(0)/(td^4)$: it is the proportion $\mathcal G_{P,d^2}(0)/d^4$ of admissible residue classes multiplied by the Gaussian integral $1/t$. The preceding proposition gives $\sum_{d\ge1}\mu(d)V_d(t)=C/t$. Subtracting this term for each $d$ leaves exactly the nonzero Fourier frequencies:
$$
 T_d(t)-V_d(t)=\frac1{td^4}\sum_{\mm\ne0}
 \mathcal G_{P,d^2}(\mm)e^{-\pi|\mm|^2/(td^4)}.
$$

\noindent\textbf{4.} Fix $0<t\le1$ and introduce a real cutoff $Y\ge1$, to be chosen below. We now have the exact decomposition
\begin{align*}
 \Theta_{\rm sf}(t)-\frac Ct
 ={}&\sum_{d\le Y}\mu(d)\bigl(T_d(t)-V_d(t)\bigr)\\
 &+\sum_{d>Y}\mu(d)T_d(t)-\sum_{d>Y}\mu(d)V_d(t).
\end{align*}
We use the Fourier expression only for $d\le Y$. For larger $d$, the divisibility condition in the original sum gives a better estimate. The last sum subtracts the part of $C/t$ belonging to those larger divisors; it must therefore be retained. We estimate these three sums in order.\\

\noindent\textbf{5.} Let $0<\eta<1$ be a small number, to be chosen in terms of $\epsilon$ at the end. For odd $q$, the two-variable finite sum satisfies Tolev's estimate \cite[Lemma 3]{Tolev12}
$$
 |\mathcal G_{P,q}(m_1,m_2)|
 \le16\tau(q)^2q^{1/2}\gcd(q,m_1,m_2)^{1/2};
$$
Tolev proves this using quadratic Gauss sums and Weil's bound for Kloosterman sums. His Fourier sign is the opposite of ours, which does not change the absolute value. Since $\tau(q)^2\ll_\eta q^{\eta/2}$ and the greatest common divisor is one of the common divisors, this implies
\begin{equation}\label{eq:sf-gcd-bound}
 |\mathcal G_{P,q}(\mm)|\ll_\eta q^{1/2+\eta/2}
 \sum_{\substack{\delta\mid q\\\delta\mid m_1,\,\delta\mid m_2}}\delta^{1/2}.
\end{equation}
Only squarefree $d$ matter. If $d$ is even, the congruence is impossible modulo $4$, so both $T_d$ and $V_d$ vanish. For odd $d$ we use this bound with $q=d^2$; the case $d=1$ also satisfies it directly.

\noindent\textbf{6.} We will need the elementary Gaussian estimate
$$
 \sum_{\zz\in\Z^2\setminus\{0\}}e^{-\pi|\zz|^2/u}\ll u
 \qquad(u>0).
$$
For $u\ge1$, comparison with the one-dimensional Gaussian integral gives
$\sum_{n\in\Z}e^{-\pi n^2/u}\le1+\sqrt u$; squaring and removing the origin proves the claim. For $0<u\le1$, use $|\zz|^2\ge1$ to bound the sum by
$$
 e^{-\pi/(2u)}\sum_{\zz\ne0}e^{-\pi|\zz|^2/2}
 \ll e^{-\pi/(2u)}\ll u.
$$
Inserting the finite-sum bound and writing $\mm=\delta\zz$ now gives
\begin{align*}
 |T_d(t)-V_d(t)|
 &\ll_\eta\frac{d^{1+\eta}}{td^4}
   \sum_{\delta\mid d^2}\delta^{1/2}
   \sum_{\zz\ne0}e^{-\pi\delta^2|\zz|^2/(td^4)}\\
 &\ll_\eta\frac{d^{1+\eta}}{td^4}
   \sum_{\delta\mid d^2}\delta^{1/2}\frac{td^4}{\delta^2}
 \ll_\eta d^{1+\eta}.
\end{align*}
In the last step we used $\sum_{\delta\mid d^2}\delta^{-3/2}\le\zeta(3/2)$. Thus the first part of the decomposition is bounded by
$$
 \sum_{d\le Y}|\mu(d)|\,|T_d(t)-V_d(t)|
 \ll_\eta\sum_{d\le Y}d^{1+\eta}\ll_\eta Y^{2+\eta}.
$$

\noindent\textbf{7.} For the second part, return to the lattice points. If $d>Y\ge1$, then $d\ge2$ and the origin does not occur, since $d^2\nmid P(0)=1$. Write $|\vv|^2+1=d^2\ell$, where $\ell\ge1$. Grouping the points by their norm gives
$$
 T_d(t)=\sum_{\ell\ge1}r_2(d^2\ell-1)e^{-\pi t(d^2\ell-1)}.
$$
Jacobi's two-square formula gives $r_2(N)\le4\tau(N)\ll_\eta N^\eta$ for $N\ge1$. Also, $e^{\pi t}\le e^\pi$ in the range under consideration. Consequently
\begin{align}\label{eq:large-d}
 \sum_{d>Y}|\mu(d)|T_d(t)
 &\ll_\eta\sum_{d>Y}d^{2\eta}
       \sum_{\ell\ge1}\ell^\eta e^{-\pi td^2\ell}\notag\\
 &\ll_\eta\sum_{d>Y}d^{2\eta}(td^2)^{-1-\eta}\notag\\
 &\ll_\eta t^{-1-\eta}\sum_{d>Y}d^{-2}
 \ll_\eta\frac{t^{-1-\eta}}Y.
\end{align}
Here $\sum_{\ell\ge1}\ell^\eta e^{-u\ell}\ll_\eta u^{-1-\eta}$ for every $u>0$. For $u\le1$ this follows by integral comparison; for $u\ge1$, extract $e^{-u/2}$ and bound the remaining sum by $\sum_{\ell\ge1}\ell^\eta e^{-\ell/2}$.\\

\noindent\textbf{8.} For the third part, the local counts computed above give, for odd squarefree $d$,
$$
 \mathcal G_{P,d^2}(0)
 =d^2\prod_{p\mid d}\left(1-\frac{\chi_{-4}(p)}p\right)
 \le d^2\tau(d)\ll_\eta d^{2+\eta}.
$$
It follows that
\begin{equation}\label{eq:zero-tail}
 \sum_{d>Y}|\mu(d)|V_d(t)
 \ll_\eta\frac1t\sum_{d>Y}d^{-2+\eta}
 \ll_\eta\frac1tY^{-1+\eta}.
\end{equation}

\noindent\textbf{9.} Combining the three estimates yields
$$
 \left|\Theta_{\rm sf}(t)-\frac Ct\right|
 \ll_\eta Y^{2+\eta}+\frac{t^{-1-\eta}}Y
                         +t^{-1}Y^{-1+\eta}.
$$
Ignoring the arbitrarily small powers involving $\eta$, the competing terms are $Y^2$ and $t^{-1}/Y$. They are equal when $Y=t^{-1/3}$. With this choice, the displayed bound is $O_\eta(t^{-2/3-\eta})$. Taking $0<\eta<\min(1,\epsilon)$ proves the asserted estimate.

\noindent\textbf{10.} Finally, removing the lattice origin and integrating term by term, initially for $\Re s>1$, gives
\begin{align*}
 \int_0^\infty\bigl(\Theta_{\rm sf}(t)-1\bigr)t^{s-1}\,dt
 &=\sum_{N\ge1}r_2(N)\mu^2(N+1)
                 \int_0^\infty e^{-\pi tN}t^{s-1}\,dt\\
 &=\pi^{-s}\Gamma(s)Z_{\rm sf}(s).
\end{align*}
The subtraction of $1$ makes the integrand decay exponentially at infinity. To improve convergence near zero, split the integral at $1$ and write, on $(0,1)$,
$$
 \Theta_{\rm sf}(t)-1
 =\left(\Theta_{\rm sf}(t)-\frac Ct\right)+\frac Ct-1.
$$
The last two terms can be integrated explicitly:
$$
 \int_0^1\left(\frac Ct-1\right)t^{s-1}\,dt
 =\frac C{s-1}-\frac1s\qquad(\Re s>1).
$$
We therefore obtain
\begin{align}\label{eq:conditional-cont}
 \pi^{-s}\Gamma(s)Z_{\rm sf}(s)
 ={}&\int_1^\infty(\Theta_{\rm sf}(t)-1)t^{s-1}\,dt\notag\\
 &+\int_0^1\left(\Theta_{\rm sf}(t)-\frac Ct\right)t^{s-1}\,dt
 +\frac C{s-1}-\frac1s.
\end{align}
The first integral is entire because of exponential decay. For the second, the estimate just proved gives
$$
 \left|\left(\Theta_{\rm sf}(t)-\frac Ct\right)t^{s-1}\right|
 \ll_\epsilon t^{\Re s-5/3-\epsilon}\qquad(0<t\le1).
$$
This is integrable whenever $\Re s>2/3+\epsilon$. On any compact subset of $\Re s>2/3$, choose $\epsilon$ smaller than the distance to the boundary; the convergence is then uniform, and the integral is holomorphic there. Thus the displayed formula continues $\pi^{-s}\Gamma(s)Z_{\rm sf}(s)$ to $\Re s>2/3$, with its only pole in that region at $s=1$, of residue $C$. Multiplication by $\pi^s/\Gamma(s)$ gives the residue $\pi C$ for $Z_{\rm sf}$ and introduces no other poles.
\end{proof}
\begin{remark}[The estimates and their relation to Tolev]
The exponent $2/3$ uses cancellation in the finite sums \eqref{eq:sf-gcd-bound}. To see what this input contributes, replace it by the elementary bound
$$
 |\mathcal G_{P,d^2}(\mm)|\le\mathcal G_{P,d^2}(0)\ll_\epsilon d^{2+\epsilon}.
$$
In the preceding proof this gives $|T_d(t)-V_d(t)|\ll_\epsilon d^{2+\epsilon}$, so the contribution of $d\le Y$ is $O_\epsilon(Y^{3+\epsilon})$, rather than $O_\epsilon(Y^{2+\epsilon})$. The tails are unchanged. Balancing $Y^3$ with $t^{-1}/Y$ gives $Y=t^{-1/4}$ and the weaker error $O_\epsilon(t^{-3/4-\epsilon})$. Thus elementary residue counting already gives continuation to $\Re s>3/4$, while the exponential-sum estimate improves it to $\Re s>2/3$. It is worth noting that the same finite sums occur in Tolev's squarefree count. Tolev's $\lambda(q;a,b)=\mathcal G_{P,q}(-a,-b)=\mathcal G_{P,q}(a,b)$ by symmetry. Let $N_q(H_1,H_2)$ count the solutions of $q\mid P(x,y)$ in $1\le x\le H_1$, $1\le y\le H_2$, and put $S_J(\xi)=\sum_{j=1}^J\ee{j\xi}$. Finite inversion gives
\begin{equation}\label{eq:sf-finite-rectangle}
 N_q(H_1,H_2)=q^{-2}\sum_{a,b\bmod q}\mathcal G_{P,q}(a,b)
 S_{H_1}(a/q)S_{H_2}(b/q).
\end{equation}
For $H=\max(H_1,H_2)$, the same divisor split gives errors $O_\epsilon(Y^{2+\epsilon}+H^{1+\epsilon})$ and $O_\epsilon(H^{2+\epsilon}/Y)$; the frequency calculation is included in the general proof below. Taking $Y=H^{2/3}$ yields
\begin{equation}\label{eq:sf-rectangle-asymptotic}
 S(H_1,H_2):=\sum_{1\le x\le H_1,\,1\le y\le H_2}\mu^2(P(x,y))
 =CH_1H_2+O_\epsilon(H^{4/3+\epsilon}),
\end{equation}
whose specialization $H_1=H_2=H$ is Tolev's result \cite{Tolev12}. The exact transforms are reversible: for Schwartz $f$, finite inversion and the period-one case of Theorem \ref{thm:polynomial-T1} give
$$
 \sum_{q\mid P(\vv)}f(\vv)
 =q^{-2}\sum_{\mm\bmod q}\mathcal G_{P,q}(\mm)\sum_{\boldsymbol{\ell}\in\Z^2}\widehat f(\boldsymbol{\ell}+\mm/q)
 =q^{-2}\sum_{\mm\in\Z^2}\mathcal G_{P,q}(\mm)\widehat f(\mm/q).
$$
The Gaussian choice recovers \eqref{eq:sf-dual}; a smooth interpolant of rectangle samples recovers \eqref{eq:sf-finite-rectangle}. Extend $S(H_1,H_2)$ to nonnegative real endpoints by taking their integer parts. Two-variable partial summation then gives
$$
 \Theta_{\rm sf}(t)=4(2\pi t)^2\int_0^\infty\!\int_0^\infty
 H_1H_2 e^{-\pi t(H_1^2+H_2^2)}S(H_1,H_2)\,dH_1\,dH_2+O(t^{-1/2}),
$$
The term $O(t^{-1/2})$ accounts for the coordinate axes. The main term $CH_1H_2$ integrates to $C/t$. In the error term, the substitutions $H_i=t^{-1/2}u_i$ give $O_\epsilon(t^{-2/3-\epsilon})$. Thus the uniform rectangle estimate implies \eqref{eq:sf-sharp-Gaussian}. Reversibility of the exact finite transforms does not by itself reverse this implication between error estimates; a Gaussian asymptotic alone need not recover the sharp-square count.
\end{remark}

\begin{remark}[What the finite selectors retain]\label{rem:sf-finite-pieces}
The product for $C$ in \eqref{eq:Csf} multiplies local proportions of admissible residue classes. It is a product for the density, not for the Dirichlet series $Z_{\rm sf}$. Indeed, its normalized coefficients $r_2(N)\mu^2(N+1)/4$ take the values $1,2,0$ at $N=2,13,26$, although $2$ and $13$ are coprime. They are therefore not multiplicative. The useful connection with classical Dirichlet series comes instead from fixing one modulus at a time. Put
\begin{equation}\label{eq:sf-finite-pieces}
 E_{P,q}(s)=\sum_{\substack{\vv\in\Z^2\setminus\{0\}\\q\mid P(\vv)}}|\vv|^{-2s},
 \qquad Z_{\rm sf}(s)=\sum_{d\ge1}\mu(d)E_{P,d^2}(s),\quad \Re s>1.
\end{equation}
Thus the squarefree series is a M\"obius sum of congruence-restricted Epstein series. For odd $d$, the condition $d^2\mid N+1$ makes $N$ a unit modulo $d^2$. Character orthogonality and $r_2=4(\mathbf1*\chi_{-4})$ give the more familiar expression
\begin{equation}\label{eq:sf-L-pieces}
 Z_{\rm sf}(s)=4\sum_{\substack{d\ge1\\d\ {\rm odd}}}\frac{\mu(d)}{\varphi(d^2)}
 \sum_{\chi\bmod d^2}\bar\chi(-1)L(s,\chi)L(s,\chi\chi_{-4}).
\end{equation}
Here $\varphi$ is Euler's totient, and the full character sum is taken before the sum over $d$; the divisor bound justifies this order for $\Re s>1$. These formulas explain the link with classical $L$-functions without asserting multiplicativity of the shifted coefficients. A fixed modulus has an Epstein reflection with a transformed finite weight \cite{Epstein03}; the unbounded moduli in the M\"obius sum do not by themselves give a functional equation for $Z_{\rm sf}$.
\end{remark}

\subsection{Higher power-free values}\label{sec:kfree-general}
Keep the preceding argument and replace $d^2$ by $d^r$ and $x^2+y^2+1$ by $P(\vv)=|\vv|^2+k$, with $n,r\ge2$ and $k\ge1$. Here $r$ denotes the excluded power and $k$ the shift. Put $\mu_r(N)=\sum_{d^r\mid N}\mu(d)$ and $\Theta_{r,k,n}(t)=\sum_{\vv\in\Z^n}\mu_r(P(\vv))e^{-\pi t|\vv|^2}$. The same specialization, now with $h(d)=d^r$, gives
\begin{align}
 \Theta_{r,k,n}(t)&=t^{-n/2}\sum_{d\ge1}\frac{\mu(d)}{d^{rn}}
 \left(\sum_{\mm\in\Z^n}\mathcal G_{P,d^r}(\mm)e^{-\pi|\mm|^2/(td^{2r})}\right),\label{eq:powerfree-exact}\\
 C_{r,k,n}&=\prod_p\left(1-\frac{\mathcal G_{P,p^r}(0)}{p^{rn}}\right)>0.\label{eq:powerfree-density}
\end{align}
As before, the divisor bound justifies the grouped order. For $p\nmid2k$, nonsingular lifting gives $\mathcal G_{P,p^r}(0)\ll_n p^{r(n-1)}$, so the product converges absolutely. Every factor is positive since $P(0)$ and $P(\mathbf e_1)$ differ by $1$. The origin contributes $\mu_r(k)$.

Figure \ref{fig:powerfree-lifting} illustrates the local factor at $p=3$ for $P(x,y)=x^2+y^2+1$. Modulo $3$ there are four solutions $(\pm1,\pm1)$. A solution $(x,y)$ modulo $3^j$ lifts to $(x+3^ja,y+3^jb)$ modulo $3^{j+1}$ exactly when
$$
 \frac{P(x,y)}{3^j}+2xa+2yb\equiv0\pmod3.
$$
There are three choices among the nine pairs $(a,b)\bmod3$, since $x,y$ are nonzero modulo $3$. Hence there are twelve excluded classes modulo $9$ and thirty-six modulo $27$, giving local factors $1-12/81=1-4/27$ and $1-36/729=1-4/81$. A higher excluded power removes fewer residue classes: these conditions do not select a single dilated sublattice and the Euler product combines the local factors; controlling the infinite sieve requires the estimates below.
\begin{figure}[htbp]
 \centering
 \includegraphics[width=\linewidth]{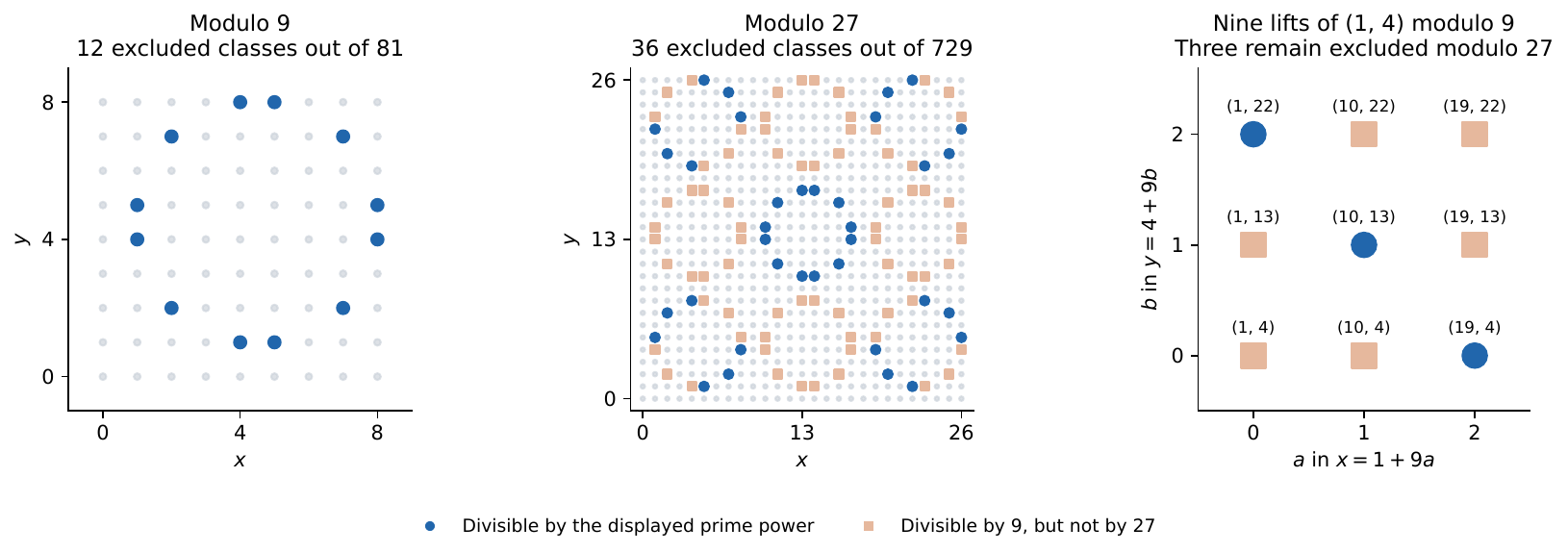}
 \caption{Local exclusions for $P(x,y)=x^2+y^2+1$. Blue circles satisfy $9\mid P$ in the first panel and $27\mid P$ in the second. Pale squares satisfy $9\mid P$ but not $27\mid P$. Each solution modulo $9$ has nine lifts modulo $27$, three of which remain excluded. For the nine lifts $(1+9a,4+9b)$ of $(1,4)$, shown on the right, these are exactly the pairs $a+b\equiv2\pmod3$. The panels depict local obstructions, not globally squarefree or cube-free values.}
 \label{fig:powerfree-lifting}
\end{figure}

\begin{theorem}\label{thm:kfree-general}
Let $\theta=n\{r(n-1)+2\}/[2r(n+1)]$. For fixed $n,r\ge2$, $k\ge1$ and every $\epsilon>0$,
\begin{align}
 \Theta_{r,k,n}(t)&=C_{r,k,n}t^{-n/2}+O_\epsilon(t^{-\theta-\epsilon})\quad(0<t\le1),\label{eq:kfree-theta-bound}\\
 \sum_{1\le v_i\le H_i}\mu_r(|\vv|^2+k)
 &=C_{r,k,n}\prod_iH_i+O_\epsilon(H^{n-1+\epsilon}+H^{2\theta+\epsilon}),\label{eq:kfree-box-bound}
\end{align}
where $H=\max_iH_i$ and the $H_i$ are positive integers. Moreover,
$Z_{r,k,n}(s)=\sum_{\vv\ne0}\mu_r(|\vv|^2+k)|\vv|^{-2s}$ continues to $\Re s>\theta$, with only a simple pole at $n/2$, of residue $\pi^{n/2}C_{r,k,n}/\Gamma(n/2)$.
\end{theorem}

\begin{proof}
\textbf{1.} We record the changes to the squarefree proof. Put $\lambda=(n-1)/2$ and  for odd
$q$ prime to $k$, group the additive multiplier by $\delta=\gcd(a,q)$ and write
$\ell=q/\delta$:
\begin{equation}\label{eq:local-gcd}
\mathcal G_{P,q}(\mm)=q^{-1}\sum_{\substack{\delta\mid q\\\delta\mid \mm}}\delta^n
\sum_{a\bmod\ell}^{*}\ee{ka/\ell}
\prod_{i=1}^n\mathcal G_\ell(a,m_i/\delta),\qquad
\mathcal G_\ell(a,b)=\sum_{x\bmod\ell}\ee{(ax^2-bx)/\ell}.
\end{equation}
Here the star restricts $a$ to the units modulo $\ell$, and $\delta\mid\mm$ means that
$\delta$ divides every coordinate of $\mm$. The letter $a$ in this finite sum is a local
residue index; $t$ remains the positive Gaussian parameter.
Completing squares leaves Gauss factors of modulus $\ell^{n/2}$ and
a Kloosterman sum when $n$ is even, or its Jacobi-character twist when $n$ is odd. At prime modulus the latter is a Sali\'e sum. At modulus
$p^e$, the quadratic multiplier is trivial if $n$ or $e$ is even.
Since $k$ is a unit, the prime and prime-power
bounds, followed by the Chinese remainder theorem, give
$\tau(\ell)\sqrt\ell$; see \cite[Lemma 2]{Tolev12} and
\cite[Lemma 2.2]{ChenWang23}. The prime-modulus Kloosterman estimate
uses Weil’s bound \cite{Weil48,Iwaniec97}, whereas the prime-modulus Sali\'e evaluation is
elementary.
For coprime $q_1,q_2$, the finite coefficients factor as
$$
\mathcal G_{P,q_1q_2}(\mm)
=\mathcal G_{P,q_1}(u_1\mm)\mathcal G_{P,q_2}(u_2\mm),
$$
where $u_1q_2\equiv1\pmod{q_1}$ and $u_2q_1\equiv1\pmod{q_2}$. Thus the linear arguments
are rescaled by the unit inverses in the Chinese remainder theorem. Substitution gives
\begin{equation}\label{eq:rho-bound}
 |\mathcal G_{P,q}(\mm)|\ll_{n,r,k,\epsilon}q^{\lambda+\epsilon}
 \sum_{\substack{\delta\mid q\\\delta\mid \mm}}\delta^\lambda
 \qquad(q=d^r,\ \mu(d)\ne0).
\end{equation}
For these moduli the part supported on primes dividing $2k$ divides the fixed integer $\prod_{p\mid2k}p^r$; a trivial bound and the Chinese remainder theorem include it in the constant.

\noindent \textbf{2.} For every $u>0$, $\sum_{\zz\ne0}e^{-\pi|\zz|^2/u}\ll_n u^{n/2}$, by Gaussian integral comparison for $u\ge1$ and exponential decay for $u\le1$. Thus \eqref{eq:rho-bound}, with $\mm=\delta \zz$, gives
\begin{equation}\label{eq:weighted-Gaussian-bound}
 \sum_{\mm\ne0}|\mathcal G_{P,q}(\mm)|e^{-\pi|\mm|^2/u}
 \ll q^{\lambda+\epsilon}u^{n/2}\sum_{\delta\mid q}\delta^{-(n+1)/2}
 \ll q^{\lambda+\epsilon}u^{n/2}.
\end{equation}
At $u=td^{2r}$ this cancels the prefactor in \eqref{eq:powerfree-exact}, giving a small-modulus error $O(Y^{r\lambda+1+\epsilon})$.

\noindent\textbf{3.} The tail follows exactly as before: fixing $n-2$ coordinates and using $r_2(M)\ll_\epsilon M^\epsilon$ (and $r_2(0)=1$) gives $r_n(N)\ll_\epsilon N^{n/2-1+\epsilon}$. With $A=n/2-1+\epsilon$, write $|\vv|^2=d^r\ell-k$ to obtain
$$
 \sum_{d>Y}\sum_{\substack{\vv\ne0\\d^r\mid P(\vv)}}e^{-\pi t|\vv|^2}
 \ll\sum_{d>Y}d^{rA}\sum_{\ell\ge1}\ell^Ae^{-\pi td^r\ell}
 \ll t^{-n/2-\epsilon}Y^{1-r}.
$$
The density tail has the same bound up to $Y^\epsilon$; the origin adds $O_k(1)$. Balancing the two errors with $Y=t^{-n/(r(n+1))}$ proves the Gaussian estimate. The split Mellin formula \eqref{eq:conditional-cont}, with $1,C/t$ replaced by $\mu_r(k),C_{r,k,n}t^{-n/2}$, proves the continuation and residue.

\noindent\textbf{4.} For rectangles, replace the two factors in \eqref{eq:sf-finite-rectangle} by $\prod_iS_{H_i}(m_i/q)$ and $q^{-2}$ by $q^{-n}$. If exactly $j$ coordinates of $\mm\ne0$ vanish, the bound
$\sum_{0<|a|\le q/2,\,\delta\mid a}|S_{H_i}(a/q)|\ll(q/\delta)\log(2q)$ gives the contribution
\begin{equation}\label{eq:box-gcd-bound}
 \ll H^jq^{\lambda-j+\epsilon}\sum_{\delta\mid q}\delta^{j-(n+1)/2}
 \ll H^jq^\epsilon(q^{\lambda-j}+q^{-1}),\qquad 0\le j\le n-1.
\end{equation}
For the cutoff chosen below, $Y^r\ge H$. If $r(\lambda-j)+1>0$, then
$$
H^jY^{r(\lambda-j)+1}
=Y^{r\lambda+1}(H/Y^r)^j\le Y^{r\lambda+1}.
$$
If $r(\lambda-j)+1\le0$, summing over $d\le Y$ contributes at most $H^{n-1+\epsilon}$ and a
possible logarithm is absorbed into $\epsilon$; the summable $d^{-r}$ terms have the same
bound. Hence the small-modulus contribution is
$O(H^{n-1+\epsilon}+Y^{r\lambda+1+\epsilon})$ and
the representation bound gives the arithmetic tail $O(H^{n+\epsilon}Y^{1-r})$, by summing
$r_n(d^r\ell-k)$ for $\ell\ll_k H^2/d^r$; since the density tail is identical in size,  taking
$Y=H^{2n/(r(n+1))}$, which satisfies $Y^r\ge H$, proves the rectangle estimate.
\end{proof}

\begin{remark}
The parallel with the classical counts is summarized in
Table \ref{tab:powerfree-comparison}. For $n=3$ and $r=2$, Zhou--Ding \cite{ZhouDing22}
obtained an error $O_\epsilon(H^{7/3+\epsilon})$; Chen--Wang \cite{ChenWang23} improved
this to $O_\epsilon(H^{9/4+\epsilon})$ and treated higher excluded powers. These APSF derivations use the same exponential-sum estimates and do not improve the known exponents. More generally, Lapkova--Xiao \cite[Theorem 1.1]{LapkovaXiao19} establish the expected
density asymptotic for power-free values of square-free polynomials under their degree
condition. Their result applies to $P(\vv)=|\vv|^2+k$ for every $n,r\ge2$ and $k>0$: this
polynomial has degree two and is square-free, the condition on the excluded power is
$r\ge7/4$, and no prime power divides all its values because $P(0)=k$ and $P(\mathbf
e_1)=k+1$. Their qualitative asymptotic does not state the explicit power error
considered here. The present argument also supplies the Gaussian estimates and the
continuation of the selector Dirichlet series.
\end{remark}
\begin{table}[htbp]
\centering\small
\begin{tabularx}{\linewidth}{@{}l>{\raggedright\arraybackslash}X>{\raggedright\arraybackslash}X@{}}
\toprule
Setting & Box error (allowing $\epsilon$) & Comparison\\
\midrule
$n=r=2$, $k=1$ & $H^{4/3+\epsilon}$ & Tolev \cite{Tolev12}.\\[3pt]
$n=3$, $r\ge2$, $k>0$ & $H^{2+\epsilon}+H^{3/2+3/(2r)+\epsilon}$ & Chen--Wang \cite[Theorem 1.1(i)]{ChenWang23}; $r=2$ gives $H^{9/4+\epsilon}$, improving $H^{7/3+\epsilon}$ of Zhou--Ding \cite{ZhouDing22}.\\[3pt]
$n,r\ge2$, $k>0$ & $H^{n-1+\epsilon}+H^{2\theta+\epsilon}$ & Theorem \ref{thm:kfree-general}; no optimality claim.\\
$n\ge1$, square-free $F$, $\deg F\ge2$ & asymptotic only & Lapkova--Xiao \cite[Theorem 1.1]{LapkovaXiao19}, for $r\ge(3\deg F+1)/4$ and with no fixed prime $r$th-power divisor.\\
\bottomrule
\end{tabularx}
\vspace{0.15cm}
\caption{Power-free counts. Chen--Wang's nonpositive shifts are outside the present hypotheses.}\label{tab:powerfree-comparison}
\end{table}

\subsection{The Terras connection and all odd zeta values}\label{sec:terras}
The idea behind Terras's formulas is to study an ordinary $L$-function through a lattice sum with more coordinates \cite{Terras73,Terras76}. For example,
$$
 Z_2(s)=\sum_{(a,b)\in\Z^2\setminus\{(0,0)\}}(a^2+b^2)^{-s}
       =4\zeta(s)\beta(s),\qquad \Re s>1.
$$
This is a ``two-variable'' sum in the lattice coordinates $a,b$, with one complex parameter $s$. Its geometry gives information about the one-variable Dirichlet series on the right. APSF organizes the passage: a divisor $d$ of a vector's content becomes an extra coordinate on the Fourier side. A planar sum is thereby related to a three-dimensional Epstein sum. At the symmetric lattice, differentiation eliminates that extra sum and leaves a formula for $Z_2(s)$, hence for the odd zeta and even beta values. At $s=3/2$ this is the argument of Terras \cite[Proposition 4]{Terras76}.

We give the identity in dimension $n$ so that this mechanism is visible. For $n\ge1$, $\nu>1/2$, and $t>0$, put $s=\nu+n/2$, $C_a=\Gamma(a)/(2\pi^a)$, and $g_{\nu,t}(\xx)=|\xx|^\nu K_\nu(2\pi t|\xx|)$, where $K_\nu$ is defined in \eqref{eq:Kdefinition}.\footnote{This is the same Mat\'ern kernel as in footnote \ref{fn:matern}, now on $\R^n$.} For $\vv\in\Z^n\setminus\{0\}$, $\gcd(\vv)$ means the positive greatest common divisor of its coordinates, including any zero coordinates, and $\sigma_a(N)=\sum_{d\mid N}d^a$. Define
\begin{align*}
 S_n(\nu,t)&=\sum_{\vv\in\Z^n\setminus\{0\}}
 \sigma_{-2\nu}(\gcd(\vv))|\vv|^\nu K_\nu(2\pi t|\vv|),\\
 Z_n(s)&=\sum_{\mm\in\Z^n\setminus\{0\}}|\mm|^{-2s},\qquad
 Z_{n+1,t}(s)=\sum_{(d,\mm)\in\Z^{n+1}\setminus\{0\}}
 (t^2d^2+|\mm|^2)^{-s}.
\end{align*}
The first sum converges exponentially; the Epstein series converge absolutely at the indicated $s>(n+1)/2$. Our Epstein series count both a vector and its negative; Terras's definition in \cite{Terras73} is one half of this sum.

\begin{theorem}[Divisor-weighted Bessel identity]\label{thm:terras-content}
With the preceding notation,
\begin{equation}\label{eq:terras-content}
 S_n(\nu,t)=\frac{C_s}{2}t^\nu
 \bigl(Z_{n+1,t}(s)-Z_n(s)\bigr)
 -C_\nu t^{-\nu}\zeta(2\nu).
\end{equation}
\end{theorem}
\begin{proof}
The usual integral for $K_\nu$ \cite[\S10.32.10]{DLMF}, followed by the Gaussian Fourier transform, gives
\begin{align*}
 g_{\nu,t}(\xx)&=\frac1{2(\pi t)^\nu}\int_0^\infty
 u^{\nu-1}\exp\left(-u-\frac{\pi^2t^2|\xx|^2}{u}\right)\,du,\\
 \widehat g_{\nu,t}(\xxi)
 &=\frac1{2(\pi t)^\nu(\pi t^2)^{n/2}}
 \int_0^\infty u^{s-1}\exp\left(-u\left(1+\frac{|\xxi|^2}{t^2}\right)\right)\,du
 =C_st^\nu(t^2+|\xxi|^2)^{-s}.
\end{align*}
In particular $g_{\nu,t}(0)=C_\nu t^{-\nu}\ne0$. These formulas also justify Poisson summation: apply it to each Gaussian under the integral. On the zero residue class both sides are nonnegative, so Tonelli's theorem applies. On any other residue class the transformed terms have factors of absolute value one and are dominated by the zero-class frequency sum, which converges since $2s>n$. Thus $g_{\nu,t}$ is admissible on all residue classes; this argument supplies the required admissibility at the origin.

For each $d$, apply Theorem \ref{thm:polynomial-T1} with the constant polynomial $P=1$, $c(u)=u$, $b(j)=\mathbf1_{j=1}$, $h(j)=j$, $\omega(j)=1$, and $g(\xx)=g_{\nu,t}(d\xx)$. As in Remark \ref{rem:constant-polynomial}, the origin is included and scaling gives
$$\sum_{\vv\in\Z^n}g_{\nu,t}(d\vv)=d^{-n}\sum_{\mm\in\Z^n}\widehat g_{\nu,t}(\mm/d).$$
Multiply by $d^{-2\nu}$ and sum over $d$. This arithmetic series is absolutely convergent since $\sum_dd^{-2\nu}<\infty$ and $g_{\nu,t}$ decays exponentially. Regrouping $d\vv$ gives $\sigma_{-2\nu}(\gcd(\vv))$ at each nonzero vector, the same vector-divisibility grouping as in Theorem \ref{thm:vector-dilation}. Subtracting the actual origin term gives
\begin{align*}
 S_n(\nu,t)
 &=\sum_{d\ge1}d^{-2\nu}
 \left(d^{-n}\sum_{\mm\in\Z^n}\widehat g_{\nu,t}(\mm/d)-g_{\nu,t}(0)\right)\\
 &=C_st^\nu\sum_{d\ge1}\sum_{\mm\in\Z^n}
 (t^2d^2+|\mm|^2)^{-s}-C_\nu t^{-\nu}\zeta(2\nu).
\end{align*}
Here $2s=2\nu+n$ cancels every power of $d$ outside the denominator: the arithmetic dilation has become the geometric coordinate $d$. The slice $d=0$ is $Z_n$, and the positive and negative slices contribute equally. Thus the double sum is
$\tfrac12(Z_{n+1,t}-Z_n)$, proving \eqref{eq:terras-content}. The subtracted term is the lattice origin retained by the data $P=1$ and $c(u)=u$.
\end{proof}

\begin{remark}Writing $\vv=d\mm$ in the divisor sum also gives
$$
 S_n(\nu,t)=\sum_{d\ge1}\sum_{\mm\ne0}
 \left(\frac{|\mm|}{d}\right)^\nu K_\nu(2\pi td|\mm|).
$$
Consequently \eqref{eq:terras-content} is Terras's Epstein expansion for the quadratic form $t^2d^2+m_1^2+\cdots+m_n^2$, written with divisor weights \cite{Terras73}. At $t=1$ all $n+1$ coordinates enter through their squares with the same coefficient. This symmetry makes it possible to eliminate $Z_{n+1,1}(s)$ by combining \eqref{eq:terras-content} with its derivative in $t$.
\end{remark}

For $n=2$, Figure \ref{fig:terras-layers} depicts the sum over $d$ and $\mm=(m_1,m_2)$. The plane $d=0$ contributes $Z_2(s)$, and the layers with $d$ and $-d$ have the same denominators. Removing that plane and keeping only $d>0$ therefore gives
$$
 \sum_{d\ge1}\sum_{\mm\in\Z^2}(t^2d^2+|\mm|^2)^{-s}
 =\frac12\{Z_{3,t}(s)-Z_2(s)\},\qquad s>3/2.
$$
This is the geometric step in \eqref{eq:terras-content}.
\begin{figure}[htbp]
 \centering
 \includegraphics[width=.84\linewidth]{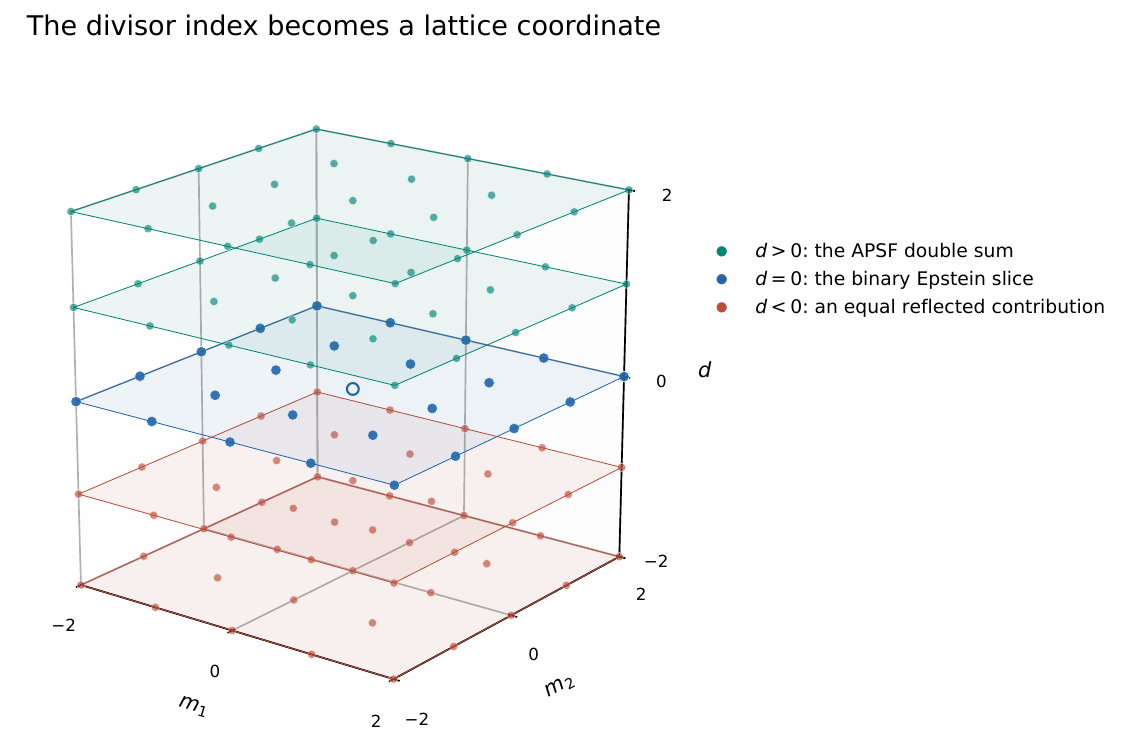}
 \caption{The layers of $(d,m_1,m_2)\in\Z^3$ at $t=1$. The blue plane $d=0$ contributes $Z_2(s)$; its origin is omitted. Equal contributions from $d>0$ and $d<0$ explain the factor $1/2$ in $\tfrac12\{Z_{3,t}(s)-Z_2(s)\}$. For general $t$, the point is placed at $(td,m_1,m_2)$, so changing $t$ changes only its first coordinate.}
 \label{fig:terras-layers}
\end{figure}

The next step extracts the planar sum from this three-dimensional identity. At $t=1$, permutation of $d,m_1,m_2$ gives
$$
 \sum_{(d,m_1,m_2)\ne0}\frac{d^2}{(d^2+m_1^2+m_2^2)^{s+1}}
 =\frac13 Z_{3,1}(s),\qquad
 \left.\frac{d}{dt}Z_{3,t}(s)\right|_{t=1}=-\frac{2s}{3}Z_{3,1}(s).
$$
The factor $1/3$ follows by adding the three equal numerator sums: their numerators add up to the denominator's quadratic form. Hence the value and derivative of \eqref{eq:terras-content} at $t=1$ give two linear equations involving $Z_{3,1}(s)$ and $Z_2(s)$. Eliminating $Z_{3,1}(s)$ leaves $Z_2(s)=4\zeta(s)\beta(s)$ expressed through the convergent Bessel series and its derivative. At $s=2k+1$, the value $\beta(2k+1)$ is explicit, so this determines $\zeta(2k+1)$. At $s=2k$, the explicit value $\zeta(2k)$ similarly determines $\beta(2k)$. The corollary gives the resulting formulas.

\begin{corollary}[Odd zeta values and even beta values]\label{cor:terras-allodd}
For an integer $\nu\ge1$, put $s=\nu+1$, $S_\nu=S_2(\nu,1)$, and
\begin{align*}
 T_\nu=\left.\partial_t S_2(\nu,t)\right|_{t=1}
 &=\sum_{\vv\in\Z^2\setminus\{0\}}\sigma_{-2\nu}(\gcd(\vv))
 2\pi|\vv|^{\nu+1}K_\nu'(2\pi|\vv|),\\
 R_\nu&=2(2\nu-1)C_\nu\zeta(2\nu)-3T_\nu+(\nu-2)S_\nu.
\end{align*}
The prime on $K_\nu$ denotes differentiation with respect to its argument. Then
\begin{equation}\label{eq:terras-zeta-beta}
 4C_s s\,\zeta(s)\beta(s)=R_\nu,
 \qquad \beta(s)=\sum_{j=0}^\infty\frac{(-1)^j}{(2j+1)^s}.
\end{equation}
In particular, for every $k\ge1$,
\begin{equation}\label{eq:terras-allodd}
 \zeta(2k+1)=\frac{R_{2k}}{4(2k+1)C_{2k+1}\beta(2k+1)},
 \qquad
 \beta(2k)=\frac{R_{2k-1}}{8kC_{2k}\zeta(2k)}.
\end{equation}
These are exponentially convergent Bessel series with explicit coefficients. In particular,
\begin{equation}\label{eq:terras-zeta3}
 \zeta(3)=\frac{4\pi^2}{45}
 -8\sum_{\vv\in\Z^2\setminus\{0\}}\sigma_{-4}(\gcd(\vv))
 2\pi|\vv|^3K_2'(2\pi|\vv|).
\end{equation}
\end{corollary}
\begin{proof}
At $t=1$, permutation symmetry of the $n+1$ coordinates gives
$$
 \left.\partial_t Z_{n+1,t}(s)\right|_{t=1}
 =-2s\sum_{(d,\mm)\ne0}\frac{d^2}{(d^2+|\mm|^2)^{s+1}}
 =-\frac{2s}{n+1}Z_{n+1,1}(s).
$$
All differentiated sums converge locally uniformly for $t>0$. Differentiating \eqref{eq:terras-content} at $t=1$ and substituting the preceding symmetry identity gives
$$
 \left.\partial_tS_n(\nu,t)\right|_{t=1}
 =\nu S_n(\nu,1)+2\nu C_\nu\zeta(2\nu)
 -\frac{C_s s}{n+1}Z_{n+1,1}(s).
$$
The original identity gives
$$
 Z_{n+1,1}(s)=Z_n(s)+\frac2{C_s}
 \{S_n(\nu,1)+C_\nu\zeta(2\nu)\}.
$$
Substitute this expression in the derivative equation and use $2s=2\nu+n$. The result is
\begin{equation}\label{eq:terras-derivative}
 \begin{aligned}
 C_s s Z_n(s)={}&n(2\nu-1)C_\nu\zeta(2\nu)
 -(n+1)\left.\partial_tS_n(\nu,t)\right|_{t=1}\\
 &+\bigl((n-1)\nu-n\bigr)S_n(\nu,1).
 \end{aligned}
\end{equation}
For $n=2$, the square-lattice identity $Z_2(s)=4\zeta(s)\beta(s)$ from Example \ref{sec:basic-characters} proves \eqref{eq:terras-zeta-beta}. The divisions in \eqref{eq:terras-allodd} are explicit because
$$
 \zeta(2k)=\frac{(-1)^{k+1}B_{2k}(2\pi)^{2k}}{2(2k)!},\qquad
 \beta(2k+1)=\frac{(-1)^kE_{2k}\pi^{2k+1}}{4^{k+1}(2k)!}.
$$
Here the Bernoulli and Euler numbers are defined by
$u/(\exp u-1)=\sum_{j\ge0}B_j u^j/j!$ and
$1/\cosh u=\sum_{j\ge0}E_{2j}u^{2j}/(2j)!$ near zero; these are the classical special-value formulas \cite{DLMF,Terras76}. Finally, with $\nu=2$, $C_3=\pi^{-3}$, $C_2=(2\pi^2)^{-1}$, $\zeta(4)=\pi^4/90$, and $\beta(3)=\pi^3/32$, equation \eqref{eq:terras-zeta-beta} becomes
$\tfrac38\zeta(3)=\pi^2/30-3T_2$, which is \eqref{eq:terras-zeta3}.
\end{proof}

\begin{remark}[Positive quadratic forms, phases, and arbitrary splits]
The same geometry works for a general positive quadratic form: separate one coordinate, complete the square in the others, and sum the resulting translated layers. This gives Terras's block expansions. For a real positive definite symmetric $n\times n$ matrix $Q$, write $Q[\vv]=\vv^{\mathsf T}Q\vv$, and, for $\xx\in\R^n$, define
$$
 S_Q^\xx(\nu,t)=\sum_{\vv\ne0}\ee{\vv\cdot \xx}
 \sigma_{-2\nu}(\gcd(\vv))Q[\vv]^{\nu/2}K_\nu(2\pi t\sqrt{Q[\vv]}).
$$
For any positive quadratic form $U$ in $r$ variables let
$Z(U;s)=\sum_{\ww\in\Z^r\setminus\{0\}}U[\ww]^{-s}$, initially $s>r/2$.
The change of variables $\yy=Q^{1/2}\vv$ in the Fourier integral introduces
$(\det Q)^{-1/2}$ and replaces $|\xxi|^2$ by $Q^{-1}[\xxi]$; multiplication by $\ee{\vv\cdot \xx}$ replaces $\xxi$ by $\xxi+\xx$. Thus the same constant-polynomial applications of Theorem \ref{thm:polynomial-T1}, and the same origin subtraction, prove
\begin{equation}\label{eq:terras-quadratic-content}
 S_Q^\xx(\nu,t)=\frac{C_s t^\nu}{2\sqrt{\det Q}}
 \bigl(Z(Y_{t,\xx};s)-Z(Q^{-1};s)\bigr)-C_\nu t^{-\nu}\zeta(2\nu),
 \quad Y_{t,\xx}[d,\mm]=t^2d^2+Q^{-1}[\mm+d\xx].
\end{equation}
The Gaussian-integral argument above remains valid, since the absolute values of the transformed Gaussian terms are positive Gaussian terms at shifted frequencies. To see every split, let $n_1,n_2\ge1$, let $T,R$ be positive definite matrices of those respective sizes, and let $B$ be a real $n_2\times n_1$ matrix describing the coupling between the two groups of coordinates. Every positive form in $n_1+n_2$ variables has a unique expression
$Y[\aaa,\mm]=T[\aaa]+R[\mm+B\aaa]$ relative to the chosen splitting.\footnote{The matrix $T$ is the Schur complement of the block $R$ in the matrix of $Y$; the displayed decomposition is completion of the square in block form.} For $s>(n_1+n_2)/2$, set $\mu=s-n_2/2$. Splitting off $\aaa=0$, and then writing each $\aaa\ne0$ uniquely as $\aaa=d\pp$ with $d\ge1$ and $\gcd(\pp)=1$, gives
\begin{align}\label{eq:terras-blocks}
 Z(Y;s)={}&Z(R;s)
 +\frac{\pi^{n_2/2}\Gamma(\mu)}{\Gamma(s)\sqrt{\det R}}Z(T;\mu)\notag\\
 &+\frac{2\pi^s}{\Gamma(s)\sqrt{\det R}}
 \sum_{\substack{\pp\in\Z^{n_1}\\\gcd(\pp)=1}}
 T[\pp]^{-\mu/2}S_{R^{-1}}^{B\pp}(\mu,\sqrt{T[\pp]}).
\end{align}
Indeed, apply \eqref{eq:terras-quadratic-content} with $Q=R^{-1}$, $t=\sqrt{T[\pp]}$, $\xx=B\pp$, and $\nu=\mu$ to the sum over $d\ge1,\mm$. Its subtracted origin contributes the middle term because
$\zeta(2\mu)\sum_{\gcd(\pp)=1}T[\pp]^{-\mu}=Z(T;\mu)$.
The sum over $\pp$ contains both signs. All rearrangements are absolutely convergent in the stated region. Formula \eqref{eq:terras-blocks}, upon expanding $S$, is Terras's arbitrary-block Bessel expansion \cite[Theorem 1]{Terras73}, written with divisor weights. The mechanism of Corollary \ref{cor:terras-allodd} is Terras’s: in \cite[Section 2, Proposition 4]{Terras76} she
compares the ternary diagonal forms $\mathrm{diag}(d,1,1)$ and $\mathrm{diag}(1,d,1)$ and differentiates at
$d=1$ to evaluate $Z_2(I,3/2)$, and \cite[Corollary to Proposition 3]{Terras76} gives a two-dimensional
Bessel series for $\zeta(3)$. Corollary \ref{cor:terras-allodd} extends her argument from $s=3/2$ to every $s
=\nu+1$.
\end{remark}

\section{Further extensions and final remarks}\label{sec:final}
More general instances of those summation theorems should exist beyond the arithmetic weights treated here. The next two formulas indicate the scope of possible extensions: \\

\noindent\textbf{1.} Let $\Lambda$ be a full lattice in $\R^N$, $L_d\subseteq\Lambda$ a finite-index sublattice, and $w_d$ a weight on $\Lambda/L_d$. Write $L_d^*$ for the Euclidean dual and $\mathcal F_{L_d}(w_d)(\xxi)=\sum_{\uu\in\Lambda/L_d}w_d(\uu)\en{\uu\cdot\xxi}$. The same residue-class specialization of Theorem \ref{thm:polynomial-T1}, after a lattice change of variables, gives
\begin{equation}\label{eq:general-lattice}
 \sum_{\vv\in\Lambda}\sum_d b(d)w_d(\vv)f_d(\vv)
 =\sum_d\frac{b(d)}{\operatorname{covol}(L_d)}
 \left(\sum_{\xxi\in L_d^*}\mathcal F_{L_d}(w_d)(\xxi)\widehat f_d(\xxi)\right),
\end{equation}
assuming Poisson admissibility on the residue classes and $\sum_d|b(d)|\sum_{\vv\in\Lambda}|w_d(\vv)f_d(\vv)|<\infty$, with completed inner Fourier sums. \\

\noindent\textbf{2.} We could also allow non-periodic weights: for $c(x)=\sum_j a_j\ee{\alpha_jx}$, $\alpha_j\in\R$, $\sum_j|a_j|<\infty$, the period-one formula and the uniform Schwartz bound give
\begin{equation}\label{eq:fourier-pair-direction}
 \sum_{n\in\Z}c(n)f(n)
 =\sum_j a_j\sum_{m\in\Z}\widehat f(m+\alpha_j)
 =\int_\R\widehat f(\xi)\,d\mu(\xi),\qquad f\in\mathcal S(\R),
\end{equation}
where $\mu=\sum_j a_j\sum_{m\in\Z}\delta_{m+\alpha_j}$. This is a concrete Fourier summation pair within the setting studied by Gon\c calves \cite{Goncalves23} and classified by Gon\c calves--Vedana \cite{Goncalves25}, after reflecting $\mu$ to match their Fourier-sign convention. Their result allows a strongly tempered measure and countably supported coefficients of finite exponential growth; extending the arithmetic constructions to that setting is a natural question for further study.

We obtained several further applications of our results that are not included here for reasons of space. These concern Epstein $L$-functions, Ingham-type results, Gauss circle and sphere problems, Lambert series for modular forms and related identities. We hope that the examples presented in this article convey the broader scope of the method and encourage readers to explore further applications. Moreover, we have not found the explicit error bounds of Theorem \ref{thm:kfree-general} in this form for $n\ge4$ in the literature we consulted; the cases
$n=r=2$, $k=1$, and $n=3$ of Theorem \ref{thm:kfree-general} recover the exponents of Tolev \cite{Tolev12} and Chen--Wang \cite{ChenWang23}.

\section{AI disclosure}\label{sec:ai-disclosure and Acknowledgments}
We issued a brief announcement of the results contained in this article \cite{OliveiraPanzone26} in light of recent concerns about research articles relying solely on AI-generated information. These developments have made us uncomfortable continuing our research under pressure to publish quickly, particularly given our concerns about the confidentiality of unpublished work and the possibility that our contributions could be shared with other researchers without our consent. We express this as a concern, not as an allegation that such sharing has occurred. 
Fable 5 and GPT 5.6a were used at several stages to improve the text, check examples, and assist the authors in generating numerical experiments. A particular example, already mention in Remarks \ref{rem:rem13} and \ref{sec:general-field-reciprocity}: extension field-reciprocity calculations began with our original attempt to extend the scalar argument using Bessel kernel and subsequent exchanges, with Claude and ChatGPT, helped to identify the Kelvin--Bessel kernels whose Fourier transforms match the rational remainders in important examples; besides the decisive step it was our two-step/two-channels application of our APSF, the contribution of detecting good candidates for the spectral/Fourier side of our formulas it was in great part a superb application of the power of those LLM. The authors remain responsible for the mathematical statements, proofs, references, and final presentation. All the errors, typographical mistakes, or unsupported claims are our responsibility and a reminder of our human fallibility. 

We thank Dimitar Dimitrov for helpful discussions on this and related topics during the preparation of this work. His insights substantially influenced the final presentation of several results. We also thank Cynthia Bortolotto, Felipe Gonçalves, and João Pedro Ramos: their collaborations with the second author over the past few years have shaped the questions motivating this research and contributed to the development of several APSF applications presented here. Finally, L. O. acknowledges the support of the Office of Naval Research through grant GRANT14201749 (award number N629092412126).

\end{document}